\documentclass[10pt]{article}

\usepackage{url}
\usepackage{mathtools}

\usepackage{amssymb}
\usepackage{amsthm}
\usepackage{empheq}
\usepackage{latexsym}
\usepackage{enumitem}
\usepackage{eurosym}
\usepackage{dsfont}
\usepackage{appendix}
\usepackage{color} 
\usepackage[unicode]{hyperref}
\usepackage{frcursive}
\usepackage[utf8]{inputenc}
\usepackage[T1]{fontenc}
\usepackage{geometry}
\usepackage{multirow}
\usepackage{todonotes}
\usepackage{lmodern}
\usepackage{anyfontsize}
\usepackage{pgfplots}
\usepackage{stmaryrd}
\usepackage{cleveref}
\usepackage{natbib}
\usepackage{upgreek}
\usepackage{comment}

\input Acorn.fd

\pgfplotsset{compat=1.16}

\setcitestyle{numbers,open={[},close={]}}

\definecolor{red}{rgb}{0.7,0.15,0.15}
\definecolor{green}{rgb}{0,0.5,0}
\definecolor{blue}{rgb}{0,0,0.7}
\hypersetup{colorlinks, linkcolor={red},citecolor={green}, urlcolor={blue}}
			
\makeatletter \@addtoreset{equation}{section}

\newtheorem{theorem}{Theorem}[section]
\newtheorem{assumption}{Assumption}

\newtheorem*{assumption*}{Assumption}

\newtheorem*{corollary*}{Corollary}

\newtheorem{lemma}[theorem]{Lemma}
\newtheorem*{lemma*}{Lemma}
\newtheorem{proposition}[theorem]{Proposition}

\newtheorem{definition}[theorem]{Definition}
\newtheorem{remark}[theorem]{Remark}
\newtheorem*{theorem*}{Theorem}

\DeclareUnicodeCharacter{014D}{\=o}
\def \D{\mathbb{D}}
\def \E{\mathbb{E}}
\def \F{\mathbb{F}}

\def \H{\mathbb{H}}
\def \I{\mathbb{I}}

\def \L{\mathbb{L}}
\def \M{\mathbb{M}}
\def \N{\mathbb{N}}
\def \P{\mathbb{P}}
\def \Q{\mathbb{Q}}
\def \R{\mathbb{R}}
\def \S{\mathbb{S}}

\def \Ho{\overline{\H}^{_{\raisebox{-1pt}{$ \scriptstyle 2,2,c$}}}}

\def\Bc{{\cal B}}
\def\Cc{{\cal C}}

\def\Ec{{\cal E}}
\def\Fc{{\cal F}}

\def\Hc{{\cal H}}

\def\Kc{{\cal K}}
\def\Lc{{\cal L}}
\def\Mc{{\cal M}}
\def\Nc{{\cal N}}

\def\Pc{{\cal P}}

\def\Sc{{\cal S}}
\def\Tc{{\cal T}}
\def\Uc{{\cal U}}
\def\Vc{{\cal V}}

\def\Xc{{\cal X}}
\def\Yc{{\cal Y}}
\def\Zc{{\cal Z}}

\def\Hf{{\mathfrak H}}

\def\Mf{{\mathfrak M}}

\def\Tf{{\mathfrak T}}

\def\Yf{{\mathfrak Y}}
\def\Zf{{\mathfrak Z}}
\def\Nf{{\mathfrak N}}

\def\eps{\varepsilon}
\def\d{{\mathrm{d}}}
\def\1{\mathbf{1}}

\def\Nt{\widetilde{N}}
\def\Zt{\widetilde{Z}}

\def\Tr{{\rm Tr}}

\def\eps{\varepsilon}
\def\d{{\mathrm{d}}}
\def\1{\mathbf{1}}
\def\e{{\mathrm{e}}}

\def\eps{\varepsilon}

\DeclareMathOperator*{\sgn}{sgn}

\DeclareMathOperator*{\es}{ess\,sup^\P} 
 
\DeclareMathOperator*{\Prob}{Prob} 

\renewcommand{\|}{\Vert}
\renewcommand{\t}{\top}

\def\as{\text{--a.s.}}
\def\ae{\text{--a.e.}}

\def\Tr{{\rm Tr }}

\allowdisplaybreaks

\begin{document}

\title{On quadratic multidimensional type-I BSVIEs, infinite families of BSDEs and their applications}

\author{Camilo {\sc Hern\'andez}\footnote{Columbia University, IEOR department, USA, camilo.hernandez@columbia.edu. Author supported by the CKGSB fellowship.}}

\date{\today}

\maketitle

\begin{abstract}
This paper investigates multidimensional \emph{extended} type-I BSVIEs and infinite families of BSDEs in the case of quadratic generators. 
	We establish existence and uniqueness results in the case of fully quadratic as well as Lipschitz-quadratic quadratic generators. We also present and discuss a type of flow property satisfied by this family of BSVIEs. 
	As a preliminary step, we establish the well-posedness of a class of infinite families of BSDEs, as introduced in \citet*{hernandez2020unified}, which are of interest in their own right. 
	Our approach relies on the strategy developed by \citet*{tevzadze2008solvability} for quadratic BSDEs and the treatment of Lipschitz extended type-I BSVIEs in \cite{hernandez2020unified}.
	We motivate their analysis of both of these objects by a series of practical applications. 
\vspace{5mm}

 \vspace{5mm}

\end{abstract}

\section{Introduction}

This paper studies type-I extended backward stochastic Volterra integral equations, BSVIEs for short, as recently revisited in \citet*{hernandez2020unified}. 
	Let $X$ be the solution to a drift-less stochastic differential equation, SDE for short, under a probability measure $\P$, $\F$ be the $\P$-augmentation of the filtration generated by $X$, see \Cref{sec:pstatement} for details. The data of the problem corresponds to a collection of $\Fc_T$-measurable random variables $(\xi(t))_{t\in [0,T]}$, referred in the literature of BSVIEs as the \emph{free term}, as well as a \emph{generator} $g$. 
	A solution to a type-I BSVIEs corresponds to a tuple $(Y_\cdot^\cdot, Z_\cdot^\cdot, N_\cdot^\cdot)$, of appropriately $\F$-adapted and integrable processes, satisfying
\begin{align}\label{eq:typeIBSVIEfe}
Y^s_t=\xi(s)+\int_t^T g_r(s,X,Y_r^s, Z_r^s,Y_r^r,Z_r^r)\d r -\int_t^T Z_r^s \d X_r- \int_t^T \d N_r^s,\;   t \in [0,T],\; \P\as,\; s \in [0,T].
\end{align}

The noticeable feature of \eqref{eq:typeIBSVIEfe} is the appearance of the `diagonal' processes $(Y_t^t)_{t\in[0,T]}$ and $(Z_t^t)_{t\in[0,T]}$ in the generator. 
	A prerequisite for rigorously introducing these processes is some regularity of the solution. 
	Indeed, the regularity of $s\longmapsto (Y^s,Z^s)$ in combination with the pathwise continuity of $Y$ and the introduction of a \emph{derivative} of $s\longmapsto Z^s$, as proposed in \cite{hernandez2020unified}, are sufficient for the analysis. 
	We also remark that, as we work with a general filtration $\F$, the additional process $N$ corresponds to a martingale process which is $\P$-orthogonal to $X$.
	In this work, we focus on the, to the best of knowledge, not studied case of multidimensional type-I BSVIEs with quadratic generators. 
	Notably,  we extend the analysis in \cite{hernandez2020unified} to the multidimensional quadratic case by exploiting the fact that the well-posedness of \eqref{eq:typeIBSVIEfe} is equivalent to that of a infinite family of backward stochastic differential equations, BSDEs for short.
	\medskip

The study of type-I BSVIEs began with the following set up: on a probability space supporting a Brownian motion $B$, one seeks for a pair $(Y_\cdot ,Z_\cdot^\cdot )$ of processes such that
	\begin{align}\label{eq:typeIBSVIE}
Y_t=\xi(t)+\int_t^T g_r(t,Y_r,Z_r^t)\d r -\int_t^T Z_r^t \d B_r,\;\P\as, \; t\in [0,T].
\end{align}
	The first mention of such equations is, to the best of our knowledge, due to \citet*{hu1991adapted} in the context of Hilbert-valued BSDEs, see the discussion following Remark 1.1 therein. 
	Two decades later, \citet*{lin2002adapted} considered \eqref{eq:typeIBSVIE} in the case $\xi(t)=\xi,$ $t\in [0,T]$. 
	The general study of type-I BSVIEs \eqref{eq:typeIBSVIE} is due to \citet*{yong2006backward,yong2008well}. For completeness, we remark that the concept of type-II BSVIEs, where the term $Z_t^r$ is also present in the generator, was also been studied in the literature. Type-II BSVIEs are beyond the scope of this paper and we refer the reader the interested reader to \cite{yong2006backward,yong2008well}. Note that BSDEs correspond to the case in which the data does not depend on the new parameter, i.e.
\begin{align*}
Y_t = \xi +\int_t^T g_r(Y_r,Z_r)\d r -\int_t^T Z_r \d B_r,\; t\in [0,T],\; \P\as,
\end{align*} 
	for which the seminal works of \citet*{pardoux1990adapted} and \citet*{el1997backward} introduced a systematic treatment and collected a wide range of their properties.
	Among such properties we recall the so-called \emph{flow property}, that is to say, for any $0\leq r\leq T$, 
\[ Y_t(T,\xi)=Y_t(r,Y_r(T,\xi)),\; t\in [0,r], \; \P\as, \text{ and } Z_t(T,\xi)=Z_t(r,Y_r(T,\xi)),\; \d t\otimes \d \P\ae\text{ on } [0,r]\times \Omega,\] 
	where $(Y(T,\xi), Z(T,\xi))$ denotes the solution to the BSDE with terminal condition $\xi$ and final time horizon $T$. 
	We highlight that, without additional assumptions, a solution to a general BSVIE does not satisfy the flow property. 
	

\medskip

Extended type-I BSVIEs \eqref{eq:typeIBSVIEfe} provide a rich framework to address new problems in mathematical finance and control. 
	For instance,  as initially suggested in \citet*{wang2019backward}, BSVIEs appear in time-inconsistent control problems via either Bellman's and Pontryagin's principles, see \citet*{yong2012time} and \citet*{wei2017time}, respectively.	
	A link was then made rigorous independently by \citet*[Section 5]{wang2019time} and \citet*[Lemma A.2.3]{hernandez2020me}. 
	Although following different approaches, both analyses lead to introduce type-I BSVIEs in which the diagonal of $Z$ appears in the generator.
	Likewise, the case of cost functionals given by the $Y$ component of a type-I BSVIE \eqref{eq:typeIBSVIE}, in which $g$ depends on a control was studied in \citet*{hamaguchi2020extended}.
	The adjoint equation induced by Pontryagin's optimal principle solves an type-I BSVIE in which the diagonal of $Y$ appears in the generator, see also \citet*{wang2020extended}. 
\medskip

We remark that, to the best of our knowledge, there are no well-posedness results for multidimensional quadratic type-I BSVIEs \eqref{eq:typeIBSVIE}, let alone for extended ones \eqref{eq:typeIBSVIEfe}. 
	In fact, to the best of our knowledge, the study of non-Lipschitz BSVIEs remains limited to \citet*{ren2010solutions}, \citet*{shi2012solvability}, \citet*{wang2018recursive},  and \citet*{wang2007nonlipschitz}. 
	In \cite{wang2007nonlipschitz} and \cite{ren2010solutions}, the authors consider solutions to general multidimensional type-I BSVIEs where the generator is increasing and concave in $y$ and Lipschitz in $z$. 
	\cite{shi2012solvability} continued the study and settled some flaws in the analysis of \cite{ren2010solutions}. 
	On the other hand, \cite{wang2018recursive} presents the first analysis of scalar BSVIEs whose generator have quadratic growth on $z$. 
	Indeed, the authors consider a standard one dimensional type-I BSVIE \eqref{eq:typeIBSVIE} in which the generator is Lipschitz in $y$ and quadratic in $z$, which we will refer to as the Lipschitz quadratic case, provided the data of the BSVIE is bounded. \medskip
	
	We emphasise that the additional assumptions in \cite{wang2018recursive} are due to underlying employed results for scalar quadratic BSDEs.\footnotemark 
	\ Indeed, extending the ideas in \cite{briand2006bsde, briand2008quadratic} was the strategy behind \cite{wang2018recursive} and it explains the framework of their result, i.e. scalar BSVIEs with Lipschitz quadratic generator and bounded data. 
	Moreover, as \cite{wang2018recursive} states: ``The case [of] $Y$ being higher dimensional will be significantly different in general.'' 
	Therefore, in the multidimensional case, new approaches become necessary as tools that are usually used in the analysis of scalar BSDEs, like monotone convergence or Girsanov transform, are no longer available. 
	In fact, \citet*{frei2011financial}, provide a simple example of a multidimensional quadratic BSDE with a bounded terminal condition for which there is no solution. 
	This counterexample shows that a direct generalization of the approaches in \cite{kobylanski2000backward,briand2006bsde, briand2008quadratic} would be unsuccessful in the case of extended multidimensional quadratic type-I BSVIEs.\medskip
	\footnotetext{We recall that the analysis of scalar quadratic BSDEs is much more delicate.
	The first result, due to \citet*{kobylanski2000backward}, recently revisited by \citet*{jackson2020characterization}, holds for bounded and Lipschitz quadratic data, and was extended to the super quadratic case in \citet*{lepeltier1998existence} and \citet*{delbaen2011backward}.
	\citet*{briand2006bsde, briand2008quadratic} showed that imposing sufficiently large exponential moments $\xi$ is enough. }
	
	Beyond imposing structural conditions on the generator, the literature on multidimensional quadratic BSDE provides general well-posedness results exploiting the theory of BMO martingales or focusing on Markovian BSDEs.\footnote{See \citet*{cheridito2015multidimensional} for specific choices of generators. For triangular and diagonally quadratic generator see \citet*{jackson2021existence}, \citet*{hu2015multi}, \citet*{hu2021quadratic}, \citet*{jamneshan2017multidimensional} and \citet*{kupper2019multidimensional}, and \citet*{luo2020type,luo2021comparison}
	. For completeness, see also \citet*{frei2014splitting} and \citet*{kramkov2016system}.}
	The result in \citet*{harter2019stability} approximates the solution of a Lipschitz quadratic BSDE, assuming the \emph{a priori} existence of uniform estimates on the BMO norm of the local martingale $\int_0^\cdot Z_r^n \d W_r$ and exploiting the theory of Malliavin calculus to pass to the limit. 
	Concerning Markovian BSDEs, \citet*{xing2016class} focuses on a class of Markovian systems whose generator satisfy an abstract structural condition.
	However, neither of these approaches extends properly when considering extended type-I BSDEs. The crux of the problem lies in the fact that \eqref{eq:typeIBSVIEfe} allows for generators in which the diagonal of both $Y$ and $Z$ appear in the generator. Moreover, the approach presented in \cite{hernandez2020unified} for Lipschitz extended type-I BSVIEs leverages suitable estimates for both of these processes to stablish a fixed point argument.
\medskip	
	
	On the other hand, the original method introduced in \citet*{tevzadze2008solvability} takes a different view of this problem and presents a fixed-point argument that is able to cover quadratic BSDEs, in both $y$ and $z$, but requires, once again, the data to be bounded and sufficiently small. 
	We stress that, unlike the approaches of \cite{harter2019stability} or \cite{xing2016class}, the approach in \cite{tevzadze2008solvability} works for quadratic BSDEs.  
	In light of our discussion at the end of the previous paragraph, this methodology can be reconcile with the analysis in \cite{hernandez2020unified} to obtain a well-posedness result for multidimensional quadratic extended type-I BSVIEs. 
	This constitute the methodological motivation of our approach. \medskip

To be able to cover multidimensional quadratic type-I BSVIEs as general as \eqref{eq:typeIBSVIEfe}, following the ideas in \cite{hernandez2020unified}, our approach is based on the fact that,  for an appropriate choice of data,  its well-posedness is equivalent to that of the system of infinite families of BSDEs of the form
\begin{align}\label{eq:systemq:intro}
\begin{split}
\Yc_t&=\xi(T)+\int_t^T h_r(X,\Yc_r,\Zc_r, Y_r^r,Z_r^r, \partial Y_r^r) \d r-\int_t^T  \Zc_r \d X_r- \int_t^T \d \Nc_r ,  \\
Y_t^s&=   \eta (s)+\int_t^T  g_r(s,X,Y_r^s,Z_r^s, \Yc_r,\Zc_r) \d r-\int_t^T Z_r^s  \d X_r - \int_t^T \d N_r^s,  \\
\partial Y_t^s&=  \partial_s \eta (s)+\int_t^T  \nabla g_r(s,X,\partial Y_r^s,\partial Z_r^s,Y_r^s,Z_r^s, \Yc_r, \Zc_r) \d r-\int_t^T\partial  {Z_r^s}  \d X_r-\int_t^T \d \partial M^s_r.
\end{split}
\end{align}
for unknown $(\Yc,\Zc,\Nc,Y,Z,N,\partial Y,\partial Z,\partial N)$ required to have appropriate integrability, see \Cref{sec:infdimsystem} for details. \medskip

The rest of the paper is organised as follows: To motivate the results of this document, i.e. the study of the well-posedness of infinite families of BSDEs \eqref{eq:systemq:intro} and type-I extended BSVIEs \eqref{eq:typeIBSVIEfe}, this introductory section closes with three practical applications of our results in the following. \Cref{sec:prelim} introduces the problem's set up and the appropriate integrability spaces for our analysis.  \Cref{sec:infdimsystem} presents our well-posedness result for infinite families of multidimensional BSDEs in both the linear quadratic and quadratic whose proofs are deferred to \Cref{sec:prooflinearquadratic} and \Cref{sec:proofquadratic}, respectively. \Cref{sec:BSVIE} establishes the equivalence of the well-posedness of type-I extended BSVIEs with that of a system of the form as those studied in \Cref{sec:infdimsystem}. We close our study of BSVIEs in \Cref{sec:flowppty} where we discuss on the nature of the flow property for type-I extended BSVIEs. Lastly, some auxiliary results are presented in the Appendix section.

\subsection{Practical motivations}\label{sec:motivation}

\begin{enumerate}[label=$(\roman*)$, ref=.$(\roman*)$,wide, labelwidth=!, labelindent=0pt]

\item An immediate practical motivation comes from the study of time-inconsistent control problems. 
	In this kind of problems the idea of optimal controls is incompatible with the underlying agent's preferences and a successful concept of solution is that of consistent plans or equilibria, as initially introduced in \citet*{ekeland2006being,ekeland2010golden}. 
	This approach is known as that of the \emph{sophisticated} agent. 
	The analysis in {\rm \cite{hernandez2020me}}, limited to the Lipschitz case, established the connection with type-I BSVIEs via an extended dynamic principle. 
	For completeness, we precise the dynamics of the controlled process $X$ in a Markovian framework. 
	Let $\sigma_t$ be bounded, and $A:=[a_1,a_2]\subseteq \R$ so that
\[
X_t=x_0+\int_0^t\sigma_t \frac{\alpha_r-X_r}{r-T} \d r+\int_0^t \sigma_r \d W_r^{\alpha}, \; t\in [0,T], \; \P^\alpha\as
\]
for some $A$-valued process $\alpha$ denoting the agent's action. 
	$\P^\alpha$ denotes a probability measure governing the distribution of the canonical process $X$ which the agent controls and $W^\alpha$ denotes a $\P^\alpha$--Brownian motion, see \Cref{sec:pstatement}. 
	We recall that $\P^\alpha$ is guarantee to exists as the previous Lipschitz SDE has a unique strong solution. 
	Note that $X_T$ is bounded as at $t=T$, we have that $X_T=\alpha_T\in [a_1,a_2]$. 
	Moreover, for real valued $k$, $F$,  and $G$ with appropriate, the reward of an agent performing $\alpha$ from time $t$ onwards and current state value $x\in \Xc$, is given by 
\[
J(t,x,\alpha):=\E^{\P^\alpha} \bigg[ \int_0^T k_r(t,X_{ r},\alpha_r) \d r + F(t,X_{T})\bigg|\Fc_t \bigg] +G\big(t,\E^{\P^\alpha}\big[X_{T} \big| \Fc_t\big]\big),
\]
where $\E^{\P^\alpha}[\cdot |\Fc_t]$ denote the classic conditional expectation operator under $\P^\alpha$. 
	The noticeable features of this type of rewards are: 
	$(a)$ the dependence of $k$, $F$ and $G$ on $t$ which besides the case of exponential discounting is a source of time-inconsistency; 
	$(b)$ the possible non-linear dependence of $G$ on a conditional expectation of ${\rm g}(X_{ T})$, another source of time-inconsistency, which would allow for mean-variance type of criteria. \medskip

Following the analysis in {\rm \cite{hernandez2020me}}, let $\xi(s,x):= F(s,x)+G(s,{\rm g}(x))$, $g_t(s,x,z,a):=k_t(s,a)+ \frac{x-a}{t-T}\sigma_t \cdot z$, 
\[
H_t(x,z,u,n,{\rm z}) :=\sup_{a \in A} \big\{ g_t(t,x,z,a)\big \}-u-\partial_s G(t,n)-\frac{1}{2}\ {\rm z}^\t\sigma_t\sigma_t^\t {\rm z} \  \partial_{nn}^2 G(t,n),
\]
and denote by $a^\star(t,x,z)$ the $A$-valued measurable mapping attaining the sup in $H$, assumed to exists. 
	Then, we find that agent's value function associated to an equilibrium action $\alpha^\star$ correspond to $\Yc_t$ and $\alpha_t^\star:=a^\star(t,X_{t},\Zc_r)$, respectively, where
\begin{align*}
\Yc_t&=\xi(T,X_{ T})+\int_t^T H_r(X_{ r},\Zc_r, \partial Y_r^r,\Nt_r,\Zt_r)\d r-\int_t^T  \Zc_r \cdot \d X_r,\\
Y_t^s&= \xi(s,X_{ T})+\int_t^T g_r(s,X_{  r}, Z_r^s, a^\star(r,X_{  r},\Zc_r)) \d r-\int_t^T Z_r^s \cdot \d X_r, \\
\Nt_t&= g(X_{ T}) +\int_t^T b_r(X_{ r},a^\star(r,X_{  r},\Zc_r))\cdot \sigma_r^\t \Zt_r\d r  -\int_t^T  \Zt_r \cdot \d X_r.
\end{align*}
	We remark that $\partial Y$ is defined as in \Cref{sec:infdimsystem}. 
	Moreover, we highlight that: 
	$(a)$ $H$ is quadratic in ${\rm z}$ and, if, for instance, $G(s,n)=\phi(s)n^2$, it is quadratic in $n$ too; 
	$(b)$ the apearance of $\Nt$ and $\Zt$ in the first equation leads to a multidimentional system even if $Y$ is real values; 
	$(c)$ having access to a well-posedness result for the previous system guarantees the existence and uniqueness of an equilibrium strategy, see \cite{hernandez2020me}. 
	This is particularly important in light of the example, stemming from a mean-variance investment problem,  in \cite{landriault2018equilibrium} in which uniqueness of the equilibrium fails. We direct to \Cref{sec:infdimsystem} and \Cref{Thm:wp:smalldata} for details on the following result. We also mention that by construction \Cref{Assumption:LQgrowth} is satisfied.
	
\begin{proposition}
	Let {\rm \Cref{Assumption:Differentiability}} hold. Suppose $x\longmapsto (\xi(t,x), \partial_s\xi(s,x))$ is monotone and continuous,
	and {\rm \Cref{Assumption:wp:eta}} holds for $\kappa=10$.  Then, the previous system has a unique solution and there is a unique equilibrium associated to the time-inconsistent control problem faced by the sophisticated time-inconsistent agent.
\end{proposition}	

\item Our next motivation comes from the risk-sensitive non-Markovian nonzero-sum game introduced in \citet*{el2003bsdes} and revisited in \cite{hu2015multi}.
	This is a situation in which many individuals are allowed to intervene on a system $X$, but contrary to the zero-sum case, their preferences are not necessarily antagonistic. 		
	In fact, we allow each one to look after her own interest. 
	It was established in \cite[Proposition5.1]{el2003bsdes} that the resolution of this game problem is obtained via multidimensional quadratic BSDE. 
	Said equation is of multidimensional type since there are several players and each one has its own associated reward functional. 
	As first noted in \cite{el2003bsdes}, and settled in \cite{harter2019stability}, the existence of a solution for multidimensional BSDEs renders whether the game has an equilibrium point. 
	However, an extension of this model to the case non-exponential discounting fits into the scope of our paper. Indeed, for $X$ and $\P^\alpha$ as in \Cref{sec:pstatement}, measurable $g^i,\xi^i,k^i$, $i=1,\dots, n$, $b$, and $\alpha:=(\alpha^1,\dots, \alpha^n)$ (resp. $\alpha^{-i}$) the vector of actions (resp. without the $i$--th entry) we set
\[
{J_t^i	}(\alpha^i,\alpha^{-i}):=\E^{\P^{\alpha}}\bigg [\exp \bigg(g^i(T-t)\xi^i -\int_t^T g^i(r-s) k_r^i(X_{\cdot\wedge  r},\alpha_r,\alpha^{-i}_r)\d r\bigg) \Big |\mathcal F_t\bigg].
\]
In this setting, we obtain that the existence to an equilibrium solution to the nonzero-sum game between time-inconsistent agents is associated to the well-posedness of the $n$-dimensional BSVIE given by
\[
Y_t^s= \xi +\int_t^T g_r(s,X_{\cdot\wedge  r}, Z_r^s, a^\star(r,X_{\cdot\wedge  r},Z_r^r)) \d r-\int_t^T Z_r^s \cdot \d X_r.
\]
As in the previous example, our result guarantee the well-posedness of the previous BSVIEs. Let us further remark that this class of models can be taylored to cover applications to: 
	financial market equilibrium problems for interacting agents, as in \citet*{bielagk2017equilibrium}, \citet*{espinosa2015optimal}, \cite{frei2011financial}, and \cite{frei2014splitting}; 
	price impact models, as in \cite{kramkov2016system} and \cite{kramkov2014stability}; 
	and Principal-Agent contracting problems with competitive agents as in \citet*{elie2019contracting} to mention just a few. 

\item One additional motivation for our results builds upon the treatment presented in \cite{wang2018recursive} of scalar continuous-time dynamic risk measures in \cite{el1997backward}. 
	This is, it is possible to extending these ideas to the more realistic situation where the risky portfolio is vector-valued? 
	In the static case, \citet*{jouini2004vector} provides a notion of multidimensional coherent risk measure that renders a convenient extension of the real-valued risk measures intially introduced in \citet*{artzner1999coherent}. 
	Building upon these ideas, \citet*{kulikov2008multidimensional} presents a model that, for instance, can accommodate the risks of changing currency exchange rates and transaction costs.  
	More generally, exploiting the partial order induced by a given convex cone $\Kc\subseteq \R^n$, namely, $y\preceq_\Kc y^\prime :\Longleftrightarrow y^\prime-y\in \Kc$ for $y,y^\prime\in \R^n$, the author introduces multidimensional coherent risk measures that extend classic one-dimensional risk measures such as the tail V@R, and weighted V@R. \medskip
	
	Let us remark that at the heart of the treatment of continuous-time dynamic risk measures in \cite{wang2018recursive} lies the access to a comparison theorem, which the authors recover for one-dimensional type-I BSVIEs \eqref{eq:typeIBSVIE}. 
	We emphasise that, as for BSDEs, this is a much harder task in the multidimensional case.
	Nevertheless, the positive result presented in \citet*{hu2006comparison} for multidimensional BSDEs makes perfectly clear that this follows from the study of the so-called viability property for (multidimensional) BSDEs as presented, for instance, in \citet*{buckdahn2000viability}. 
	We recall that the approach in \cite{buckdahn2000viability} is based on the convexity of the distance function induced by $\Kc$, where, in general, $\Kc$ can be taken to be any closed convex set.
	This is certainly is the case whenever $\Kc$ is a convex cone.
	All things considered, the crux of the problem lies in establishing the appropriate extension of the viability property for multidimensional type-I BSVIEs. 
	One this is available, as shown in \cite{hu2006comparison}, one should access a comparison theorem thus answering the aforementioned question.
\end{enumerate}





\section{Preliminaries}\label{sec:prelim}

\medskip
{\bf Notations:}
we fix a time horizon $T>0$. Given $(E,|\cdot |)$ a finite-dimensional Euclidean space,  a positive integer $d$, and a non--negative integer $q$, $\Cc^{d}_{q}(E)$ (resp. $\Cc_{q,b}^d(E)$) will denote the space of functions from $E$ to $\R^{d}$ which are $q$ times continuously differentiable (resp. and bounded with bounded derivatives). When $d=1$ we write $\Cc_{q}(E)$ and $\Cc_{q,b}(E)$. For $\phi \in \Cc_{0,q}([0,T]\times E)$ with $q\geq 2$,
if $s \longmapsto \phi(s,\alpha)$ is uniformly continuous uniformly in $\alpha$, we denote by $\rho_{\phi }:[0,T]\longrightarrow \R$ its modulus of continuity. $\partial_\alpha \phi$ and $\partial_{\alpha \alpha}^2 \phi$ denote the gradient and Hessian with respect to $\alpha$, respectively. For $(u,v) \in (\R^p)^2$, $u\cdot v$ will denote their usual inner product, and $|u|$ the corresponding norm. $\S_n^+(\R)$ denotes the set of $n\times n$ symmetric positive semi-definite matrices, while $\Tr [M]$ denotes the trace of $M\in  R^{m
\times m}$, and $| M|:=\sqrt{\Tr[M^\t M]}$ for $M\in \R^{m
\times n}$.

\medskip
For $(\Omega, \Fc)$ a measurable space, $\Prob(\Omega)$ denotes the collection of probability measures on $(\Omega, \Fc)$. For a filtration $\F:=(\Fc_t)_{t\in [0,T]}$ on $(\Omega, \Fc)$, $\Pc_{\rm pred}(E,\F)$ (resp. $\Pc_{\rm prog}(E,\F)$, $\Pc_{\rm opt}(E,\F)$, $\Pc_{\rm meas}(E,\F)$) denotes the set of $E$-valued, $\F$-predictable processes (resp. $\F$-progressively measurable processes, $\F$-optional processes, $\F$-adapted and measurable).. When $\F$ is fixed we remove the dependece, e.g. we write $\Pc_{\rm opt}(E)$ for $\Pc_{\rm opt }(E,\F)$.  For $\P\in \Prob(\Omega)$, $\F^\P:=(\Fc_t^\P)_{t\in[0,T]},$ denotes the $\P$-augmentation of $\F$. With this, $(\Omega,\Fc, \F^\P,\P)$ can be extended to a complete probability space, see \citet*[Chapter II.7]{karatzas1991brownian}. $\F^\P_+$ denotes the right limit of $\F^\P$, so that $\F^{\P}_+$ is the minimal filtration that contains $\F$ and satisfies the usual conditions. For $\{s,t\}\subseteq [0,T]$, with $s\leq t$, $\Tc_{s,t}(\F)$ denotes the collection of $[t,T]$-valued $\F$-stopping times.\medskip

\subsection{The stochastic basis on the canonical space}\label{sec:pstatement}

We fix two positive integers $n$ and $m$, which represent respectively the dimension of the martingale which will drive our equations, and the dimension of the Brownian motion appearing in the dynamics of the former. We consider the canonical space $\Xc:=\Cc([0,T],\R^{n})$, with canonical process $X$. We let $\Fc$ be the Borel $\sigma$-algebra on $\Xc$ (for the topology of uniform convergence), and we denote by $\F^o:=(\Fc^o_t)_{t\in[0,T]}$ the natural filtration of $X$. We fix a bounded Borel measurable map $\sigma:[0,T]\times\Xc \longrightarrow \R^{n\times m}$, $\sigma_\cdot(X)\in \Pc_{{\rm meas}}(\R^{n\times m},\F^o)$, and an initial condition $x_0\in\R^n$. We assume there is $\P\in \Prob (\Xc)$ such that $\P[X_0=x_0]=1$ and $X$ is martingale, whose quadratic variation, $\langle X\rangle=(\langle X\rangle_t)_{t\in [0,T]}$, is absolutely continuous with respect to Lebesgue measure, with density given by $\sigma\sigma^\t$. Enlarging the original probability space, see \citet*[Theorem 4.5.2]{stroock2007multidimensional}, there is an $\R^{m}$-valued Brownian motion $B$ with
\[
X_t=x_0+\int_0^t\sigma_r(X_{\cdot \wedge r}) \mathrm{d}B_r,\; t\in[0,T],\; \P\as
\]
We now let $\F:=(\Fc_t)_{t\in[0,T]}$ be the (right-limit) of the $\P$-augmentation of $\F^o$. We stress that we will not assume $\P$ is unique. In particular, the predictable martingale representation property for $(\F,\P)$-martingales in terms of stochastic integrals with respect to $X$ might not hold. 
\begin{remark}
We remark that the previous formulation on the canonical is by no means necessary. Indeed, any probability space supporting a Brownian motion $B$ and a process $X$ satisfying the previous {\rm SDE} will do, and this can be found whenever that equation has a weak solution.
\end{remark}

\subsection{Functional spaces and norms}\label{Sec:spaces}
We now introduce the spaces of interest for our analysis. In the following, $(\Omega,\Fc_T, \F,\P)$ denotes the filtered probability space as defined in the introduction. We are given a non-negative real number $c$ and $(E,|\cdot |)$ a finite-dimensional Euclidean space, i.e. $E=\R^{\tilde d}$ for some non-negative integer $\tilde d$ and $|\cdot |$ denotes the $L^2$-norm. We also introduce the $\Lc^{\infty}$-norm which for an arbitrary $E$-valued random variable $\zeta$ is given by $\|\zeta \|_{\infty}:=\inf \{ C\geq 0 :  | \zeta|\leq C,\;  \P\as\}$ as well as the spaces 

\begin{list}{\labelitemi}{\leftmargin=1em}

\item $\Lc^{\infty,c}(E)$ of $\xi \in \Pc_{\rm meas}(E,\Fc_T)$ $\P$-essentially bounded, such that $ \|\xi\|_{\Lc^{\infty,c}}:= \|\e^{\frac{c}2T}  \xi \|_{\infty} <\infty;$

\item $\Sc^{\infty,c}(E)$ of $Y\in \Pc_{\text{opt}}(E)$, with $\P\as$ c\`adl\`ag paths on $[0,T]$ and $\|Y\|_{\Sc^{\infty,c} }:=\bigg\| \displaystyle \sup_{t\in [0,T]}\e^{\frac{c}2 t} |Y_t|\bigg\|_{ \infty } <\infty;$

\item $\S^{2,c}(E)$ of $Y\in \Pc_{\text{opt}}(E)$, with $\P\as$ c\`adl\`ag paths on $[0,T]$ and $ \|Y\|_{\S^{2,c} }^2:=\E \bigg[ \displaystyle \sup_{t\in [0,T]} \e^{\frac{c}2 t} | Y_t|^2\bigg] <\infty;$

\item  $\L^{1,\infty,c} (E)$ of $Y\in \Pc_{\text{opt}}(E)$ with $\|Y\|_{\L^{1,\infty,c} }:=  \bigg\| \displaystyle \int_0^T \e^{\frac{c}2 t}|Y_t|\d t \bigg\|_{\infty} <\infty;$

\item $\H^{2,c}(E)$ of $Z\in \Pc_{\text{pred}}(E)$, which are defined $\sigma\sigma_t^\t\d t\ae$, with $ \|Z\|_{\H^{2,c} }^2:=\E \bigg[  \displaystyle \int_0^T \e^{cr} |\sigma_r Z_r|^2\d r \bigg] <\infty;$

\item ${\rm BMO}^{2,c}(E)$ of square integrable $E$-valued $(\F,\P)$-martingales $M$ with $\P\as$ c\`adl\`ag paths on $[0,T]$ and
\[ 
\|M\|_{{\rm BMO}^{2,c} }^2:= \sup_{\tau \in \Tc_{0,T}} \bigg \|  \E^\P\bigg[\bigg |  \int_{\tau-}^T \e^{\frac{c}2 r-}\d M_r\bigg |^2 \bigg| \Fc_\tau\bigg] \bigg \|_\infty <\infty;
\]

\item $\H^{2,c}_{{\rm BMO}}(E)$ of $Z\in \Pc_{\text{pred}}(E)$, which are defined $\sigma\sigma_t^\t\d t\ae$, with $\|Z\|_{\H^{2,c}_{{\rm BMO}} }^2:=\bigg\|  \displaystyle \int_0^\cdot Z_r \d  X_r \bigg\|_{{\rm BMO}^{2,c} }^2<\infty;$

\item $\M^{2,c}_{\rm BMO}(E)$ of c\`adl\`ag martingales $N \in \Pc_{\text{\rm opt}}(E)$, $\P$-orthogonal to $X$ $($that is the product $XN$ is an $(\F,\P)$-martingale$)$, with $N_0=0$ and $\|N\|^2_{\M^{2,c}_{\rm BMO} }:=\big\| N  \big \|_{{\rm BMO}^{2,c}}^2 <\infty;$

\item $\M^{2,c}(E)$ of c\`adl\`ag martingales $N \in \Pc_{\text{\rm opt}}(E)$, $\P$-orthogonal to $X$, $N_0\!=0$ and
$\|N\|^2_{\M^{2,c}}\! :=\E \bigg[   \displaystyle \int_0^T \!\! \e^{c r-} \d \Tr [ N]_r    \bigg]\! <\! \infty$\footnote{We remark the use of $[\, \cdot\, ]$ instead of $\langle\, \cdot\, \rangle$ as $N$ is only assumed to be c\`adl\`ag};

\item $\Pc^2_{\text{meas}}(E,\Fc_T)$ of two parameter processes $(U_t^s)_{(s,t) \in [0,T]^2 }$ $:([0,T]^2\times \Omega, \Bc([0,T]^2)\otimes \Fc_T)  \longrightarrow (E,\Bc(E))$ measurable.

\item $\Lc^{\infty,2,c}(E)$ denotes the space of collections $\eta:=(\eta(s))_{s\in [0,T]}\in \Pc^2_{\text{meas}}(E,\Fc)$ such that the mapping $([0,T]\times \Omega, \Bc([0,T])\otimes \Fc)\longrightarrow (\Lc^{\infty,c}(\Fc,\P)),\| \cdot \|_{\infty,\P })):s\longmapsto \eta(s)$ is continuous and $\| \eta\|_{\Lc^{\infty,2,c} }:= \displaystyle \sup_{s\in[0,T]} \|\eta(s)\|_{\Lc^{\infty,c} } <\infty.$

\item Given a Banach space $(\I^c(E),\| \cdot \|_{\I^c})$, we define $(\I^{2,c}(E),\|\cdot \|_{\I^{2,c}})$ the space of $U\in \Pc^2_{\text{meas}}(E,\Fc_T)$ such that $([0,T],\Bc([0,T]))$ $ \longrightarrow (\I^{c}(E),\|\cdot \|_{ \I^{c}}): s \longmapsto U^s $ is continuous and $ \|U\|_{\I^{2,c}}:=\displaystyle \sup_{s\in [0,T]}  \|U^s\|_{\I} <\infty.$

For example, $\L^{1,\infty,2,c} (E)$  denotes the space of  $U\in \Pc^2_{\text{meas}}(E,\Fc_T)$ such that $([0,T],\Bc([0,T])) \longrightarrow (\L^{1,\infty,c} (E),\|\cdot \|_{ \L^{1,\infty,c} })$ $: s \longmapsto U^s $ is continuous and $\| U\|_{\L^{1,\infty,2,c} }:=\displaystyle \sup_{s\in[0,T]} \| U^s\|_{\L^{1,\infty,c}} <\infty;$ 

\item $\Ho(E)$ of $(Z_t^s)_{(s,t) \in [0,T]^2 }\in \Pc^2_{\text{meas}}(E,\Fc_T)$ such that $([0,T],\Bc([0,T])) \longrightarrow (\H^{p}(E),\|\cdot \|_{ \H^{p}}): s \longmapsto Z^s $ is absolutely continuous with respect to the Lebesgue measure, $\Zc\in \H^{2,c}(E)$, where $\Zc:=(Z_t^t)_{t\in [0,T]}$ is given by
\[ Z_t^t:=Z_t^T-\int_t^T\partial Z_t^r \d r,\text{ and, } \|Z\|_{\overline \H^{2,c,2}}^2:=\|Z\|_{\H^{2,2}}^2+\|\Zc\|_{\H^2}^2<\infty \]

\item $\Ho_{\rm BMO}(E)$ of $(Z_t^s)_{(s,t) \in [0,T]^2 }\in \Pc^2_{\text{meas}}(E,\Fc_T)$ such that $([0,T],\Bc([0,T])) \longrightarrow (\H^{p}_{\rm BMO}(E),\|\cdot \|_{ \H^{p}_{\rm BMO}}): s \longmapsto Z^s $ is absolutely continuous with respect to the Lebesgue measure, $\Zc\in \H^{2,c}_{\rm BMO}(E)$, where $\Zc:=(Z_t^t)_{t\in [0,T]}$ is given by
\[ Z_t^t:=Z_t^T-\int_t^T\partial Z_t^r \d r,\text{ and, } \|Z\|_{\overline \H_{\rm BMO}^{2,c,2}}^2:=\|Z\|_{\H^{2,c,2}_{\rm BMO}}^2+\|\Zc\|_{\H^{2,c}_{\rm BMO}}^2<\infty  \]
\end{list}
\begin{remark}\label{remark:defspaces}
\begin{enumerate}[label=$(\roman*)$, ref=.$(\roman*)$,wide, labelwidth=!, labelindent=0pt]
\item We remark that the first set of spaces in the previous list, but $\M_{\rm BMO}^{2,c}(E)$, are the corresponding weighted version of the classic spaces in the literature for {\rm BSDEs}, which are recovered by taking $c=0$. Such weighted spaces are known to be more suitable to handle existence results. Moreover, given our assumption of finite time horizon these spaces are known to be isomorphic for any value of $c$.

\item The second set of these spaces are weighted versions of suitable extensions of the classical ones, whose norms are tailor--made to the analysis of the systems we will study. Some of these spaces have been previously considered in the literature on {\rm BSVIEs}, see {\rm \cite{hernandez2020unified} }, {\rm \cite{yong2006backward}} and {\rm\cite{wang2019time}}.Of particular interest are the spaces $\Ho(E)$ and $\Ho_{\rm BMO}(E)$. Indeed, the space $\H^{2,c}(E)$ being closed implies $\overline{\H}^{_{\raisebox{-1pt}{$ \scriptstyle 2,2,c$}}}(E)$ is a closed subspace of $\H^{2,2,c}(E)$ and thus a Banach space. Let us recall that the space $\overline{\H}^{_{\raisebox{-1pt}{$ \scriptstyle 2,2,c$}}}(E)$ allows us to define a good candidate for $(Z_t^t)_{t\in [0,T]}$ as an element of $\H^{2,c}(E)$. Let $\widetilde \Omega:=[0,T]\times \Xc$, $\tilde \omega:=(t,x)\in \widetilde \Omega$ and
\[
\mathfrak{Z}_s(\tilde \omega):= Z^T_t(x)-\int_s^T \partial Z^r_t(x) \d r, \; \d t \otimes \d \P\ae\ \tilde \omega\in \widetilde \Omega,\; s\in [0,T],\vspace{-0.5em}
\]
so that the Radon--Nikod\'ym property and Fubini's theorem imply $\mathfrak{Z}_s=Z^s, \d t\otimes \d \P\ae$, $s\in [0,T]$. Lastly, as for $\tilde \omega\in \widetilde \Omega$, $s \longmapsto \mathfrak{Z}_s(\tilde \omega)$ is continuous, we may define
\[
Z^t_t := Z^T_t-\int_t^T \partial Z^r_t \d r, \text{ for } \; \d t \otimes \d \P\ae \ (t,x) \text{ \rm in } [0,T]\times \Xc.\vspace{-0.5em}
\]
\item Lastly, we comment on our choice to introduce the spaces $\M_{\rm BMO}^{2,c} (E)$ and $\M_{\rm BMO}^{2,2,c}(E)$. Those familiar with the theory of {\rm BSDEs} would recognise the integrability in $\M^{2,c}(E)$ as the typical one for orthogonal martingales. However, given the setting of this paper, one might argue whether it would be more natural to require a ${\rm BMO}$--type of integrability, as the space $\M_{\rm BMO}^{2,c}(E)$ does. This had been noticed since {\rm \cite{tevzadze2008solvability}}. Therefore, a natural question is how requiring one specific type of integrability would quantitatively affect our well-posedness results. 
\end{enumerate}
\end{remark}
\subsection{Auxiliary inequalities}
We list some useful inequalities. Young's inequality states that for $\eps >0$, $ 2 a b \leq \eps a^2 + \eps^{-1} b^2$. For any positive integer $n$ and any collection $(a_i)_{1\leq i\leq n}$ of non--negative numbers it holds that
\begin{align}\label{Eq:ineqsquare}
\bigg(\sum_{i=1}^n a_i\bigg)^2\leq n\sum_{i=1}^n a_i^2.
\end{align}

A particularly useful inequality in our setting is obtained from the so--called energy inequality, see \citeauthor*{meyer1966probability} \cite[Chapter VII. Section 6]{meyer1966probability}. For a positive integer $p$ and a potential $X$, i.e. a positive right--continuous super--martingale s.t. $\E[X_t]\longrightarrow 0$, $t\longrightarrow\infty$, the $p$th-energy is defined by $\e_p(X_t):=\frac{1}{p!}\E [(A_\infty)^p]$, where $A$ is the increasing, right--continuous process appearing in the Doob--Meyer decomposition of $X$. The $p$--energy inequality states that 
\[ \e_p(X_t)\leq C^p,\; \text{whenever}\; |X_t|\leq C. \]
In our framework, it leads to obtain the following auxiliary inequalities, whose proof we present in \Cref{sec:proofsprelim}.
\begin{lemma}\label{lemma:energy}
Let $\tilde d$ be a positive integer.
\begin{enumerate}[label=$(\roman*)$, ref=.$(\roman*)$,wide, labelwidth=!, labelindent=0pt]
\item Let $Z\in \H^2_{{\rm BMO}^{2,c}}(\R^{n\times \tilde d})$. Then,
\begin{align}\label{Eq:ineqBMO}
\E\bigg[ \bigg(\int_0^T\e^{cr } |\sigma_r^\t Z_r|^2 \d r\bigg)^p\bigg]\leq p !\| Z\|_{\H^{2,c}_{{\rm BMO}}}^{2p}
\end{align}
\item Let $Z\in \Ho(\R^{n\times \tilde d})$, $\Zc=(Z_t^t)_{t\in [0,T]}$ and $c>0$, $\eps>0$. Then, $\P\as$
\begin{align*}
 \int_t^T \e^{cu} |\sigma^\t_u \Zc_u|^2 \d u \leq\int_t^T\e^{cu}  | \sigma^\t_uZ_u^t|^2 \d u+ \int_t^T\int_r^T \eps\e^{cu}   |\sigma_u^\t  Z^r_u|^2+ \eps^{-1} \e^{cu}  |\sigma_u^\t \partial Z^r_u |^2 \d u \d r,\ t\in [0,T].
\end{align*}
Moreover,  for any $t\in [0,T]$
\begin{align*}
\E_t\bigg[ \bigg( \int_t^T \e^{cu} |\sigma^\t_u \Zc_u|^2 \d u\bigg)^2\bigg]& \leq   6 \big( (1+T^2)\|Z\|_{\H^{2,2,c}_{\rm BMO}}^4+ T^2\|\partial Z\|_{\H^{2,2,c}_{\rm BMO}}^4\big)\\
\E_t\bigg[  \int_t^T \e^{cu} |\sigma^\t_u \Zc_u|^2 \d u \bigg]& \leq   (1+T)\|Z\|_{\H^{2,2,c}_{\rm BMO}}^2+ T\|\partial Z\|_{\H^{2,2,c}_{\rm BMO}}^2
\end{align*}
\end{enumerate}
\end{lemma}

\section{The BSDE system}\label{sec:infdimsystem}

For a fixed integers $d_1$ and $d_2$ we are given jointly measurable mappings $h$, $g$, $\xi$ and $\eta$, such that for any $(y,z,u,v,{\rm u})\in \R^{d_1\!}\times\R^{n\times d_1 \!}\times\R^{d_2\!}\times\R^{n\times d_2 \!}\times\R^{d_2} $
\begin{align*} 
\begin{split}
 &h:  [0,T] \times \Xc\times \R^{d_1}\! \times \R^{ n\times d_1}\! \times \R^{d_2}\!\times \R^{ n\times d_2}\! \times \R^{d_2}\! \longrightarrow\R^{d_1} ,\;   h_\cdot(\cdot ,y,z,u,v,{\rm u} )\in \Pc_{{\rm prog}}(\R^{d_1},\F),\\
&{g}: [0,T]^2 \times \Xc\times \R^{d_2}\! \times \R^{n\times d_2 }\! \times \R^{d_1}\!\times \R^{n\times d_1}\! \longrightarrow \R^{d_2} ,\;   g_\cdot(s,\cdot ,u,v,y,z)\in \Pc_{{\rm prog}}(\R^{d_2},\F),\\
 &\xi:  [0,T]\times \Xc\longrightarrow \R^{d_1}, \; \eta:[0,T]\times \Xc\longrightarrow \R^{d_2}.
 \end{split}
\end{align*}

Moreover, throughout this section we assume the following condition on $(\eta,g)$.

\begin{assumption}\label{Assumption:Differentiability}
$(s,u,v) \longmapsto  g_t(s,x,u,v,y,z)$ $($resp. $s\longmapsto \eta(s,x))$ is continuously differentiable, uniformly in $(t,x,y,z)$ $($resp. in $x)$. Moreover, the mapping $\nabla g:[0,T]^2\times \Xc \times (\R^{d_2}\!\times \R^{n\times d_2 } )^2 \! \times \R^{d_1}\!\times \R^{n\times d_1}\! \longrightarrow \R^{d_2}$ defined by \vspace{-1em}

\[ \nabla g_t (s,x,{\rm u},{\rm v},u, v ,y,z):=\partial_s g_t(s,x,u,v,y,z)+\partial_u g_t(s,x,u,v,y,z){\rm u}+\sum_{i=1}^n \partial_{v_{:i}} g_t(s,x,u,v,y,z){\rm v}_{i:},\]\vspace{-1em}

satisfies $\nabla g _\cdot(s,\cdot,{\rm u},{\rm v},u, v ,y,z)\in \Pc_{\rm prog}(\R^{d_2},\F), s\in [0,T]$. Set $\big( \tilde h_\cdot ,  \tilde g_\cdot(s),   \nabla  \tilde g_\cdot(s)\big) :=\big( h_\cdot(\cdot,{\bf 0},0), g_\cdot(s,\cdot,{\bf 0}), \partial _s g_\cdot(s,\cdot,{\bf 0})\big)$,
for ${\bf 0}:=(u,v,y,z)|_{(0,...,0)}$.
\end{assumption}

To ease the readability, in the rest of the document we will remove the dependence of the previous spaces on the underlying Euclidean spaces where the processes take value, i.e. we will write $\Sc^{\infty,c}$ for $\Sc^{\infty,c}(\R^{d_1})$. \medskip

Given $\Fc_T-$measurable $(\xi , \eta)$ and $h$ and $g$,  with $\partial_s \eta$ and $\nabla g$ given by \Cref{Assumption:Differentiability}, which we will refer as the data, we consider the following infinite family of BSDEs defined for $s\in [0,T]$, $\P\as$ for any $t\in [0,T]$ by
\begin{align}\label{Eq:systemBSDEq}\tag{$\Sc$}
\begin{split}
\Yc_t&=\xi(T,X_{\cdot\wedge T})+\int_t^T h_r(X,\Yc_r,\Zc_r, U_r^r,V_r^r,\partial U_r^r )\d r-\int_t^T  \Zc_r^\t \d X_r-\int_t^T \d \Nc_r, \\
U_t^s&=   \eta (s,X_{\cdot\wedge T})+\int_t^T  g_r(s,X,U_r^s,V_r^s, \Yc_r, \Zc_r) \d r-\int_t^T {V_r^s}^\t  \d X_r-\int_t^T \d M^s_r, \\
\partial U_t^s&=  \partial_s \eta (s,X_{\cdot\wedge T})+\int_t^T  \nabla g_r(s,X,\partial U_r^s,\partial V_r^s,U_r^s,V_r^s, \Yc_r, \Zc_r) \d r-\int_t^T\partial  {V_r^s}^\t  \d X_r-\int_t^T \d \partial M^s_r,
\end{split}
\end{align}

Under the above assumption, we are interested in studying the well--posedness of \eqref{Eq:systemBSDEq}. For $c>0$,  we define the space $(\Hc^c, \|\cdot \|_{\Hc^c})$, whose generic elements we denote $\mathfrak{h}=(Y,Z,N,U,V,M,\partial U,\partial V,\partial M)$, where 
\begin{align*}
 \Hc^c& :=\Sc^{\infty,c} \times \H^{2,c}_{{\rm BMO}} \times \M^{2,c} \times\Sc^{\infty,2,c}\times \Ho_{\rm BMO}\times \M^{2,2,c}  \times\Sc^{\infty,2,c} \times \H^{2,2,c}_{\rm BMO} \times \M^{2,2,c} 
\end{align*}
and $\|\cdot \|_{\Hc^c}$ denotes the respective induced norm.

\begin{remark}
\begin{enumerate}[label=$(\roman*)$, ref=.$(\roman*)$]
\item We highlight that system \eqref{Eq:systemBSDEq} is fully coupled. This means that a solution to the system has to be determined simultaneously. Moreover, the reader might notice our choice to preven the generator of the first {\rm BSDE} to depend on the diagonal of $\partial V$. This is due to our interest, as in {\rm \cite{hernandez2020unified}}, to establish the connection between these systems and {\rm type-I BSVIEs} \eqref{eq:typeIBSVIEfe}. For this, the presence of the diagonal $\partial U$ plays a key role. It should be clear from our arguments and {\rm \Cref{lemma:energy}} that these can be easily extended to accommodate this case.

\item We remark that for any $c\geq0$, the space $(\Hc^c, \|\cdot \|_{\Hc^c})$ is a Banach space. Indeed, it is clearly a normed space. Moreover, the fact that the spaces $\Sc^{\infty,c}$ and $\M^{2,c}$ are complete is classical in the literature. The completeness of the spaces $\H^{2,c}_{{\rm BMO}}$ and $\M^{2,c}_{{\rm BMO}}$, which are endowed with {\rm BMO}--type norms, follows from {\rm \citet*[Ch VII, Theorem 88]{dellacherie1982probabilities}}. Indeed, a dual space is always complete. Finally, the completeness extends clearly to spaces of the form $\Sc^{\infty,2,c} , \H^{2,2,c}_{{\rm BMO}} $ and $\M^{2,2,c}_{\rm BMO}$. 
\item We also remark that implicit in the definition of $ \Hc^c$ is the fact that $\partial V$ coincides, $\d t\otimes \d \P\ae$, with the density with respect to the Lebesgue measure of the application $s\longmapsto V^s$ which appears in the definition of the space $\Ho_{\rm BMO}$. This is for any $s\in [0,T]$
\[
V^s-V^0=\int_0^s \partial V^r \d r, \text{ in }\ \H^{2,c}_{{\rm BMO}}
\]
This contrasts with the result in the Lipschitz case obtained in {\rm \cite{hernandez2020unified}} where this was a consequence of the result, see also {\rm \cite[Remark 4.2]{hernandez2020unified}.} The reason for this is that the quadratic growth of the generator is incompatible with the contraction specified in the proof of {\rm \cite[Theorem 3.5]{hernandez2020unified}}.
\end{enumerate}
\end{remark}

 We now state precisely what we mean by a solution to \eqref{Eq:systemBSDEq}.

\begin{definition}\label{Def:solq}
A tuple $\mathfrak{h}=(\Yc,\Zc,\Nc,U,V,M, \partial U, \partial V, \partial M)$ is said to be a solution to \eqref{Eq:systemBSDEq} with terminal condition $(\xi,\eta )$ and generators $(f,g)$ under $\P$, if $\mathfrak{h}$ satisfies \eqref{Eq:systemBSDEq} $\P\as$, and, $\mathfrak{h} \in \Hc^c$ for some $c>0$.
\end{definition}

For $c>0$ and $R>0$, we define $\Bc_R\subseteq \Hc^c$ to be the subset of $\Hc^c$ of processes $(Y,Z,N,U,V,M)\in\Hc^c$ such that 
\[ \|(Y,Z,N,U,V,M)\|_{\Hc^c}^2\leq R^2.\]
The need to introduce $\Bc_R$ is inherent to the quadratic growth nature of \eqref{Eq:systemBSDEq}. Systems of the type of \eqref{Eq:systemBSDEq} were recently studied in a Lipschitz framework in \cite{hernandez2020me}. By choosing a weight $c$ large enough, and exploiting the fact that all weighted norms are equivalent, the authors of \cite{hernandez2020me} are able to obtain the well--posedness of \eqref{Eq:systemBSDEq} in the space $\Hf^2:=\S^{2}\times \H^{2}\times \M^{2}\times \S^{2,2}\times \H^{2,2}\times\M^{2,2}\times \S^{2,2}\times \H^{2,2}\times\M^{2,2}$. In the setting of this paper, the Lipschitz assumption for the generators is replaced by some kind of local quadratic growth. As a consequence, one cannot recover a contraction by simply choosing a weight large enough. In fact, as noticed in \cite{tevzadze2008solvability} and \cite{kazi2015quadratic1}, given our growth assumptions, our candidate for providing a contractive map is no longer Lipschitz continuous, but only locally Lipschitz continuous. The idea, initially proposed in \cite{tevzadze2008solvability}, is then to localise the usual procedure to a ball, thus making the application Lipschitz continuous again, and then to choose the radius of such a ball so as to recover a contraction. The crucial contribution of \cite{tevzadze2008solvability} is to show that such controls can be obtained by taking the data of the system small enough in norm.\medskip

Our procedure is inspired by this idea and incorporates it into the strategy devised in \cite{hernandez2020unified} to address the well--posedness of this kind of systems. We have decided to work on weighted spaces as, in our opinion, it does significantly simplify the arguments in the proof. We also mention that we tried to estimate the greatest ball, i.e. the largest radius $R$, for which such a localisation procedure leads to a contraction. Details are found in the proof. As we work on weighted spaces, we will find $c>0$ large enough so that, given data with sufficiently small norm, \eqref{Eq:systemBSDEq} has a unique solution in $\Bc_R\subseteq \Hc^c$. In words, throughout the proof, we will accumulate conditions on any candidate value for $c$ that allows to verify the necessary steps to obtain the result. As such, an appropriate value of $c$ must satisfy all such conditions. This should be clear from the statement of the result.

\subsection{The Lipschitz--quadratic case}\label{sec:lipschitzquadraticsystem}

\begin{assumption}\label{Assumption:LQgrowth} 
\begin{enumerate}[label=$(\roman*)$, ref=.$(\roman*)$,wide, labelwidth=!, labelindent=0pt]
\item \label{Assumption:LQgrowth:1} $\exists \tilde c\in (0,\infty)$ such that $(\xi,\eta,\partial_s \eta, \tilde f, \tilde g, \nabla \tilde g) \in \Lc^{\infty,\tilde c} \times\Lc^{\infty,2,\tilde c}  \times \L^{1,\infty,\tilde c}\times  \L^{1,\infty,2,\tilde c} $.
\item \label{Assumption:LQgrowth:2}  $\exists (L_y,L_u ,L_{\rm u})\in (0,\infty)^3$ s.t. $\forall(s,t,x,y, \tilde y, u, \tilde u, {\rm u}, \tilde {\rm u} ,z, v, {\rm v})\in [0,T]^2\times \Xc \times (\R^{d_1})^2\times (\R^{d_2})^4\times \R^{n\times {d_1}}\times (\R^{n\times d_2})^2$ 
\begin{align*}
&  |h_t(x,y,z,u,v, {\rm u})-h_t(x,\tilde y,z,\tilde u,v, \tilde {\rm u})| + |g_t(s,x,u,v,y,z)-g_t(s,x,\tilde u,v,\tilde y,z)|\\
&+ |\nabla g_t(s,x,{\rm u}, {\rm v},u,v,y,z)-g_t(s,x,\tilde {\rm u}, {\rm v},\tilde u,v,\tilde y,z)|  \leq  L_y|y-\tilde y|+ L_u|u-\tilde u|+L_{\rm u} |{\rm u}-\tilde {\rm u}|  ;
 \end{align*}
\item \label{Assumption:LQgrowth:3} $\exists (L_z,L_v ,L_{\rm v})\in (0,\infty)^3$, $\phi \in \H^{2,\tilde c}_{{ \rm BMO}}$ s.t. $\forall(s,t,x,y, u,{\rm u}, z, \tilde z, v, \tilde v, {\rm v}, \tilde{\rm  v})\in [0,T]^2\times \Xc\times \R^{d_1}\times (\R^{d_2})^2\times (\R^{n\times {d_1}})^2\times (\R^{n\times d_2})^4$ 
\begin{align*}
 & | h_t(x,y,z,u,v, {\rm u} )-h_t(x,y,\tilde z,u,\tilde v, {\rm u})-  (z-\tilde z)^\t \sigma_r(x)\phi_t  | + |g_t(s,x,u,v,y,z)-g_t(s,x,u,\tilde v,y,\tilde z)-  (v-\tilde v)^\t \sigma_r(x)\phi_t |\\
  &+ |\nabla g_t(s,x,{\rm u}, {\rm v},u,v,y,z)-\nabla g_t(s,x,{\rm u}, \tilde {\rm v}, u,\tilde v,y,\tilde z)-  ({\rm v}-\tilde {\rm v})^\t \sigma_r(x)\phi_t |\\
   \leq&\,  L_z\big| |\sigma^\t_r(x) z|\!+\! |\sigma^\t_r\! (x) \tilde z|\big| |\sigma^\t_r\!(x) (z-\tilde z)|\!+\! L_v\big| |\sigma^\t_r\! (x) v|\!+\!|\sigma^\t_r\! (x) \tilde v|\big| |\sigma^\t_r\! (x)(v-\tilde v )| \!+\! L_{\rm v}\big||\sigma^\t_r\! (x) {\rm v}|\!+\!|\sigma^\t_r\! (x) \tilde {\rm v}|\big| |\sigma^\t_r\! (x)({\rm v}-\tilde {\rm v} )| .
\end{align*}
\end{enumerate}
\end{assumption}

\begin{remark}
We now comment on the previous set of assumptions. 
	{\rm \Cref{Assumption:LQgrowth}\ref{Assumption:LQgrowth:1}} imposes integrability on the data of the system. 
	We highlight that in our setting we require the integral with respect to the time variable of the generators $(\tilde f, \tilde g, \nabla \tilde g)$ to be bounded. 
	This is in contrast to requiring the generators itself to be bounded. 
	On the other hand, {\rm \Cref{Assumption:LQgrowth}\ref{Assumption:LQgrowth:2}} imposes a classic uniformly Lipschitz growth assumption on the $(\Yc,U,\partial U)$ terms for the system. 
	Finally, {\rm \Cref{Assumption:LQgrowth}\ref{Assumption:LQgrowth:3}} imposes a slight generalisation of a local Lipschitz quadratic growth, this corresponds to the presence of the process $\phi$, and is similar to the one found in {\rm \cite{tevzadze2008solvability}}. 
	This property is almost equivalent to saying that the underlying function is quadratic in $z$. 
	The two properties would be equivalent if the process $\phi$ was bounded. 
	Here we allow something a bit more general by letting $\phi$ be unbounded but in $\H^2_{\rm BMO}$. 
	As we will see next, since this assumption allow us to apply the Girsanov transformation, we do not need to bound the processes and {\rm BMO}-type conditions are sufficient. 
	Lastly, we also remark that $\phi$ is common for the three generators and do not depend on $s$. 
	This is certainly a limitation in terms of the system \eqref{Eq:systemBSDEq}. 
	Yet, as we are working towards establishing the well-posedness of the {\rm BSVIE} \eqref{Eq:bsvieq}, we will see in {\rm \Cref{sec:BSVIE}}, namely \eqref{Eq:systemBSDEfq}, that we will chose $h$ and $g$ in such a way that such condition is sensible.
\end{remark}

As a preliminary to our analysis, we note that \Cref{Assumption:LQgrowth}\ref{Assumption:LQgrowth:3} can be simplified. This is the purpose of the next lemma. Therefore, without lost of generality in the rest of this section we assume $\phi=0$.

\begin{lemma}
Let 
\begin{align*}
&\hat  h_t(x,y,z,u,v,{\rm u}):=  h_t(x, y, z,u,v,{\rm u})- z^\t \sigma_r(x) {\phi_t},\; \hat  g_t(s,x,u,v,y,z):=  g_t(s,x,u,v,y,z)- v^\t\sigma_r(x) {\phi_t},\\
&\nabla \hat  g_t(s,x,{\rm u}, {\rm v}, u,v,y,z):=  g_t(s,x,{\rm u}, {\rm v},u,v,y,z)- {\rm v}^\t\sigma_r(x) {\phi_t},
\end{align*}
Then $(Y,Z,N,U,V,N)$ is a solution to \eqref{Eq:systemBSDEq} with terminal condition $(\xi,\eta)$ and generator $(f,g)$ under $\P$ if and only if $(Y,Z,N,U,V,N)$ is a solution to \eqref{Eq:systemBSDEq} with terminal condition $(\xi,\eta)$ and generator $(\hat f,\hat g )$ under $\Q\in \Prob(\Omega)$ given by 
\begin{align*}
\frac{\d \Q}{\d \P}=\Ec \bigg(\int_0^T  
\phi_t\cdot 
  \d X_t \bigg).
\end{align*}
\end{lemma}
\begin{proof}
We first verify that $\Q$ above is well--defined. Indeed, from the fact that $\phi  \in \H^{2,c}_{{ \rm BMO}}(\R^m)$ we have that the process defined above is a uniformly integrable martingale and Girsanov's theorem holds, see \cite[Section 1.3]{kazamaki1994continuous}. To verify the assertion of the lemma we note, for instance,  that
\[
\Yc_t= 
\xi+\int_t^T 
\hat h_r(X,\Yc_r,\Zc_r,U_r^r,V_r^r,\partial U_r^r) \d r
-\int_t^T Z_t^\t \big(
\d X_r- \sigma_r\phi_r \d r\big)
+\int_t^T 
\d N_r,\; t\in [0,T],\; \Q\as \]\end{proof}

To ease the presentation of our result, for $(\gamma,c, R) \in  (0,\infty)^3$, $\eps_i \in  (0,\infty)$, $i\in \{1,...,11 \}$, and $\kappa \in \N$, we define 
\begin{align*}
I_0^\eps&:=\|\xi\|_{\Lc^{\infty,c}}^2\! +2 \|\eta \|_{\Lc^{\infty,2,c}}^2\!   + (1+ \eps_1+\eps_2)\|\partial_s \eta \|_{\Lc^{\infty,2,c}}^2\!  +  \eps_3  \|  \tilde h\|^2_{\L^{1,\infty,c}} \! + ( \eps_4 +\eps_{5}) \|  \tilde g\|^2_{\L^{1,\infty,2,c}}  \! + ( \eps_1+\eps_2+ \eps_{6}) \| \nabla  \tilde g\|^2_{\L^{1,\infty,2,c}},\\
c^\eps &:= \max   \{ 2 L_y+ \eps_1^{-1} 7T  L_{\rm u}^2+( \eps_1+\eps_2) T  L_y^2  +\eps_7^{-1}\! L_u^2\! +\eps_8\!  +\eps_9\!  +\eps_{10}  ,  \, 2  L_u+ \eps_2^{-1} 7T +\eps_7 + \eps_8^{-1}\!  L_y^2 , \, 2  L_{\rm u}+ \eps_{10}^{-1} \! L_y^2 + \eps_{11}^{-1}\! L_u^2  \\
 &\hspace{4em}   2  L_u +( \eps_1+\eps_2)TL_u^2  + \eps_9^{-1}  L_y^2+\eps_{11}, \, 8 L_y+2 T   L_y+ 2 T  L_{\rm u} L_y, \ 4 L_u+2 T L_u +2 T L_{\rm u} L_u  \},\\
L_{\star}\!&:=\max \{L_z,L_v,L_{\rm v}\},\;  \Uc(\kappa):=\frac{1}{ 168 \kappa L_{\star}^{2 }}.
\end{align*}

\begin{remark}
Let us mention that the previous expressions arise in the analysis given our goal of finding the largest ball over which we can guarantee a contraction. 
	In particular, for this reason there are several degrees of freedom, $\eps_i$'s, that determine $c^\eps$. 
	In particular, let us mention that there are many simplifying choices that can be made, at the risk of loosing some flexibility. 
	For instance, letting $\tilde\eps_i\in (0,\infty)^5$, $i\in \{1,...,5\}$,  $\eps_1=\eps_2=\eps_6=\tilde \eps_1$, $\eps_3=\tilde\eps_2$, $\eps_4=\eps_5=\tilde \eps_3$, $\eps_7=\eps_{11}=\tilde \eps_4$ and $\eps_8=\eps_9=\eps_{10}=\tilde \eps_5$ we only need to choose 5 variables and
\begin{align*}
I_0^{\tilde \eps}&=\|\xi\|_{\Lc^{\infty,c}}^2\! +2 \|\eta \|_{\Lc^{\infty,2,c}}^2\!   + (1+2 \tilde \eps_1)\|\partial_s \eta \|_{\Lc^{\infty,2,c}}^2\!  +  \tilde\eps_2  \|  \tilde h\|^2_{\L^{1,\infty,c}} \! + 2\tilde \eps_3  \|  \tilde g\|^2_{\L^{1,\infty,2,c}}  \! + 3 \tilde \eps_1 \| \nabla  \tilde g\|^2_{\L^{1,\infty,2,c}},\\
c^\eps &= \max    \{ 2 L_y+ \tilde \eps_1^{-1} 7T  L_{\rm u}^2+2 \tilde \eps_1TL_y^2 +\tilde \eps_4^{-1}L_u^2+3\tilde \eps_5  , \   2  L_u+ \tilde \eps_1^{-1} 7T +\tilde \eps_4 + \tilde\eps_5^{-1}  L_y^2 ,  2  L_{\rm u}+\tilde \eps_{5}^{-1}  L_y^2 + \tilde\eps_{4}^{-1} L_u^2 \\
 &\hspace{3.8em}   2  L_u +2 \tilde \eps_1TL_u^2 + \tilde \eps_5^{-1}  L_y^2+\tilde \eps_{4}, \     8 L_y+2 T   L_y+ 2 T  L_{\rm u} L_y, \ 4 L_u+2 T L_u +2 T L_{\rm u} L_u  \}
\end{align*}
\end{remark}

\begin{assumption}\label{Assumption:wp:eta} 
Let $(\gamma,c, R) \in  (0,\infty)^3$, $\eps_i \in  (0,\infty)$, $i\in \{1,...,11 \}$, $\kappa \in \N$. We say {\rm \Cref{Assumption:wp:eta}} holds for $\kappa$ if
\begin{align*} 
(\sqrt{\eps_1+\eps_2+3\kappa}+\sqrt{3\kappa})^2\leq 28  \kappa,\;  I_0^\eps  \leq \gamma  R^2/\kappa , \;  R^2< \Uc(\kappa), \; c\geq c^\eps.
\end{align*}

\end{assumption}

\begin{theorem}\label{Thm:wp:smalldata}
Let  {\rm \Cref{Assumption:LQgrowth}} holds. 
	Suppose {\rm \Cref{Assumption:wp:eta}} holds for $\kappa=10$. 
	Then, there exists a unique solution to \eqref{Eq:systemBSDEq} in $\Bc_R\subseteq \Hc^c$ with  
\[ R^2< \frac{1}{168  \kappa L^2_\star}\]
\end{theorem}
\begin{remark}
We would like to comment on both the qualitative and quantitative statements in {\rm \Cref{Thm:wp:smalldata}}.
\begin{enumerate}[label=$(\roman*)$, ref=.$(\roman*)$,wide, labelwidth=!, labelindent=0pt]
\item As a result of our procedure, and consistent with the results available in the literature, we require the data of \eqref{Eq:systemBSDEq} to be sufficiently small in order to obtain the well--posedness of \eqref{Eq:systemBSDEq}. 
	We stress this property on the data is influenced not only by the value of $\gamma$ and $R$, but also, by the the value of $c$ which determines the norms.
\item In addition, we introduced {\rm \Cref{Assumption:wp:eta}} which depends on a parameter $\kappa\in \N$. 
	This parameter is related with the number of processes for which we have to keep track of the integrability of. 
	In particular, we mention that, in the proof we present in \Cref{sec:prooflinearquadratic}, in addition to the 9 elements that prescribe a solution to the system, we also control the norm of the diagonal process $(V_t^t)_{t\in [0,T]}$ which appears in the definition of $\Ho$. 
	An alternative line of reasoning is available by leaving out the norm of process $(V_t^t)_{t\in [0,T]}$ in the definition of $\Ho$ and exploit the inequality derived in {\rm \Cref{lemma:energy}}.
\item We also want to point out that we tried to obtain the largest ball for which the whole procedure goes through. 
	This can be appreciated in {\rm \Cref{Eq:Rwelldefined}} which introduces an upper bound on $R$ for which our candidate map is well--defined. 
	A word of caution nonetheless, since our bound of $R$ cannot be directly compared to the ones obtained in {\rm\cite{tevzadze2008solvability}} or {\rm\cite{kazi2015quadratic1}}. 
	Indeed, we recall that the norms involved depend on our choice of $c$. 
	As such, the radius of the ball, $\Bc_R\subseteq \Hc^c$, in {\rm \Cref{Thm:wp:smalldata}} depends on the choice of $c$. 
	Consequently, if we were to exploit the equivalence between all the spaces $\Hc^c$ the radius would have to be appropriately adjusted.
\end{enumerate}
\end{remark}

\begin{remark}\label{rmk:repproperty}
Alternative versions of {\rm \Cref{Thm:wp:smalldata}} are available. These are related to the orthogonal martingales.
\begin{enumerate}[label=$(\roman*)$, ref=.$(\roman*)$,wide, labelwidth=!, labelindent=0pt]
\item We first highlight that we are able to show, {\it a posteriori}, that the $(\Nc,M,\partial M)$ part of the solution to system \eqref{Eq:systemBSDEq} in $\Hc^c$ actually has a finite {\rm BMO}--type norm. 
	This is proved in {\rm \Cref{Thm:solutionBMOnorm}}. 
	We recall that in this quadratic setting with bounded terminal condition, this might be considered as a more natural type of integrability for the martingale part of the solutions.
\item Furthermore, one could wonder about how the statement of {\rm \Cref{Thm:wp:smalldata}} changes in the case where $(\Nc,M,\partial M)$ are required, {\it a priori}, to have finite {\rm BMO}--norm. 
	This would allow, for instance, to keep a control on the {\rm BMO}--norm inside the resulting ball. 
	We cover this in {\rm \Cref{Thm:solBMOnorm}}.
\item Finally, we consider the case when the representation property for $(\F,\P)$--martingales in terms of stochastic integrals with respect to $X$ holds. 
	This is the case, for example, if $\P$ is an extremal point of the convex hull of the set of probability measures satisfying the requirements in {\rm \Cref{sec:pstatement} }, see {\rm \cite[Theorem 4.29]{jacod2003limit}}. 
	In such a scenario, we have that $\Nc=M=\partial M=0$. This case is covered in {\rm \Cref{Thm:solBMOnorm}}.
\end{enumerate}
\end{remark}

\begin{theorem}\label{Thm:solutionBMOnorm}
Let {\rm \Cref{Assumption:LQgrowth}} hold and $\mathfrak{h}$ be a solution to \eqref{Eq:systemBSDEq} in the sense of {\rm \Cref{Def:solq}}. Then
\[ \|\Nc \|_{\M_{\rm BMO}^{2,c}}^2+\|M \|_{\M_{\rm BMO}^{2,2,c}}^2+\|\partial M \|_{\M_{\rm BMO}^{2,2,c}}^2<\infty\]
\end{theorem}

We now the address the well--posedness of \eqref{Eq:systemBSDEq} in the following two scenarii

\medskip
$(i)$ when $(\Nc,M,\partial M)\in \M_{\rm BMO}^{2,c}\times \M_{\rm BMO}^{2,2,c}\times \M_{\rm BMO}^{2,2,c}$, i.e. $(\Nc,M,\partial M)$ are required to have finite BMO norm; 

\medskip
$(ii)$ when the predictable martingale representation property for $(\F,\P)$--martingales in term of stochastic integral with respect to $X$ hold, i.e. both $\Nc$, $M$ and $\partial M$ vanish. 

\medskip
For this, let us introduce the spaces $(\widehat \Hc^c, \|\cdot \|_{\widehat \Hc^c})$ and $(\overline \Hc^c, \|\cdot \|_{\overline \Hc^c})$ for $c>0$ where 
\begin{gather*}
\widehat \Hc^c:=\Sc^{\infty,c} \times \H^{2,c}_{{\rm BMO}} \times \M^{2,c}_{\rm BMO} \times\Sc^{\infty,2,c} \times \overline \H^{2,2,c}_{{\rm BMO}} \times \M^{2,2,c}_{\rm BMO} \times\Sc^{\infty,2,c} \times \H^{2,2,c}_{{\rm BMO}} \times \M^{2,2,c}_{\rm BMO} , \\
  \overline \Hc^c:=\Sc^{\infty,c} \times \H^{2,c}_{{\rm BMO}}   \times\Sc^{\infty,2,c} \times \overline \H^{2,2,c}_{{\rm BMO}} \times\Sc^{\infty,2,c} \times \H^{2,2,c}_{{\rm BMO}}.
\end{gather*}
and $\|\cdot \|_{\widehat \Hc^c}$ and $\|\cdot \|_{\overline \Hc^c}$ denote the associated norms. Moreover, $\widehat \Bc_R  \subseteq \widehat \Hc^c$ $\big($resp. $\overline \Bc_R  \subseteq \overline \Hc^c\, \big)$ denotes the ball of radius $R$ for the norm $\|\cdot\|_{\widehat \Hc^c}$ $\big($resp. $\|\cdot\|_{\overline \Hc^c}\, \big)$.

\begin{theorem}\label{Thm:solBMOnorm}
Let {\rm \Cref{Assumption:LQgrowth}} hold. If in addition
\begin{enumerate}[label=$(\roman*)$, ref=.$(\roman*)$,wide, labelwidth=!, labelindent=0pt]
\item {\rm \Cref{Assumption:wp:eta}} holds for $\kappa=11$, then, there exists a unique solution to \eqref{Eq:systemBSDEq} in $\widehat \Bc_R\subseteq (\widehat \Hc^c,\|\cdot\|_{\widehat \Hc^c})$.
\item {\rm \Cref{Assumption:wp:eta}} holds for $\kappa=8$, and the predictable martingale representation property for $(\F,\P)$--martingales in term of stochastic integral with respect to $X$ hold. Then, there exists a unique solution to \eqref{Eq:systemBSDEq} in $\overline \Bc_R\subseteq (\overline \Hc^c,\|\cdot\|_{\overline \Hc^c})$.
\end{enumerate}
\end{theorem}

\subsection{The quadratic case}\label{sec:quadraticsystem}

Our approach allows us to consider also the case of quadratic generators as follows.

\begin{assumption}\label{Assumption:Qgrowth} 
\begin{enumerate}[label=$(\roman*)$, ref=.$(\roman*)$,wide, labelwidth=!, labelindent=0pt]
\item \label{Assumption:Qgrowth:1} $\exists \bar c>0$ s.t $(\xi,\eta,\partial \eta, \tilde f, \tilde g, \nabla \tilde g) \in \Lc^{\infty,\bar c} \times\Lc^{\infty,\bar c,2}  \times \Lc^{\infty,\bar c,2}   \times \L^{1,\infty,\bar c} \times  \L^{1,\infty,\bar c,2} \times  \L^{1,\infty,\bar c,2} $.
\item \label{Assumption:Qgrowth:3} $\exists (L_y,L_u,L_{\rm u})\in (0,\infty)^3$,  s.t.  $\forall(s,t,x,y, \tilde y, u, \tilde u, {\rm u}, \tilde {\rm u} ,z, v, {\rm v})\in [0,T]^2\times \Xc \times (\R^{d_1})^2\times (\R^{d_2})^4\times \R^{n\times {d_1}}\times (\R^{n\times d_2})^2$ 
\begin{align*}
& | h_t(x,y,z,u,v,{\rm u})-h_t(x,\tilde y, z,\tilde u,  v, \tilde {\rm u})|   \leq  L_y|y-\tilde y|  \big(|y|+|\tilde y|\big)+ L_u|u-\tilde u| \big(|u|+|\tilde u|\big)+ L_{\rm u}|{\rm u}-{\rm \tilde u} | ,\\
&   | g_t(s,x,u,v,y,z)-g_t(s,x,\tilde u,  v, \tilde y, z)| +   | \nabla g_t(s,x,{\rm u},{\rm v},u,v,y,z)-\nabla g_t(s,x,{\rm \tilde u},  {\rm v},\tilde u, v \tilde y, z)|\\
   \leq &\,  L_{\rm u} |{\rm u}-\tilde {\rm u}| \big||{\rm u}|+|\tilde {\rm u}|\big| + L_u|u-\tilde u| \big||u|+|\tilde u|\big|+ L_y|y-\tilde y|  \big||y|+|\tilde y|\big|;
\end{align*}
\item $\exists (L_z,L_v ,L_{\rm v})\in (0,\infty)^3$, $\phi \in \H^{2,\tilde c}_{{ \rm BMO}}$ s.t. $\forall(s,t,x,y, u,{\rm u}, z, \tilde z, v, \tilde v, {\rm v}, \tilde{\rm  v})\in [0,T]^2\times \Xc\times \R^{d_1}\times (\R^{d_2})^2\times (\R^{n\times {d_1}})^2\times (\R^{n\times d_2})^4$ 
\begin{align*}
 & | h_t(x,y,z,u,v, {\rm u} )-h_t(x,y,\tilde z,u,\tilde v, {\rm u})-  (z-\tilde z)^\t \sigma_r(x)\phi_t  | + |g_t(s,x,u,v,y,z)-g_t(s,x,u,\tilde v,y,\tilde z)-  (v-\tilde v)^\t \sigma_r(x)\phi_t |\\
&+  |\nabla g_t(s,x,{\rm u}, {\rm v},u,v,y,z)-\nabla g_t(s,x,{\rm u}, \tilde {\rm v}, u,\tilde v,y,\tilde z)-  ({\rm v}-\tilde {\rm v})^\t \sigma_r(x)\phi_t |\\
   \leq&\,  L_z\big| |\sigma^\t_r(x) z|\!+\! |\sigma^\t_r\! (x) \tilde z|\big| |\sigma^\t_r\!(x) (z-\tilde z)|\!+\! L_v\big| |\sigma^\t_r\! (x) v|\!+\!|\sigma^\t_r\! (x) \tilde v|\big| |\sigma^\t_r\! (x)(v-\tilde v )| \!+\! L_{\rm v}\big||\sigma^\t_r\! (x) {\rm v}|\!+\!|\sigma^\t_r\! (x) \tilde {\rm v}|\big| |\sigma^\t_r\! (x)({\rm v}-\tilde {\rm v} )| .
\end{align*}
\end{enumerate}
\end{assumption} 
\begin{remark}
In the previous set of assumptions we have allowed the generators to have quadratic growth in all of the terms, with the exception of $h$ on the term ${\rm u}$. 
	The reason for this is twofold: 
	$(i)$ for the kind of systems that motivate our analysis the term $(\partial U_t^t)_{t\in [0,T]}$ always appears linearly in the generator, 
	$(ii)$ when making the connection with the {\rm type-I BSVIE} \eqref{eq:typeIBSVIEfe} $(\partial U_t^t)_{t\in [0,T]}$ plays the auxiliary role of keeping track of the diagonal processes $(U_t^t)_{t\in [0,T]}$ and $(V_t^t)_{t\in [0,T]}$ and for this it suffices that it appears linearly in the generator. 
	This is, the Lipschitz assumption on this variable will suffice for the purposes of our results and the problems that motivated our study.
\end{remark}

As before, we need to introduce some auxiliary notation. For $(\gamma,c, R) \in  (0,\infty)^3$, $\eps_i \in  (0,\infty)$, $i\in \{1,...,6 \}$, and $\kappa \in \N$, we let $I_0^\eps$ be as in \Cref{sec:lipschitzquadraticsystem} and define
\begin{align*}
L_{\star}& :=\max \{L_y,L_u,L_{\rm u},L_z,L_v,L_{\rm v}\},\; c^\eps  := \max \{  \eps_1^{-1} 7T  L_{\rm u}^2,  \eps_2^{-1} 7T ,\; 2 L_{\rm u} \} ,\; \Uc(\kappa) := \frac{1}{ 336 \kappa L_{\star}^{2 } \max\{2,T^2\} }.
\end{align*}

\begin{assumption}\label{Assumption:wpq:eta} 
Let $(\gamma,c, R) \in  (0,\infty)^3$, $\eps_i \in  (0,\infty)$, $i\in \{1,...,6 \}$, and $\kappa \in \N$. We say {\rm \Cref{Assumption:wpq:eta}} holds for $\kappa$ if
\begin{align*} 
(\sqrt{\eps_1+\eps_2+3\kappa}+\sqrt{3\kappa})^2\leq 56  \kappa,\;  I_0^\eps  \leq \gamma  R^2/\kappa , \;  R^2< \Uc(\kappa) ,\;  c\geq c^\eps .
\end{align*}
\end{assumption}

\begin{theorem}\label{Thm:wpq:smalldata}
Let  {\rm \Cref{Assumption:Qgrowth}} holds. Suppose {\rm \Cref{Assumption:wpq:eta}} holds for $\kappa=10$. Then, there exists a unique solution to \eqref{Eq:systemBSDEq} in $\Bc_R\subseteq \Hc^c$ with  
\[ R^2< \frac{1}{336 \kappa L_{\star}^{2 }\max\{2,T^2\} }\]
\end{theorem}

\begin{remark}
We remark that the best way to appreciate the previous result is by contrasting it with our well-posedness result in the linear quadratic case. 
	For simplicity, let us assume $L_y=L_u=L_{\rm u}=L_z=L_v=L_{\rm v}$. 
	Regarding the constraint on the weight parameter of the resulting norm, the result is pretty much in the same order of magnitude. 
	Nevertheless, the most noticeable feature is that by allowing the system to have quadratic growth the greatest radius under which our argument is able to guarantee the well-posedness of the solution decreases by a factor $2\max\{2,T^2\}$. 
	This quantity can be significant in light of its dependence on the time horizon $T$.
\end{remark}

\section{Muldimensional type-I BSVIEs of quadratic growth}\label{sec:BSVIE}

We now address the well-posedness of multidimensional linear quadratic and quadratic type-I BSVIEs. Let $d$ be a non-negative integer, and $f$ and $\xi$ be jointly measurable functionals such that for any $(s,y,z,u,v)\in [0,T]\times (\R^d \times \R^{n\times d})^2$
\begin{align*}
 & f: [0,T]^2\times \Xc\times (\R^d \times \R^{n\times d})^2 \longrightarrow \R^d ,\; f_\cdot(s,\cdot ,y,z,u,v )\in \Pc_{{\rm prog}}(\R^{d},\F),\\
& \xi: [0,T]\times \Xc\longrightarrow \R^d,\; \xi(s,\cdot) \; \text{\rm is}\; \Fc\text{\rm -measurable}.
\end{align*}

The main result in this section follows exploiting the well-posedness of \eqref{Eq:systemBSDEq}. Therefore, we work under the following set of assumptions.

\begin{assumption}\label{Assumption:SystemBSVIEwp} $(s,y,z)\longmapsto  f_t(s,x,y,z,u,v)$ $($resp. $s\longmapsto \xi(s,x))$ is continuously differentiable, uniformly in $(t,x,u,v)$ $($resp. in $x)$. Moreover, the mapping $\nabla f:[0,T]^2\times \Xc \times (\R^{d}\times \R^{n\times d } )^3\longrightarrow \R^{d}$ defined by \vspace{-1em}

\[\nabla f_t(s,x,{\rm u}, {\rm v}, y,z ,u,v):=\partial_s f_t(s,x,y,z,u,v)+\partial_y f_t(s,x,y,z,u,v){\rm u}+\sum_{i=1}^n \partial_{z_{:i}} f_t(s,x,y,z,u,v){\rm v}_{i:},
\] \vspace{-1em}

satisfies $\nabla f_\cdot(s,\cdot,y,z,u,v,{\rm u}, {\rm v})\in \Pc_{\rm prog}(\R^d,\F), s\in [0,T].$ Set $\big( \tilde f_\cdot ,  \tilde f_\cdot(s),   \nabla  \tilde f_\cdot(s)\big) :=\big( f_\cdot(\cdot,{\bf 0},0), f_\cdot(s,\cdot,{\bf 0}), \partial _s f_\cdot(s,\cdot,{\bf 0})\big)$,
for ${\bf 0}:=(u,v,y,z)|_{(0,...,0)}$.
\end{assumption}

Let $(\Hc^{\star,c}, \|\cdot \|_{\Hc^{\star,c}})$ denote the space of $(Y,Z,N)\in \Hc^{\star,c}$ such that $\|(Y,Z,N)\|_{\Hc^{\star,c}}<\infty$ where
\begin{gather*}
 \Hc^{\star,c}:= \Sc^{\infty,2,c}\times \Ho_{\rm BMO} \times \M^{2,2,c}_{\rm BMO},\;
 \|\cdot\|_{\Hc^{\star,c}}:=\|Y\|_{\Sc^{\infty,2,c}}^2 + \|Z\|_{\Ho}^2+\|N \|_{\M^{2,2}_{\rm BMO}}^2.
\end{gather*}

We consider the $n$-dimensional type-I BSVIE on $(\Hc^\star,\|\cdot\|_{\Hc^\star})$
\begin{align}\label{Eq:bsvieq}
Y_t^s =   \xi (s,X)+\int_t^T  f_r(s,X,Y_r^s,Z_r^s, Y_r^r, Z_r^r) \d r-\int_t^T {Z_r^s}^\t\d X_r-\int_t^T \d N^s_r,\; t\in [0,T],\; \P\as,\; s\in [0,T].
\end{align}

We work under the following notion of solution.

\begin{definition}\label{Def:soltypeIBSVIEfeq}
We say $(Y,Z,N)$ is a solution to the {\rm type-I BSVIE} \eqref{Eq:bsvieq} if $(Y,Z,N)\in \Hc^\star$ verifies \eqref{Eq:bsvieq}.
\end{definition}

We may consider the system, given for any $s\in [0,T]$ by
\begin{align*}
\Yc_t&=\xi(T,X)+\int_t^T \Big( f_r(r,X,\Yc_r,\Zc_r,Y_r^r,Z_r^r)-\partial Y_r^r\Big)\d r-\int_t^T \Zc_r^\t \d X_r-\int_t^T \d \Nc_r,\; t\in [0,T],\;\P\as, \\
Y_t^s&=   \xi (s,X)+\int_t^T  f_r(s,X,Y_r^s,Z_r^s, \Yc_r, \Zc_r) \d r-\int_t^T {Z_r^s}^\t  \d X_r-\int_t^T \d N^s_r,\;  t\in [0,T],\; \P\as,\label{Eq:systemBSDEfq}\tag{$\Sc_f$}\\
\partial Y_t^s&=  \partial_s \xi (s,X)+\int_t^T  \nabla f_r(s,X, \partial Y_r^s,\partial Z_r^s ,Y_r^s,Z_r^s, \Yc_r, \Zc_r) \d r-\int_t^T\partial  {Z_r^s}^\t  \d X_r-\int_t^T \d \partial N^s_r,\; t\in [0,T],\; \P\as
\end{align*}

\begin{remark}
Let us briefly comment that our set-up for the study {\rm type-I BSVIE} \eqref{Eq:bsvieq} is based on the systems introduced in {\rm  \Cref{sec:infdimsystem}}. As such, the necessity of the set of assumptions in {\rm \Cref{Assumption:SystemBSVIEwp}} is clear.
\end{remark}

We are now in position to prove the main result of this paper. The next result shows that under the previous choice of data for \eqref{Eq:systemBSDEfq}, its solution solves the {\rm type-I BSVIE} with data $(\xi,f)$ and {\it vice versa} in both the linear quadratic and quadratic case. For this we introduce the following set of assumptions.

\begin{assumption}\label{Assumption:LQgrowthBSVIE} 
\begin{enumerate}[label=$(\roman*)$, ref=.$(\roman*)$,wide, labelwidth=!, labelindent=0pt]
\item $\exists \tilde c\in (0,\infty)$ such that $(\xi,\eta,\partial_s \eta, \tilde f, \tilde g, \nabla \tilde g) \in \Lc^{\infty,\tilde c} \times\Lc^{\infty,2,\tilde c}  \times \L^{1,\infty,\tilde c}\times  \L^{1,\infty,2,\tilde c} $.
\item $\exists (L_y,L_u ,L_{\rm u})\in (0,\infty)^3$ s.t. $\forall(s,t,x,y, \tilde y, u, \tilde u, {\rm u}, \tilde {\rm u} ,z, v, {\rm v})\in [0,T]^2\times \Xc \times  (\R^{d})^6\times \R^{n\times {d}}\times (\R^{n\times d})^3$ 
\begin{align*}
  &  |f_t(s,x,y,z,u,v)-f_t(s,x,\tilde y,z,\tilde u,v)| +
|\nabla f_t(s,x,{\rm u}, {\rm v},y,z,u,v)-\nabla f_t(s,x,\tilde {\rm u}, {\rm v},\tilde y,z,\tilde u,v)| \\
 \leq &\,  L_y|y-\tilde y|+ L_u|u-\tilde u|+L_{\rm u} |{\rm u}-\tilde {\rm u}|  ;
 \end{align*}
\item $\exists (L_z,L_v ,L_{\rm v})\in (0,\infty)^3$, $\phi \in \H^{2,\tilde c}_{{ \rm BMO}}$ s.t. $\forall(s,t,x,y, u,{\rm u}, z, \tilde z, v, \tilde v, {\rm v}, \tilde{\rm  v})\in [0,T]^2\times \Xc\times (\R^{d})^3\times (\R^{n\times {d}})^6$ 
\begin{align*}
& |f_t(s,x,y,z,u,v)-f_t(s,x,y,\tilde z,u,\tilde v)-  (z-\tilde z)^\t \sigma_r(x)\phi_t | \\
& +  |\nabla f_t(s,x,{\rm u}, {\rm v},y,z,u,v)-\nabla f_t(s,x,{\rm u}, \tilde {\rm v},y,\tilde z, u,\tilde v)-  ({\rm v}-\tilde {\rm v})^\t \sigma_r(x)\phi_t |\\
   \leq & \,  L_z\big| |\sigma^\t_r(x) z|\!+\! |\sigma^\t_r\! (x) \tilde z|\big| |\sigma^\t_r\!(x) (z-\tilde z)|\!+\! L_v\big| |\sigma^\t_r\! (x) v|\!+\!|\sigma^\t_r\! (x) \tilde v|\big| |\sigma^\t_r\! (x)(v-\tilde v )| \!+\! L_{\rm v}\big||\sigma^\t_r\! (x) {\rm v}|\!+\!|\sigma^\t_r\! (x) \tilde {\rm v}|\big| |\sigma^\t_r\! (x)({\rm v}-\tilde {\rm v} )| .
\end{align*}
\end{enumerate}
\end{assumption}

\begin{assumption}\label{Assumption:QgrowthBSVIE} 
\begin{enumerate}[label=$(\roman*)$, ref=.$(\roman*)$,wide, labelwidth=!, labelindent=0pt]
\item $\exists \tilde c\in (0,\infty)$ such that $(\xi,\eta,\partial_s \eta, \tilde f, \tilde g, \nabla \tilde g) \in \Lc^{\infty,\tilde c} \times\Lc^{\infty,2,\tilde c}  \times \L^{1,\infty,\tilde c}\times  \L^{1,\infty,2,\tilde c} $.
\item $\exists (L_y,L_u ,L_{\rm u})\in (0,\infty)^3$ s.t. $\forall(s,t,x,y, \tilde y, u, \tilde u, {\rm u}, \tilde {\rm u} ,z, v, {\rm v})\in [0,T]^2\times \Xc \times  (\R^{d})^6\times \R^{n\times {d}}\times (\R^{n\times d_2})^3$ 
\begin{align*}
  & |f_t(s,x,y,z,u,v)-f_t(s,x,\tilde y,z,\tilde u,v)| +
|\nabla f_t(s,x,{\rm u}, {\rm v},y,z,u,v)-\nabla f_t(s,x,\tilde {\rm u}, {\rm v},\tilde y,z,\tilde u,v)| \\
 \leq &  L_y|y-\tilde y|\big| |y|+|\tilde y|\big| + L_u|u-\tilde u|\big|  |u|+|\tilde u|\big| +L_{\rm u} |{\rm u}-\tilde {\rm u}|\big|  |{\rm u}|+|\tilde {\rm u}|\big|   ;
 \end{align*}
\item  $\exists (L_z,L_v ,L_{\rm v})\in (0,\infty)^3$, $\phi \in \H^{2,\tilde c}_{{ \rm BMO}}$ s.t. $\forall(s,t,x,y, u,{\rm u}, z, \tilde z, v, \tilde v, {\rm v}, \tilde{\rm  v})\in [0,T]^2\times \Xc\times (\R^{d})^3\times (\R^{n\times {d}})^6$ 
\begin{align*}
& |f_t(s,x,y,z,u,v)-f_t(s,x,y,\tilde z,u,\tilde v)-  (z-\tilde z)^\t \sigma_r(x)\phi_t |\\
&+  |\nabla f_t(s,x,{\rm u}, {\rm v},y,z,u,v)-\nabla f_t(s,x,{\rm u}, \tilde {\rm v},y,\tilde z, u,\tilde v)-  ({\rm v}-\tilde {\rm v})^\t \sigma_r(x)\phi_t |\\
   \leq& \,  L_z\big| |\sigma^\t_r(x) z|\!+\! |\sigma^\t_r\! (x) \tilde z|\big| |\sigma^\t_r\!(x) (z-\tilde z)|\!+\! L_v\big| |\sigma^\t_r\! (x) v|\!+\!|\sigma^\t_r\! (x) \tilde v|\big| |\sigma^\t_r\! (x)(v-\tilde v )| \!+\! L_{\rm v}\big||\sigma^\t_r\! (x) {\rm v}|\!+\!|\sigma^\t_r\! (x) \tilde {\rm v}|\big| |\sigma^\t_r\! (x)({\rm v}-\tilde {\rm v} )| .
\end{align*}
\end{enumerate}
\end{assumption}

\begin{theorem}\label{Thm:wpqbsvie}
Let {\rm \Cref{Assumption:SystemBSVIEwp}} hold.  Then, the well-posedness of \eqref{Eq:systemBSDEfq} is equivalent to that of the {\rm type-I BSVIE} \eqref{Eq:bsvieq} if either:
\begin{enumerate}[label=$(\roman*)$, ref=.$(\roman*)$,wide, labelwidth=!, labelindent=0pt]
\item {\rm \Cref{Assumption:LQgrowthBSVIE}} holds and {\rm \Cref{Assumption:wp:eta}} holds for $\kappa=7$. In such case, there exists a unique solution to the {\rm type-I BSVIE} \eqref{Eq:bsvieq} in $\Bc_R\subseteq \Hc^{\star,c}$ with  
\[ R^2< \frac{1}{168  \kappa L^2_\star}\]
\item {\rm \Cref{Assumption:QgrowthBSVIE}} holds and {\rm \Cref{Assumption:wpq:eta}} holds for $\kappa=7$. In such case, there exists a unique solution to the {\rm type-I BSVIE} \eqref{Eq:bsvieq} in $\Bc_R\subseteq \Hc^{\star,c}$ with  
\[ R^2< \frac{1}{336 \kappa L_{\star}^{2 }\max\{2,T^2\} }\]
\end{enumerate}
\begin{proof}
Let us first note that the second part of the statements in $(i)$ and $(ii)$ are a direct consequente of \Cref{Thm:wp:smalldata} and \Cref{Thm:wpq:smalldata}, respectively. \medskip

Let us first argue why it suffices to have {\rm \Cref{Assumption:wp:eta}} and {\rm \Cref{Assumption:wpq:eta}} hold for $\kappa=7$ instead of $10$. This follows from the specification of data for \eqref{Eq:systemBSDEfq}. Indeed, following the notation of the proof of \Cref{Thm:wp:smalldata} in this case we have that $\Yc=\Uc$, $\Zc=\Vc$ and $\Nc=\Mc$ so the auxiliary equation introduced in {\bf step 1} $(ii)$ is not necessary and, as from \eqref{Eq:thm:wd:ineq:final}, the argument in the proof holds with $7$ instead on $10$. \medskip

We are only left to argue the equivalence of the solutions. Let us argue $(i)$, the argument for $(ii)$ is analogous. For this we follow \cite[Theorem 4.3]{hernandez2020unified}. Let $(\Yc,\Zc,\Nc,Y,Z,N,\partial Y,\partial Z,\partial N)\in \Bc_R\subseteq \Hc^c$ be a solution to \eqref{Eq:systemBSDEfq}. It then follows from \cite[Lemma 6.2]{hernandez2020unified} that
\begin{align*}
Y_t^t&=\xi(T,X)+\int_t^T \Big( f_r(r,X,Y_r^r,Z_r^r,\Yc_r,\Zc_r)-\partial Y_r^r\Big)\d r-\int_t^T {Z_r^r}^\t \d X_r- \int_t^T \d \widetilde N_r, \ t\in [0,T],\; \P\as,
\end{align*}
where $\widetilde N_t := N_t^t-\int_0^t \partial N_r^r \d r$, $t\in[0,T]$, and $\widetilde N\in \M^{2,c}$.  As in \Cref{Thm:solutionBMOnorm}, we obtain that $\widetilde N\in \M^{2,c}_{\rm BMO}$. This shows that $\big((Y_t^t)_{t\in[0,T]},(Z_t^t)_{t\in[0,T]},\Yc_\cdot,\Zc_\cdot,  (\widetilde N_t)_{t\in [0,T]} \big)$, solves the first BSDE in \eqref{Eq:systemBSDEfq}. It then follows from the well-posedness of \eqref{Eq:systemBSDEfq}, which holds by \Cref{Assumption:SystemBSVIEwp}, \ref{Assumption:QgrowthBSVIE}, {\rm \ref{Assumption:wpq:eta}} and \Cref{Thm:wp:smalldata}, that $\big((Y_t^t)_{t\in[0,T]},(Z_t^t)_{t\in[0,T]}, (\widetilde N_t)_{t\in [0,T]}\big)=(\Yc_\cdot,\Zc_\cdot, \Nc_\cdot)$ in $\Sc^{2,c}\times \H^{2,c}_{\rm BMO} \times\M^{2,c}_{\rm BMO}$ and the result follows.\medskip

We show the converse result. Let $(Y,Z,N)\in \Bc_R \subseteq \Hc^{\star,c}$ be a solution to type-I BSVIE \eqref{Eq:bsvieq}. It is clear that the processes $\Yc:=(Y_t^t)_{t\in [0,T]},\Zc:=(Z_t^t)_{t\in [0,T]},\Nc:=(N_t^t)_{t\in [0,T]}$ are well-defined. Then, since \Cref{Assumption:SystemBSVIEwp} holds and $(\Yc,\Zc, Y,Z,N)\in \L^{1, \infty,c} \times \H^{2,c}_{\rm BMO} \times \Sc^{2,2,c}\times \Ho_{\rm BMO} \times \M^{2,2,c}_{\rm BMO}$, we can apply \cite[Lemma 6.2]{hernandez2020unified} and obtain the existence of $(\partial Y, \partial Z, \partial N)\in \Sc^{2,2,c} \times \H^{2,2,c}  \times \M^{2,2,c}$ such that for $s\in [0,T]$
\[
\partial Y_t^s=  \partial_s \xi (s,X)+\int_t^T  \nabla f_r(s,X, \partial Y_r^s,\partial Z_r^s,Y_r^s,Z_r^s, Y_r^r, Z_r^r) \d r-\int_t^T\partial  {Z_r^s}^\t  \d X_r-\int_t^T \d \partial N^s_r, \, t\in [0,T],\,  \P\as
\]
Moreover, from the fact that {\rm \Cref{Assumption:wp:eta}} holds for $\kappa=7$ we obtain that $\|\partial Y\|_{\Sc^{2,2,c}}^2+\|\partial Z\|_{\H_{\rm BMO}^{2,2,c}}^2+\|\partial N\|_{\M^{2,2,c}}^2\leq R^2$.\medskip

Let us claim that $\mathfrak{h}:= (\Yc , \Zc ,  \widetilde N , Y,  Z,  N,  \partial Y, \partial Z, \partial N)$ is a solution to \eqref{Eq:systemBSDEfq}, where $\widetilde N_t:=N_t^t-\int_0^t \partial N_r^r \d r$, $t\in [0,T]$. For this, we first note that in light of \cite[Lemmata 6.1-6.2]{hernandez2020unified} we have that 
\begin{align}\label{Eq:claimwpbsvie0}
\Yc_t&=\xi(T,X)+\int_t^T h_r(X,\Yc_r,\Zc_r,Y_r^r,Z_r^r,\partial Y_r^r)\d r-\int_t^T {\Zc_r}^\t \d X_r-\int_t^T \d \widetilde N_r , \; t \in [0,T],\; \P\as.
\end{align}
Now, $\widetilde \Nc\in \M^{2,2,c}_{\rm BMO}$ follows as in \Cref{Thm:solutionBMOnorm}. As in {\bf step 1} $(iii)$ in the proof of \Cref{Thm:wp:smalldata}, we obtain that $\Yc\in \Sc^{2,c}$. We are only left to argue $\| \mathfrak{h}\|\leq R$. This follows readily following {\bf step 1} $(iv)$ in the proof of \Cref{Thm:wp:smalldata} and the fact that $\|\Zc\|_{\H^{2,c}}^2+\|Z\|_{\Ho_{\rm BMO}}^2+\|\partial Z\|_{\H^{2,2,c}_{\rm BMO}}^2\leq R^2$.
\end{proof}
\end{theorem}

\section{On the flow property for type-I extended BSVIEs}\label{sec:flowppty}

In this section we present a brief discussion on the so-called \emph{flow property} in the context of the type-I BSVIEs studied in the previous section.
	Contrary to the case of BSDEs, said BSVIEs fail to have the so-called \emph{flow property}, this is
\begin{align}\label{eq:noflowproperty}
Y_t^t \neq  Y_s^s+\int_t^s  f_r(t,X,Y_r^t,Z_r^t, Y_r^r, Z_r^r) \d r-\int_t^s {Z_r^t}^\t\d X_r-\int_t^s \d N^t_r.
\end{align}

This is known to be a distinguishing and recurrent feature of Volterra processes. Indeed, \citet*{viens2019martingale} studied the dynamic backward problems in a framework where the forward state process satisfies a Volterra type SDE. 
	Note that BSVIEs can be interpreted as Volterra-type extensions of the classic dynamic backward problem, e.g. of computing conditional expectations. 
	Thus, the situation described in \eqref{eq:noflowproperty} corresponds to the antipodal scenario of \cite{viens2019martingale} in the sense that the forward state process satisfies a classic SDE and the Volterra feature is present in the kind of the process we want to take conditional expectation of. 
	However, beneath the surface of both scenarii is the presence of a certain manifestation of time inconsistency. 	
	The main result in \cite{viens2019martingale} is a functional It\^o formula. 
	For this, the nature of the forward process, neither Markov processes nor a semimartingales, makes necessary to concatenate the observed path up to the current time with a certain smooth observable curve derived from the distribution of the future paths. 
	We stress that this new feature is due to the underlying time inconsistency. All in all, having access to an It\^o formula unravels a type of flow property that holds for this kind of problems. \medskip

As such, a natural question in our framework is: \emph{can we recover, in a appropriate sense, a flow property for type-I extended {\rm BSVIEs}?} 
	By appropriate sense we mean that the sought version of this property must be compatible with the negative result in \eqref{eq:noflowproperty}. 
	The answer to this question can be extracted from the proof of \Cref{Thm:wpqbsvie}. Indeed, as $\big((Y_t^t)_{t\in[0,T]},(Z_t^t)_{t\in[0,T]}, (\widetilde N_t)_{t\in [0,T]}\big)=(\Yc_\cdot,\Zc_\cdot, \Nc_\cdot)$ in $\Sc^{2,c}\times \H^{2,c}_{\rm BMO} \times\M^{2,c}_{\rm BMO}$, for $\widetilde N_t := N_t^t-\int_0^t \partial N_r^r \d r$, $t\in[0,T]$ and $\partial Y$ is given as in \eqref{Eq:systemBSDEfq}, we have that
\begin{align}\label{eq:flow}
Y_t^t&=Y_s^s+\int_t^s \Big( f_r(r,X,Y_r^r,Z_r^r,Y_r^r,Z_r^r)-\partial Y_r^r\Big)\d r-\int_t^s {Z_r^r}^\t \d X_r- \int_t^s \d \widetilde N_r.
\end{align}

Several comments are in order:
\begin{enumerate}[label=$(\roman*)$, ref=.$(\roman*)$,wide, labelwidth=!, labelindent=0pt]
\item In the same spirit as \cite{viens2019martingale}, our version of the flow property for type-I BSVIEs requires to take into account and additional variable, namely $(\partial Y_t^t)_{t\in [0,T]}$. 
	We remark that in full generality one has to take into account the process $\widetilde N$. 
	However, the presence of $\widetilde N$ is mainly due to our formulation of the dynamic of the state variable $X$, namely, in weak formulation. 
	In fact, whenever the representation property for $(\F,\P)$--martingales in terms of stochastic integrals with respect to $X$ holds, we have that $\Nc=M=\partial M=0$, see also \Cref{rmk:repproperty}.

\item Equation \ref{eq:flow} has implications that go beyond it being a mere enunciation of the flow property for type-I extended BSVIEs. 
	Actually, in the context of time-inconsistent control problems such as the ones presented in \Cref{sec:motivation}, \cite{hernandez2020me} leveraged \eqref{eq:flow} to provide a justification of the choice of a equilibrium policies for sophisticated time-inconsistent agents.
	
\item We also emphasise that this serves as a further motivation for our approach via systems of infinite families of BSDEs such as \eqref{Eq:systemBSDEfq}.
	Indeed, it is able to handle the well-posedness of extended type-I BSVIEs, and, as a by product, it leverages \eqref{eq:flow}, the underlying flow property, to accommodate BSVIEs where the diagonal of $Z$ appears in the generator.
\end{enumerate}

\section{Proof of the Linear Quadratic case}\label{sec:prooflinearquadratic}

\begin{proof}[Proof of {\rm\Cref{Thm:wp:smalldata}}]
For $c>0$, let us introduce the mapping
\begin{align*}
\Tf:(\Bc_R, \|\cdot \|_{\Hc^c})&\longrightarrow (\Bc_R, \|\cdot \|_{\Hc^c})\\
(y,z,n,u,v,m,\partial u,\partial v,\partial m)&\longmapsto (Y,Z,N,U,V,M,\partial U,\partial V,\partial M),
\end{align*}
with $(\Yc,\Zc,\Nc,U,V,M,\partial U, \partial V, \partial M)$ given for any $s\in [0,T]$, $\P\as$ for any $t\in [0,T]$ by
\begin{align*}
\Yc_t&=\xi(T,X_{\cdot\wedge T})+\int_t^T h_r(X,\Yc_r, z_r, U_r^r, v_r^r,\partial U_r^r)\d r-\int_t^T \Zc_r^\t d X_r-\int_t^T \d  \Nc_r,\\
U_t^s&= \eta (s,X_{\cdot\wedge,T})+\int_t^T  g_r(s,X,U_r^s,v_r^s, \Yc_r, z_r) \d r-\int_t^T  {V_r^s}^\t \d X_r-\int_t^T \d  M^s_r,\\
\partial U_t^s&= \partial_s \eta (s,X_{\cdot\wedge,T})+\int_t^T   \nabla g_r(s,X,\partial U_r^s,\partial v_r^s,U_r^s,v_r^s, \Yc_r, z_r) \d r-\int_t^T  \partial {V_r^s}^\t  \d X_r-\int_t^T \d  \partial M^s_r.
\end{align*}

{\bf Step 1:} We first argue that $\Tf$ is well--defined.

\begin{enumerate}[label=$(\roman*)$, ref=.$(\roman*)$,wide, labelwidth=!, labelindent=0pt]
\item In light of \Cref{Assumption:LQgrowth}, there is $c>0$ such that $(\xi,\eta,\partial \eta, \tilde f, \tilde g, \nabla \tilde g) \in \Lc^{\infty, c}\times\big( \Lc^{\infty, 2, c}\big)^2 \times \L^{1,\infty, c}\times \big( \L^{1,\infty, 2,c}\big)^2$ and $(z,v,\partial v )\in \H^{2,c}_{{\rm BMO}} \times \Ho_{{\rm BMO}}\times \H^{2,2,c}_{{\rm BMO}}$, thus, we may use \eqref{Eq:ineqBMO} to obtain
\begin{align*}
& \E \bigg[ |\xi(T)|^2 +\bigg| \int_0^T | h_t(0,z_t,0,v_t^t,0)| \d t\bigg|^2\bigg] + \sup_{s\in [0,T]}\E\bigg[  |\eta (s)|^2 +\bigg| \int_0^T |  g_t(s,0,v^s_t,0,z_t) | \d t\bigg|^2 \bigg]\\
&\; +  \sup_{s\in [0,T]}\E\bigg[  |\partial \eta (s)|^2 +\bigg| \int_0^T |  \nabla g_t(s,0,\partial v^s_t,0,v^s_t,0,z_t) | \d t\bigg|^2 \bigg]\\
 \leq&\ \E \bigg[ |\xi(T)|^2 +3 \bigg|\int_0^T | \tilde h_t | \d t\bigg|^2+10 L_z^2 \bigg| \int_0^T | z_t |^2 \d t\bigg|^2+3L_v^2  \bigg|\int_0^T | v_t^t |^2 \d t\bigg|^2\bigg] \\
&\; + \sup_{s\in [0,T]}\E\bigg[  |\eta (s)|^2+3 \bigg|\int_0^T | \tilde g_t(s) | \d t\bigg|^2+7 L_v^2 \bigg|\int_0^T |v_t^s |^2 \d t\bigg|^2\bigg]\\
&\; + \sup_{s\in [0,T]}\E\bigg[  |\partial \eta (s)|^2+4 \bigg|\int_0^T | \nabla \tilde g_t(s) | \d t\bigg|^2+4 L_{{\rm v} }^2 \bigg| \int_0^T |\partial v_t^s |^2 \d t\bigg|^2\bigg]\\
 \leq&\ \| \xi\|_{\Lc^{\infty,c}}^2+ 3 \| \tilde h\|_{\L^{1,\infty,c}}^2+ \| \eta\|_{\Lc^{\infty,2,c}}^2+ 3 \| \tilde g\|_{\L^{1,\infty,2,c}}^2 + \|\partial  \eta\|_{\Lc^{\infty,2,c}}^2 + 4 \| \nabla \tilde g\|_{\L^{1,\infty,2,c}}^2 \\
&\;  + 20L_z^2 \|z\|_{\H^{2,c}_{\rm BMO}}^4 + 14 L_v^2 \|v\|_{\overline \H^{2,2,c}_{\rm BMO}}^4 + 8 L_{\rm  v}^2 \|\partial v\|_{\H^{2,2,c}_{\rm BMO}}^4<\infty.
\end{align*}
Therefore, by \cite[Theorem 3.2]{hernandez2020unified}, $\Tf$ defines a well--posed system of BSDEs with unique solution in the space $\Hf^2$. We recall the spaces involved in the definition of $\Hf^2$, and their corresponding norms, were introduced in \Cref{Sec:spaces}. \medskip

\item Arguing as in \cite[Lemmata 6.1 and 6.2]{hernandez2020unified}, we may use \Cref{Assumption:LQgrowth} and $v\in \Ho_{\rm BMO}$, i.e. $\partial v$ is the density with respect to the Lebesgue measure of $s\longmapsto v^s$, to obtain that $(\Uc,\Vc,\Mc)\in \S^{2}\times \H^{2}\times \M^{2}$ given by
\[ \Uc_t:=U_t^t, \; \Vc_t:=V_t^t, \; \Mc_t:=M_t^t-\int_0^t \partial M_r^r \d r,\; t\in [0,T],\]
satisfy the equation
\[
\Uc_t=\eta(T,X_{\cdot\wedge T})+\int_t^T \big( g_r(r,X,\Uc_r,v_r^r,\Yc_r,z_r)-\partial U_r^r\big) \d r -\int_t^T {\Vc_r}^\t \d X_r-\int_t^T \d \Mc_r, \; t\in [0,T],\P\as
\]
\item We show $(\Yc,\Uc)\in \Sc^{\infty,c}\times \Sc^{\infty,c}$ and $\|U\|_{\Sc^{\infty,2,c}}+\|\partial U\|_{\Sc^{\infty,2,c}}<\infty $. \medskip

To alleviate the notation we introduce
\begin{align*}
h_r&:=h_r(\Yc_r,z_r,\Uc_r,v_r^r,\partial U_r^r) ,\;   g_r:=g_r(r,\Uc_r,v_r^r,\Yc_r,z_r)- \partial U_r^r,\\
 g_r(s)& :=g_r(s,U_r^s,v_r^s,\Yc_r,z_r),\;  \nabla g_r(s):= \nabla g_r(s,\partial U_r^s,\partial v_r^s,U_r^s,v_r^s,\Yc_r,z_r).
\end{align*} 
and,
\[ \Yf:=(\Yc,\Uc,U^s,\partial U^s),\; \Zf:=(\Zc,\Vc,V^s,\partial V^s),\; \Nf:=(\Nc,\Mc,M^s,\partial M^s) ,\]
whose elements we may denote with superscripts, e.g. $\Yf^1, \Yf^2, \Yf^3,\Yf^4$ correspond to $\Yc,\Uc,U^s,\partial U^s$.\medskip

In light of \Cref{Assumption:LQgrowth}, $\d r\otimes\d \P\ae$
\begin{align}\label{Eq:thm:wd:ineq:lip0}
\begin{split}
| h_r|&\leq L_y |\Yc_r|+ L_z |\sigma^\t_r z_r|^2+  L_u  |\Uc_r|+ L_v |\sigma^\t_r v_r^r|^2+  L_{\rm u}  |\partial U_r^r|+ |\tilde h_r|,  \\
|   g_r|&\leq  L_u  |\Uc_r|  + L_v |\sigma^\t_r v_r^r|^2+ L_y |\Yc_r|+ L_z |\sigma^\t_r z_r|^2+ |\partial U_r^r|+ |\tilde g_r|,  \\
|g_r(s)| &\leq  L_u |U_r^s|+ L_v |\sigma^\t_r v_r^s|^2+ L_y |\Yc_r|+ L_z| \sigma^\t_r z_r|^2  +    |\tilde g_r(s)|,\\
|\nabla g_r(s)| &\leq  L_{\rm  u} |\partial U_r^s|+ L_{\rm v} |\sigma^\t_r \partial v_r^s|^2+L_u |U_r^s|+ L_v |\sigma^\t_r v_r^s|^2+ L_y |\Yc_r|+ L_z| \sigma^\t_r z_r|^2  +    |\nabla \tilde g_r(s)|.
\end{split}
\end{align}

Applying Meyer--It\^o's formula to $\e^{\frac{c}2  t}\big( |\Yc_t|+|\Uc_t|+|U_t^s|+|\partial U_t^s|\big)$ we obtain
\begin{align}\label{Eq:thm:wd:ito}
\begin{split}
&\e^{\frac{c}2  t}\big( |\Yc_t|+|\Uc_t|+|U_t^s|+|\partial U_t^s|\big) +\Mf_t -\Mf_T + \widehat L_T^0 \\
 =&\ \e^{\frac{c}2  T}\big( |\xi|+|\eta(T)|+|\eta(s)|+|\partial \eta(s)|\big)  \\
&\; + \int_t^T  \e^{\frac{c}2  r} \bigg( \sgn(  \Yc_r) \cdot h_r +\sgn(  \Uc_r) \cdot g_r + \sgn(  U_r^s) \cdot g_r(s)+ \sgn( \partial U_r^s) \cdot \nabla g_r(s)-\frac{c}2 \sum_{i=1}^4    |\Yf_r^i| \bigg) \d r 
\end{split}
\end{align}
where $\widehat L_t^0:=\widehat L_t^0(\Yc,\Uc,U^s,\partial U^s)$ denotes the non--decreasing and pathwise--continuous local time of the semimartingale $(\Yc,\Uc,U^s,\partial U^s)$ at $0$, see \cite[Theorem 70]{protter2005stochastic}, and, we introduced the martingale (recall that $(i)$ and $(ii)$ guarantee $(Z,\Vc,\Nc,\Mc,V,\partial V,M,\partial M)\in (\H^2)^2 \times ( \M^2 )^2\times (\H^{2,2})^2 \times (\M^{2,2})^2 $)
\[
\Mf_t:=\sum_{i=1}^4 \int_0^t \e^{\frac{c}2  r} \Zf_r^i \sgn( \Yf_r^i ) \cdot  \d X_r  +\int_0^t \e^{\frac{c}2  r-}   \sgn( \Yf_{r-})\cdot  \d \Nf_r^i,\; t\in[0,T].
\]

Again, we take conditional expectations with respect to $\Fc_t$ in \Cref{Eq:thm:wd:ito} and exploit the fact $\widehat L_T^0$ is non--decreasing. Moreover, in combination with \eqref{Eq:thm:wd:ineq:lip0} and \Cref{Lemma:estimatescontraction}, we obtain back in \eqref{Eq:thm:wd:ito} that for $C_1:=4 L_y+TL_y+ T  L_{\rm u} L_y$, $C_2:=   2L_u$, $C_3:= 2  L_u+TL_u+T L_{\rm u}L_u $, and $C_4:=   L_{\rm u}$
\begin{align*}
& \e^{\frac{c}2  t}\big( |\Yc_t|+|\Uc_t|+|U_t^s|+|\partial U_t^s|\big)  + \E_t\bigg[ \int_t^T \e^{\frac{c}2 r}  |\Yc_r| (c/2 -C_1)\d r\bigg] +  \E_t\bigg[ \int_t^T \e^{\frac{c}2 r}  |\Uc_r| (c/2 -C_2)\d r\bigg]  \\
&\;  +\sup_{s\in [0,T]} \E_t\bigg[ \int_t^T \e^{\frac{c}2 r}  |U_r^s|  (c/2 - C_3 ) \d r \bigg] +\sup_{s\in [0,T]} \E_t\bigg[ \int_t^T \e^{\frac{c}2 r} |\partial U_r^s|  (c/2 - C_4) \Big) \d r \bigg] \\
\leq &\   \E_t\big[\e^{\frac{c}2 T}\big( |\xi|+|\eta(T)|+|\eta(s)|+|\partial_s \eta(s)|\big)\big] +\E_t\bigg[ \int_t^T  \e^{\frac{c}2 r} \Big( |\tilde h_r| +|\tilde g_r|+   |\tilde g_r(s)|+   |\nabla \tilde g_r(s)| \Big) \d r \bigg]  \\
&\; + (T+ L_{\rm u} T) \Big (  \|\partial_s \eta \|_{\Lc^{\infty,2,c}}   + \|\nabla \tilde g\|_{\L^{1,\infty,2,c}} + L_{\star}   \Big(\|\partial v\|^2_{\H^{2,2,c}_{{\rm BMO}}}+ \|v\|^2_{\H^{2,2,c}_{{\rm BMO}}}+ \|z\|^2_{\H^{2,c}_{{\rm BMO}}}\Big)\Big)\\
& \;  +\E_t\bigg[   \int_t^T \e^{cr}\Big( 4 L_z  |\sigma^\t_r z_r|^2 +  2 L_v  |\sigma^\t_r v_r^r|^2+2  L_v  |\sigma^\t_r v_r^s|^2+  L_{\rm v}|\sigma^\t_r \partial v_r^s|^2\Big) \d r \bigg].
\end{align*}
where we recall the notation $L_\star=\max\{L_z,L_v,L_{\rm v}\}$. Thus, for 
\begin{align}\label{Eq:cYwelldefined}
c& \geq  2\max \{4 L_y+TL_y+ T  L_{\rm u} L_y  , 2L_u ,   2  L_u+T L_u+T L_{\rm u}L_u , L_{\rm u} \}\nonumber \\
& =\max \{8 L_y+2 T   L_y+ 2 T  L_{\rm u} L_y, 4 L_u+2 T L_u +2 T L_{\rm u} L_u  ,2 L_{\rm u} \},
\end{align}
we obtain 
\begin{align*}
\mathrm{max}\big\{ \e^{\frac{c}2 t} |\Yc_t| ,\, \e^{\frac{c}2 t} |\Uc_t| ,\,\e^{\frac{c}2 t} |U_t^s|,\,\e^{\frac{c}2 t} |\partial U_t^s| \big\} & \leq \e^{\frac{c}2 t}\big( |\Yc_t|+|\Uc_t| +|U_t^s|+|\partial U_t^s| \big)\\
&  \leq  \|\xi\|_{\Lc^{\infty,c}} +  \|\tilde h\|_{\L^{1,\infty,c}}+ 2\big(  \|\eta \|_{\Lc^{\infty,2,c}} + \|\tilde g\|_{\L^{1,\infty,2,c}} \big) \\
&\quad   +(1+T+TL_{{\rm u}} )\big( \|\partial_s \eta \|_{\Lc^{\infty,2,c}}+ \|\nabla \tilde g\|_{\L^{1,\infty,2,c}} \big)  \\
 & \quad  + (4+T+L_{\rm u} T)  L_{\star}  \Big( \|\partial v\|^2_{\H^{2,2,c}_{{\rm BMO}}}+\|v\|^2_{\overline \H^{2,2,c}_{{\rm BMO}}}+ \|z\|^2_{\H^{2,c}_{{\rm BMO}}}\Big),
\end{align*}

\item We show $(\Zc,\Vc,\Nc,\Mc)\in \big(\H^{2,c}_{{\rm BMO}}\big)^2\times \big(\M^{2,c}\big)^2$ and $\| V\|_{ \H^{2,2,c}_{{\rm BMO}}}^2+\|M\|_{{\M}^{2,2,c}}^2+\| \partial V\|_{ \H^{2,2,c}_{{\rm BMO}}}^2+\|\partial M\|_{{\M}^{2,2,c}}^2<\infty $. Applying It\^o's formula to $\e^{ct}\big( |\Yc_t|^2+ |\Uc_t|^2+|U_t^s|^2+|\partial U_t^s|^2\big)$ we obtain
\begin{align}\label{eq:eq2}
\begin{split}
&\sum_{i=1}^4 \e^{ct}  |\Yf_t^i|^2+\int_t^T \e^{cr}  |\sigma^\t_r  \Zf_r^i|^2 \d r+\int_t^T\e^{c r-} \d \Tr [\Nf^i]_r  +\widetilde \Mf_t -\widetilde  \Mf_T\\
 =&\ \e^{cT}\big( |\xi|^2+ |\eta(T)|^2+|\eta(s)|^2+|\partial_s \eta(s)|^2\big)\\
&\; + \int_t^T  \e^{c  r} \bigg( 2  \Yc_r \cdot h_r +  2 \Uc_r  \cdot  g_r +  2 U_r^s  \cdot g_r(s)+ 2 \partial U_r^s \cdot \nabla g_r(s)-c \sum_{i=1}^4    |\Yf_r^i|^2 \bigg) \d r 
\end{split}\end{align}
where for any $s\in [0,T]$ we introduced the martingale
\[
\widetilde \Mf_t:=\sum_{i=1}^4 \int_0^t \e^{c r}   \Zf_r^i\Yf_r^i  \cdot \d X_r +\int_0^t \e^{c r-}  \Yf_{r-}^i\cdot  \d \Nf_r^i,\; t\in[0,T].
\]
Indeed,  Burkholder--Davis--Gundy's inequality and the fact that $(Y,\Uc,Z,\Vc,U,\partial U,V,\partial V)\in (\S^2 )^2\times  (\H^2 )^2\times  ( \S^{2,2} )^2\times  ( \H^{2,2} )^2$, recall $(i)$ and $(ii)$, implies that there exists $C>0$ such that
\[ 
\E\bigg[ \sup_{t\in [0,T]} \bigg | \int_0^t \e^{cr} \Zf_r^i\Yf_r^i \cdot \d X_r \bigg| \bigg] \leq C \E\bigg[ \bigg|\int_0^T \e^{cr}|\Yf_r^i \sigma_r^\t \Zf_r^i|^2\d r\bigg|^{\frac1{2}}\bigg]\leq C \e^{cT} \| \Yf^i\|_{\S^2} \| \Zf^i\|_{\H^2},\; i\in \{1,...,4\},
\]
 
which guarantees each of the processes in $\widetilde \Mf$ is an uniformly integrable martingale. Thus, taking conditional expectations with respect to $\Fc_t$ in \Cref{eq:eq2}, we obtain
\begin{align*}
&\sum_{i=1}^4 \e^{ct}  |\Yf_t^i|^2+\E_t\bigg[  \int_t^T \e^{cr}  |\sigma^\t_r  \Zf_r^i|^2 \d r+\int_t^T\e^{c r-} \d \Tr [\Nf^i]_r\bigg] \\
 =&\ \E_t\Big[  \e^{cT}\big( |\xi|^2+ |\eta(T)|^2+|\eta(s)|^2+|\partial_s \eta(s)|^2\big)\Big]\\
&\; +\E_t\bigg[  \int_t^T  \e^{c  r} \bigg(   2 \Yc_r \cdot h_r +   2\Uc_r  \cdot  g_r +  2 U_r^s  \cdot g_r(s)+ 2 \partial U_r^s \cdot \nabla g_r(s)-c \sum_{i=1}^4    |\Yf_r^i|^2 \bigg) \d r \bigg],\; t\in[0,T].
\end{align*}

From $(iii)$, \Cref{Assumption:LQgrowth}\ref{Assumption:LQgrowth:2} and \ref{Assumption:LQgrowth}\ref{Assumption:LQgrowth:3}, together with Young's inequality, yield that, for any $\eps_i>0$, and defining $\widetilde C_{\eps_{1,7}}:=2 L_y+ \eps_1^{-1} 7T  L_{\rm u}^2+\eps_7^{-1} L_u^2$, $\widetilde C_{\eps_{2,8}}:=2  L_u+ \eps_2^{-1} 7T + \eps_8^{-1}  L_y^2$, $\widetilde C_{\eps_{9}}:=2  L_u+ \eps_9^{-1}  L_y^2$, as well as, $\widetilde C_{\eps_{10,11}}:=2  L_{\rm u}+ \eps_{11}^{-1} L_u^2 + \eps_{10}^{-1}  L_y^2$ we have
\begin{align*}
2 \Yc_r \cdot h_r-c |\Yc_r|^2 & \leq   (  \widetilde C_{\eps_{1,7}}-c) |\Yc_r|^2+2 \|\Yc\|_{\Sc^{\infty,c}}  \big( L_z  |\sigma^\t_r z_r|^2+L_v  |\sigma^\t_r v_r^r|^2+ |\tilde h_r|\big)+ \eps_1(7T)^{-1}  |\partial U_r^r|^2+ \eps_7  | \Uc_r|^2, \\[0,5em]
2 \Uc_r \cdot  g_r-c |\Uc_r|^2 & \leq   (  \widetilde C_{\eps_{2,8}}-c) |\Uc_r|^2+2 \|\Uc\|_{\Sc^{\infty,c}}  \big( L_z  |\sigma^\t_r z_r|^2+L_v  |\sigma^\t_r v_r^r|^2+ |\tilde g_r|\big)+ \eps_2(7T)^{-1}  |\partial U_r^r|^2+ \eps_8  | \Yc_r|^2, \\[0,5em]
2 U_r^s \cdot  g_r(s)-c |U_r^s|^2 & \leq   (  \widetilde C_{\eps_{9}}-c) |U_r^s|^2+2 \|U\|_{\Sc^{\infty,2,c}}  \big(  L_v |\sigma^\t_r v_r^s|^2+ L_z | \sigma^\t_r z_r|^2+ |\tilde g_r(s)| \big) + \eps_9  | \Yc_r|^2, \\[0,5em]
 2 \partial U_r^s \cdot \nabla g_r(s)-c |\partial U_r^s|^2&  \leq   (  \widetilde C_{\eps_{10,11}}-c ) |\partial U_r^s|^2+2\|\partial U\|_{\Sc^{\infty,c,2}}  \big(  L_{\rm v} |\sigma^\t_r \partial v_r^s|^2+L_v |\sigma^\t_r v_r^s|^2+ L_z | \sigma^\t_r z_r|^2+ |\nabla \tilde g_r(s)| \big)\\
 &\quad  + \eps_{11}       |U_r^s|^2+ \eps_{10}       |Y_r|^2.
\end{align*}

These inequalities in combination with \Cref{Lemma:estimatescontraction}, and Young's inequality, show that if we define $C^{\eps}_1:=\widetilde C_{\eps_{1,7}}+\eps_8+\eps_9+\eps_{10}+ (\eps_1+\eps_2) T  L_y^2 $, $C^{\eps}_2:=\widetilde C_{\eps_{2,8}}+\eps_7 $, $C^{\eps}_3:= \widetilde C_{\eps_{9}}+\eps_{11} + (\eps_1+\eps_2) TL_u^2 $, $  C^{\eps}_4:=\widetilde C_{\eps_{10,11}}$, then for any $\eps_i>0$, $i\in\{7,...,15\}$
\begin{align*}
&\sum_{i=1}^4 \e^{ct}  |\Yf_t^i|^2+\E_t\bigg[  \int_t^T\! \e^{cr}  |\sigma^\t_r  \Zf_r^i|^2 \d r+\int_t^T\!\e^{c r-} \d \Tr [\Nf^i]_r\bigg]+ \E_t\bigg[ \int_t^T\! \e^{cr}\big(  |\Yc_r|^2 (c-C^{\eps}_1)+|\Uc_r|^2 (c- C^{\eps}_2 )\big)\d r\bigg] \\
&\; +\sup_{s\in [0,T]}  \E_t\bigg[ \int_t^T \e^{cr}|U_r^s|^2(c-C^{\eps}_3) \d r\bigg] +\sup_{s\in [0,T]}  \E_t\bigg[ \int_t^T \e^{cr} |\partial U_r^s|^2(c-C^{\eps}_4)  \d r\bigg]\\
= & \ ( \eps_1+\eps_2)  \Big( \|\partial_s \eta \|_{\Lc^{\infty,2,c}}^2  +\|\nabla \tilde g\|^2_{\L^{1,\infty,2,c}} + 2 L_\star^2  \Big( \|\partial v\|^4_{\H^{2,2,c}_{{\rm BMO}}} + \|v\|^4_{\H^{2,2,c}_{{\rm BMO}}}+ \|z\|^4_{\H^{2,c}_{{\rm BMO}}}\Big) \Big)\\
&\;  + \E_t\Big[  \e^{cT}\big( |\xi|^2+ |\eta(T)|^2+|\eta(s)|^2+|\partial_s \eta(s)|^2\big)\Big]  +\big(\eps_{3}^{-1}+ \eps_{12}^{-1}+\eps_{13}^{-1} \big) \|\Yc\|^2_{\Sc^{\infty,c}} \\
&\;   +\big(\eps_{4}^{-1}+ \eps_{14}^{-1}+\eps_{15}^{-1} \big) \|\Uc\|^2_{\Sc^{\infty,c}} +\big(\eps_{5}^{-1}+ \eps_{16}^{-1}+\eps_{17}^{-1} \big)  \|U\|^2_{\Sc^{\infty,c,2}} +\big(\eps_{6}^{-1}+ \eps_{18}^{-1}+\eps_{19}^{-1}+\eps_{20}^{-1} \big)  \|\partial U\|^2_{\Sc^{\infty,c,2}} \\
&\; +\eps_3  \E_t\bigg[ \bigg|\int_t^T  \e^{cr}  |\tilde h_r|\d r\bigg|^2\bigg] + \eps_4 \E_t\bigg[\bigg|\int_t^T  \e^{cr} |\tilde g_r|\d r\bigg|^2 \bigg]   +  \eps_{5}\E_t\bigg[ \bigg|\int_t^T  \e^{cr} |\tilde g_r(s)|\d r\bigg|^2\bigg]+ \eps_{6}\E_t\bigg[ \bigg|\int_t^T  \e^{cr} |\nabla \tilde g_r|\d r\bigg|^2 \bigg]\\
&\; +( \eps_{12}+\eps_{14}+\eps_{16}+\eps_{18} ) L_z^2\E_t\bigg[  \bigg|\int_t^T \e^{cr} |\sigma^\t_r z_r|^2 \d r\bigg|^2\bigg]+( \eps_{13}+\eps_{15} ) L_v^2\E_t\bigg[  \bigg|\int_t^T \e^{cr} |\sigma^\t_r v_r^r|^2 \d r\bigg|^2\bigg]\\
&\;  + (\eps_{17}+\eps_{19}) L_v^2\E_t\bigg[  \bigg|\int_t^T \e^{cr} |\sigma^\t_r v_r^s|^2 \d r\bigg|^2\bigg] +\eps_{20} L_{\rm v}^2\E_t\bigg[  \bigg|\int_t^T \e^{cr} |\sigma^\t_r \partial v_r^s|^2 \d r\bigg|^2\bigg],
\end{align*}

We now let $\tau \in \Tc_{0,T}$. In light of \eqref{Eq:cYwelldefined}, for
\begin{align}\label{Eq:cZwelldefined}
\begin{split}
 c\geq \max &\   \{2 L_y+ \eps_1^{-1} 7T  L_{\rm u}^2+ \eps_1 T  L_y^2 + \eps_2 T  L_y^2+\eps_7^{-1}L_u^2 +\eps_8+\eps_9+\eps_{10}  , \    \\
 &\quad   2  L_u+ \eps_2^{-1} 7T+\eps_7 + \eps_8^{-1}  L_y^2 ,\ 2  L_u +\eps_1TL_u^2 +\eps_2 T L_u^2 + \eps_9^{-1}  L_y^2+\eps_{11}, \\
 &\quad    2  L_{\rm u}+ \eps_{11}^{-1} L_u^2 + \eps_{10}^{-1}  L_y^2 , \ 8 L_y+2 T   L_y+ 2 T  L_{\rm u} L_y, \ 4 L_u+2 T L_u +2 T L_{\rm u} L_u  \} ,
\end{split}
\end{align}
\eqref{Eq:ineqBMO} yields
\begin{align*}
&\sum_{i=1}^4 \e^{ct}  |\Yf_t^i|^2+\E_t\bigg[  \int_t^T \e^{cr}  |\sigma^\t_r  \Zf_r^i|^2 \d r+\int_t^T\e^{c r-} \d \Tr [\Nf^i]_r\bigg]\\
 =&\ I_0^\eps +2 L_\star^2  ( \eps_1+\eps_2+ \eps_{12}+\eps_{14}+\eps_{16}+\eps_{18} )  \|z\|^4_{\H^{2,c}_{{\rm BMO}}}  \\
& +2 L_\star^2 \Big(( \eps_1+\eps_2+ \eps_{13}+\eps_{15}+\eps_{17}+\eps_{19} ) \|v\|^4_{\overline \H^{2,2,c}_{{\rm BMO}}} + ( \eps_1+\eps_2+\eps_{20}) \|\partial v\|^4_{\H^{2,2,c}_{{\rm BMO}}}\Big)+\big(\eps_{3}^{-1}+ \eps_{12}^{-1}+\eps_{13}^{-1} \big) \|Y\|^2_{\Sc^{\infty,c}} \\
&   +\big(\eps_{4}^{-1}+ \eps_{14}^{-1}+\eps_{15}^{-1} \big) \|\Uc\|^2_{\Sc^{\infty,c}} +\big(\eps_{5}^{-1}+ \eps_{16}^{-1}+\eps_{17}^{-1} \big)  \|U\|^2_{\Sc^{\infty,c,2}} +\big(\eps_{6}^{-1}+ \eps_{18}^{-1}+\eps_{19}^{-1}+\eps_{20}^{-1} \big)  \|\partial U\|^2_{\Sc^{\infty,c,2}} 
\end{align*}
which in turn leads to
\begin{align}\label{Eq:thm:wd:ineq:final}
\begin{split}
&\frac{1}{10}\Big(\|\Yc\|^2_{\Sc^{\infty,c}} +\|\Uc\|^2_{\Sc^{\infty,c}} +\|U\|^2_{\Sc^{\infty,2,c}}+\|\partial U\|^2_{\Sc^{\infty,2,c}} +    \|\Zc\|_{\H^{2,c}_{ {\rm BMO}}}^2  \\
&\quad + \|V\|_{\overline \H^{2,2,c}_{{\rm BMO}}}^2+ \|\partial V\|_{\H^{2,2,c}_{{\rm BMO}}}^2 + \|\Nc\|_{{\M}^{2,c}}^2+  \|M\|_{{\M}^{2,2,c}}^2+  \|\partial M\|_{{\M}^{2,2,c}}^2
\Big)   \\
 \leq&\  \|\xi\|_{\Lc^{\infty,c}}^2+2 \|\eta \|_{\Lc^{\infty,2,c}}^2  + (1+ \eps_1+\eps_2)\|\partial_s \eta \|_{\Lc^{\infty,2,c}}^2+  \eps_3 \|  \tilde h\|^2_{\L^{1,\infty,c}} + ( \eps_4 +\eps_{5}) \|  \tilde g\|^2_{\L^{1,\infty,2,c}}  \\
&\; + ( \eps_1+\eps_2+ \eps_{6}) \| \nabla  \tilde g\|^2_{\L^{1,\infty,2,c}}  +2 L_\star^2  ( \eps_1+\eps_2+ \eps_{12}+\eps_{14}+\eps_{16}+\eps_{18} )  \|z\|^4_{\H^{2,c}_{{\rm BMO}}}  \\
&\; +2 L_\star^2 ( \eps_1+\eps_2+ \eps_{13}+\eps_{15}+\eps_{17}+\eps_{19} ) \|v\|^4_{\overline \H^{2,2,c}_{{\rm BMO}}} +2 L_\star^2 ( \eps_1+\eps_2+\eps_{20}) \|\partial v\|^4_{\H^{2,2,c}_{{\rm BMO}}} \\
&\;   +\big(\eps_{3}^{-1}+ \eps_{12}^{-1}+\eps_{13}^{-1} \big) \|Y\|^2_{\Sc^{\infty,c}} +\big(\eps_{4}^{-1}+ \eps_{14}^{-1}+\eps_{15}^{-1} \big) \|\Uc\|^2_{\Sc^{\infty,c}} +\big(\eps_{5}^{-1}+ \eps_{16}^{-1}+\eps_{17}^{-1} \big)  \|U\|^2_{\Sc^{\infty,c,2}} \\
&\; +\big(\eps_{6}^{-1}+ \eps_{18}^{-1}+\eps_{19}^{-1}+\eps_{20}^{-1} \big)  \|\partial U\|^2_{\Sc^{\infty,c,2}}
 \end{split}
\end{align}

From \eqref{Eq:thm:wd:ineq:final} we conclude $(Z,N)\in \H^{2,c}_{{\rm BMO}}\times  {\M}^{2,c}$, $\| V\|_{\overline \H^{2,2,c}_{{\rm BMO}}}^2+\| \partial V\|_{\H^{2,2,c}_{{\rm BMO}}}^2+ \|M\|_{{\M}^{2,2,c}}^2+ \|\partial M\|_{{\M}^{2,2,c}}^2<\infty$.\medskip

At this point, we can highlight a crucial step in this approach. It is clear from \eqref{Eq:thm:wd:ineq:final} that the norm of $\Tf$ does not have a linear growth in the norm of the input. In the following, we will see that choosing the data of the system small enough and localising $\Tf$ will bring us back to the linear growth scenario. For this, we observe that if we define
\[
C_{\eps}:= \min\big \{  1-10(\eps_{3}^{-1}+\eps_{12}^{-1}+\eps_{13}^{-1}) ,\; 1-10(\eps_{4}^{-1}+\eps_{14}^{-1}+\eps_{15}^{-1}) ,\;1-10(\eps_{5}^{-1}+\eps_{16}^{-1}+\eps_{17}^{-1}) ,\;1-10(\eps_{6}^{-1}+\eps_{18}^{-1}+\eps_{19}^{-1}+\eps_{20}^{-1})\big \}  ,
\]
\[
 \{ 1-10(\eps_{3}^{-1}+\eps_{12}^{-1}+\eps_{13}^{-1}) ,1-10(\eps_{4}^{-1}+\eps_{14}^{-1}+\eps_{15}^{-1}) ,1-10(\eps_{5}^{-1}+\eps_{16}^{-1}+\eps_{17}^{-1}) ,1-10(\eps_{6}^{-1}+\eps_{18}^{-1}+\eps_{19}^{-1}+\eps_{20}^{-1})\} \subseteq (0,1]^4
\]
and for some $\gamma\in(0,\infty)$, $I_0^\eps  \leq \gamma R^2/10$, we obtain back in \eqref{Eq:thm:wd:ineq:final}
\begin{align*}
& \|(Y,Z,N,U,V,M,\partial U,\partial V,\partial M)\|^2_{\Hc^{c}}\\
  \leq &\  C_{\eps}^{-1} \Big( 10 I_0^\eps  +20L_\star^2  \big(  ( \eps_1+\eps_2+ \eps_{12}+\eps_{14}+\eps_{16}+\eps_{18} )  \|z\|^4_{\H^{2,c}_{{\rm BMO}}} + ( \eps_1+\eps_2+ \eps_{13}+\eps_{15}+\eps_{17}+\eps_{19} ) \|v\|^4_{\overline \H^{2,2,c}_{{\rm BMO}}}\\
 &\hspace*{3em}  + ( \eps_1+\eps_2+\eps_{20}) \|\partial v\|^4_{\H^{2,2,c}_{{\rm BMO}}}\big)\Big)\\
 \leq  & \ C_{\eps}^{-1} R^2  \Big( \gamma  +20L_\star^2  \big(  ( \eps_1+\eps_2+ \eps_{12}+\eps_{14}+\eps_{16}+\eps_{18} )  \|z\|^2_{\H^{2,c}_{{\rm BMO}}} + ( \eps_1+\eps_2+ \eps_{13}+\eps_{15}+\eps_{17}+\eps_{19} ) \|v\|^2_{\overline \H^{2,2,c}_{{\rm BMO}}}\\
 &\hspace*{4em}  + ( \eps_1+\eps_2+\eps_{20}) \|\partial v\|^2_{\H^{2,2,c}_{{\rm BMO}}}\big)\Big)\\
  \leq  &\   C_{\eps}^{-1} R^2  \bigg(\gamma  +20 L_\star^2 R^2   \bigg( \eps_1+\eps_2+\sum_{i=12}^{20} \eps_{i}  \bigg)   \bigg) 
\end{align*}

Therefore, to obtain $\Tf(\Bc_R)\subseteq \Bc_R$, that is to say that the image under $\Tf$ of the ball of radius $R$ is contained in the ball of radius $R$, it is necessary to find $R^2$ such that the term in parentheses above is less or equal than $C_{\eps}$, i.e.
\begin{align}\label{Eq:thm:wd:rineq}
 R^2 \leq \frac1{2 0 L_\star^2}\frac{ C_{\eps} -  \gamma }{  \eps_1+\eps_2+ \sum_{i=12}^{20} \eps_{i} }
\end{align}

Clearly, there are many choices of $\eps$'s so that the above holds. Among such, we wish to choose $\gamma$, $\eps_i$, so that the expression to the right in \eqref{Eq:thm:wd:rineq} is maximal. In light of \cref{Lemma:upperbound}, we have that $ \|(Y,Z,N,U,V,M)\|_{\Hc^{c}} \leq  R$ provided that
\begin{align}\label{Eq:Rwelldefined}
 R^2 < \frac{1}{2^5\cdot 3\cdot 5^2\cdot 7\cdot L^2_\star}
\end{align}

\item Lastly, we are left to argue $(U,V,M,\partial U,\partial V,\partial M)\in \Sc^{\infty,2,c}\times \H^{2,2,c}_{{\rm BMO}}\times {\M}^{2,2,c}\times \Sc^{\infty,2,c}\times \H^{2,2,c}_{{\rm BMO}}\times {\M}^{2,2,c}$. The argument for $(U,V,M)$ is analogous to that of $(\partial U,\partial V,\partial M)$, thus we argue the continuity of the applications $([0,T],\Bc([0,T])) \longrightarrow (\Sc^{\infty,c},\|\cdot \|_{ \Sc^{\infty,c}})\big($resp. $(\H_{{\rm BMO}}^{2,c},\|\cdot \|_{\H_{{\rm BMO}}^{2,c}}),\, ({\M}^{2,c},\|\cdot \|_{{\M}^{2,c}} )\big)  : s \longmapsto \varphi^s $ for $\varphi=U^s\, ($resp. $V^s, M^s).$\medskip

Recall $\rho_{g}$ denotes the modulus of continuity of $g$. Let $(s_n)_n\subseteq [0,T], s_n \xrightarrow{ n \to \infty} s_0\in [0,T]$ and for $\varphi\in \{ U,V,M,u,v,\eta\}$, let $\Delta \varphi^n:= \varphi^{s_n}-\varphi^{s_0}$. Applying It\=o's formula to $\e^{ct}   |\Delta U_t^n|^2$, proceeding as in \textbf{Step 1} $(iv)$ we obtain
\begin{align*}
 \|\Delta U^n \|_{\Sc^{\infty,c}}^2  +  \|\Delta V^n \|_{\H^{2,c}_{{\rm BMO}}}^2 +\| \Delta M^n\|_{{\M}^{2,c}}^2 \leq 4\Big( \|\Delta \eta^n \|_{\Lc^{\infty,c}}^2 +  4 L_{v}^2 \|\Delta v^n\|_{\H^{2,c}_{{\rm BMO}}}^4 + \rho_{g}(|s_n-s_0|)^2 \Big).
\end{align*}

We conclude, $\Tf(\Bc_R)\subseteq \Bc_R$ for all $R$ satisfying \eqref{Eq:Rwelldefined}.\medskip

\end{enumerate}

{\bf Step 2:} We now argue that $\Tf$ is a contraction in $\Bc_R\subseteq \Hc$ for the norm $\| \cdot \|_{\Hc^c}$.\medskip

Let $(Y^i,Z^i,N^i,U^i,V^i,M^i,,U^i,V^i,M^i):=\Tf(y^i,z^i,n^i,u^i,v^i,m^i)$ for
$(y^i,z^i,n^i,u^i,v^i,m^i,\partial u^i,\partial v^i,\partial  m^i)\in \Bc_R$, $i=1,2$.\medskip

For $\varphi\in \{y,z,n,u,v,m,\partial u,\partial v,\partial m,\Yc, \Zc,\Nc,\Uc, \Vc,\Mc,U, V,M,\partial U, \partial V,\partial M\},$ we denote $\delta \varphi:= \varphi^1-\varphi^2$ and
\begin{align*}
  \delta h_t&:=h_t(Y^1_t,z^1_t,\Uc_t^{1}, v_t^{1,t},\partial U_t^{1,t})-h_t(Y^2_t,z^2_t,\Uc_t^{2}, v_t^{2,t},\partial U_t^{2,t}),\\
   \delta   g_t&:=g_t(t,\Uc^{1}_t,v^{1,t}_t,Y_t^1,z_t^1)-\partial U_t^{1,t} -g_t(t,\Uc^{2}_t,v^{2,t}_t,Y_t^2,z_t^2)+\partial U_t^{2,t},\\ 
    \delta  g_t(s)&:=g_t(s,U^{1,s}_t,v^{1,s}_t,Y_t^1,z_t^1)-g_t(s,U^{2,s}_t,v^{2,s}_t,Y_t^2,z_t^2),\\
      \delta \nabla  g_t(s)&:=\nabla g_t(s,\partial U^{1,s}_t,\partial v^{1,s}_t,U^{1,s}_t,v^{1,s}_t,Y_t^1,z_t^1)-g_t(s,\partial U^{2,s}_t,\partial v^{2,s}_t,U^{2,s}_t,v^{2,s}_t,Y_t^2,z_t^2);\\
   \delta \tilde h_t&:=h_t(Y^2_t,z^1_t,\Uc_t^{2}, v_t^{1,t},\partial U_t^{2,t})-h_t(Y^2_t,z^2_t,\Uc_t^{2}, v_t^{2,t},\partial U_t^{2,t}),\\
 \delta \tilde g_t&:=g_t(t,\Uc^{2}_t,v^{1,t}_t,Y_t^2,z_t^1)-\partial U_t^{2,t} -g_t(t,\Uc^{2}_t,v^{2,t}_t,Y_t^2,z_t^2)+\partial U_t^{2,t},\\
 \delta \tilde g_t(s)&:=g_t(s,U^{2,s}_t,v^{1,s}_t,Y_t^2,z_t^1)-g_t(s,U^{2,s}_t,v^{2,s}_t,Y_t^2,z_t^2),\\
 \delta \nabla \tilde g_t(s)&:=\nabla g_t(s,\partial U^{2,s}_t,\partial v^{1,s}_t,U^{2,s}_t,v^{1,s}_t,Y_t^2,z_t^1)-g_t(s,\partial U^{2,s}_t,\partial v^{2,s}_t,U^{2,s}_t,v^{2,s}_t,Y_t^2,z_t^2).
 \end{align*}

Applying It\^o's formula to $\e^{ct}\big( |\delta Y_t|^2+ |\delta \Uc_t|^2+|\delta U_t^s|^2+|\delta \partial U_t^s|^2\big)$, we obtain that for any $t\in[0,T]$
\begin{align*}
&\sum_{i=1}^4 \e^{ct}  |\delta \Yf_t^i|^2+\int_t^T \e^{cr}  |\sigma^\t_r  \delta \Zf_r^i|^2 \d r+\int_t^T\e^{c r-} \d \Tr [\delta \Nf^i]_r  +\delta \widetilde  \Mf_t -\delta \widetilde  \Mf_T\\
 =&\  \int_t^T  \e^{c  r} \bigg( 2  \delta \Yc_r \cdot \delta h_r +  2 \delta \Uc_r  \cdot \delta  g_r +  2 \delta U_r^s  \cdot \delta g_r(s)+ 2 \delta \partial U_r^s \cdot \delta \nabla g_r(s)-c \sum_{i=1}^4    |\delta \Yf_r^i|^2 \bigg) \d r \\
 \leq &\ \int_t^T  \e^{c  r} \bigg(2 | \delta \Yc_r|  \big( L_y | \delta \Yc_r|+L_u |\delta \Uc_r|+L_{\rm u} |\delta \partial U_r^r|+|\delta  \tilde h_r|\big) -c |\delta \Yc_r|^2\\
& \hspace{4em}+2 | \delta \Uc_r|  \big( L_u |\delta \Uc_r|+L_y | \delta \Yc_r|+ |\delta \partial U_r^r|+|\delta  \tilde g_r| \big)-c |\delta \Uc_r|^2\\
&\hspace{4em }+2 | \delta U_r^s|  \big( L_u |\delta U_r^s|+L_y | \delta \Yc_r|+|\delta  \tilde g_r(s)| \big)-c |\delta U_r^s|^2\\
&\hspace{4em } +2 | \delta \partial U_r^s|  \big( L_{\rm u} |\delta \partial U_r^s|+L_u |\delta U_r^s|+L_y | \delta \Yc_r|+|\delta  \nabla \tilde g_r(s)| \big)-c |\delta \partial U_r^s|^2\bigg) \d r
\end{align*}

where $\delta \widetilde \Mf$ denotes the corresponding martingale term. Let $\tau \in \Tc_{0,T}$, as in \Cref{Lemma:estimatescontraction} we obtain for $c>2L_u$
\begin{align}\label{Eq:ineqdeltaUtt}
\begin{split}
\E_\tau\bigg[ \int_\tau^T \frac{ \e^{cr}}{3   T}  |\delta \partial U_r^r|^2\d r\bigg] &\leq    \sup_{s\in [0,T]}  \es_{\tau \in \Tc_{0,T}}  \bigg|  \E_\tau \bigg[  \int_\tau^T \e^{c r} |\delta \nabla g_r(s,0,\partial v_r^s,0,v_r^s,0,z_r)|\d r \bigg] \bigg|^2 \\
&\;  + T L_y^2  \E_\tau \bigg[   \int_\tau ^T \e^{cr}|\delta Y_r|^2\d r\bigg] + T L_u^2 \sup_{s\in [0,T]}  \E_\tau \bigg[   \int_\tau ^T \e^{cr}|\delta U_r^s|^2\d r\bigg].
\end{split}
\end{align}

We now take conditional expectation with respect to $\Fc_\tau$ in the expression above. In addition, we use \Cref{Assumption:LQgrowth} in combination with \eqref{Eq:ineqdeltaUtt}, exactly as in {\bf Step 1} $(iv)$. We then obtain from Young's inequality that for any $\tilde \eps_i\in (0,\infty)$, $i\in \{1,...,11\}$,
and 
\begin{align}\label{Eq:c:contraction1}
\begin{split}
c\geq  \max&\  \{2L_y+\tilde \eps_1^{-1} 3T L_{\rm u}^2+ \tilde \eps_8+ \tilde \eps_9+ \tilde \eps_{10}+ \tilde \eps_7^{-1}L_u^2+( \tilde \eps_1+ \tilde \eps_2)TL_y^2,\; 2L_u+ \tilde \eps_7 +3T\tilde \eps_2^{-1}+\tilde\eps_8^{-1} L_y^2,\\
&\quad 2 L_u +\eps_9^{-1} L_y^2 \tilde+\tilde\eps_{11}+( \tilde \eps_1+ \tilde \eps_2)TL_u^2,\; 2L_{\rm u} +\tilde  \eps_{11}^{-1} L_u^2+\tilde\eps_{10}^{-1}L_y^2 \}
\end{split}
\end{align}

it follows that
\begin{align}\label{Eq:contractionIto}\begin{split}
& \sum_{i=1}^4 \e^{ct}  |\delta \Yf_t^i|^2+\E_\tau\bigg[\int_t^T \e^{cr}  |\sigma^\t_r  \delta \Zf_r^i|^2 \d r+\int_t^T\e^{c r-} \d \Tr [\delta \Nf^i]_r  \bigg] \\
  \leq &\  \tilde \eps_3^{-1} \|\delta Y\|^2_{\Sc^{\infty,c}}+ \tilde \eps_4^{-1} \|\delta \Uc\|^2_{\Sc^{\infty,2,c}} + \tilde \eps_{5}^{-1} \|\delta U\|^2_{\Sc^{\infty,2,c}} + \tilde \eps_{6}^{-1} \|\delta \partial U\|^2_{\Sc^{\infty,2,c}} \\
&\; + ( \tilde \eps_1+\tilde \eps_2)  \sup_{s\in [0,T]}  \es_{\tau \in \Tc_{0,T}}  \bigg|  \E_\tau \bigg[  \int_\tau^T \e^{c r} |\delta \nabla g_r(s,0,\partial v_r^s,0,v_r^s,0,z_r)|\d r \bigg] \bigg|^2\\
&\; +  \tilde \eps_3 \es_{\tau \in \Tc_{0,T}}  \bigg|  \E_\tau \bigg[  \int_\tau^T \e^{c r} |\delta \tilde h_t |\d r \bigg] \bigg|^2+ \tilde \eps_4     \es_{\tau \in \Tc_{0,T}}  \bigg|  \E_\tau \bigg[  \int_\tau^T \e^{c r} |\delta \tilde g_t|\d r \bigg] \bigg|^2\\
&\; + \tilde \eps_{5}  \sup_{s\in [0,T]}  \es_{\tau \in \Tc_{0,T}}  \bigg|  \E_\tau \bigg[  \int_\tau^T \e^{c r} |\delta \tilde g_t(s)|\d r \bigg] \bigg|^2 + \tilde \eps_{6} \sup_{s\in [0,T]}  \es_{\tau \in \Tc_{0,T}}  \bigg|  \E_\tau \bigg[  \int_\tau^T \e^{c r} |\delta \nabla \tilde g_t(s)|\d r \bigg] \bigg|^2.
\end{split}
\end{align}

We now estimate the terms on the right side of \eqref{Eq:contractionIto}. Note that in light of \Cref{Assumption:LQgrowth}\ref{Assumption:LQgrowth:3} we have

\begin{align*}
& \max\bigg\{ \bigg|  \E_\tau \bigg[  \int_\tau^T \e^{c r} |\delta \nabla g_r(s,0,\partial v_r^s,0,v_r^s,0,z_r)|\d r \bigg] \bigg|^2,\;  \bigg|  \E_\tau \bigg[  \int_\tau^T \e^{c r} |\delta \nabla \tilde g_t(s)|\d r \bigg] \bigg|^2\bigg\} \\
\leq &\ \bigg|  \E_\tau \bigg[  \int_\tau^T \e^{c r} \Big( L_{\rm v} |\sigma^\t_r \delta \partial v^s_r|\big( |\sigma^\t_r \partial v_r^1|+|\sigma^\t_r \partial v_r^2|\big)+ L_v|\sigma^\t_r \delta v^s_r|\big( |\sigma^\t_r v_r^1|+|\sigma^\t_r v_r^2|\big) \\
&\hspace{6em}+ L_z|\sigma^\t_r \delta z_r|\big( |\sigma^\t_r  z_r^1|+|\sigma^\t_r z_r^2|\big)\Big) 	\d r  \bigg]  \bigg|^2 \\
\leq &\ 3 L_\star^2   \E_\tau \bigg[  \int_\tau^T \e^{c r}   |\sigma^\t_r \delta \partial v^s_r|^2\d r \bigg] \E_\tau \bigg[  \int_\tau^T \e^{c r} \big( |\sigma^\t_r \partial v_r^1|+|\sigma^\t_r \partial v_r^2|\big)^2\d r\bigg]\\
& + 3 L_\star^2 \E_\tau \bigg[  \int_\tau^T \e^{c r}   |\sigma^\t_r \delta  v^s_r|^2\d r \bigg] \E_\tau \bigg[  \int_\tau^T \e^{c r} \big( |\sigma^\t_r   v_r^1|+|\sigma^\t_r   v_r^2|\big)^2\d r\bigg]\\
&+  3 L_\star^2  \E_\tau \bigg[  \int_\tau^T \e^{c r}   |\sigma^\t_r \delta z_r|^2\d r \bigg] \E_\tau \bigg[  \int_\tau^T \e^{c r} \big( |\sigma^\t_r z_r^1|+|\sigma^\t_r z_r^2|\big)^2\d r\bigg]\\
\leq&\  6 L_\star^2 R^2\bigg(  \E_\tau \bigg[  \int_\tau^T \e^{c r}   |\sigma^\t_r \delta \partial v^s_r|^2\d r \bigg]+\E_\tau \bigg[  \int_\tau^T \e^{c r}   |\sigma^\t_r \delta  v^s_r|^2\d r \bigg]+\E_\tau \bigg[  \int_\tau^T \e^{c r}   |\sigma^\t_r \delta z_r|^2\d r \bigg]\bigg)\\
\leq&\  6 L_\star^2 R^2\Big( \| \delta \partial v\|_{\H^{2,2,c}_{{\rm BMO}}}^2 + \| \delta v\|_{\overline \H^{2,2,c}_{{\rm BMO}}}^2 +\| \delta z\|_{\H^{2,c}_{{\rm BMO}}}^2\Big)
\end{align*}

where in the second inequality we used \eqref{Eq:ineqsquare} and Cauchy--Schwartz's inequality. Similarly
\begin{align*}
&\max\bigg\{ \bigg|\E_\tau \bigg[  \int_\tau^T \e^{c r} |\delta\tilde  h_r|\d r \bigg]  \bigg|^2 , \bigg|\E_\tau \bigg[  \int_\tau^T \e^{c r} |\delta \tilde g_r(s)|\d r \bigg]  \bigg|^2, \bigg|\E_\tau \bigg[  \int_\tau^T \e^{c r} |\delta g_r|\d r \bigg]  \bigg|^2\bigg\}\leq 4 L_{\star}^2 R^2 \Big( \| \delta z\|_{\H^{2,c}_{{\rm BMO}}}^2+ \| \delta v \|_{\overline \H^{2,c}_{{\rm BMO}}}^2\Big)
\end{align*}
Overall, we obtain back in \eqref{Eq:contractionIto} that 
\begin{align*}
& \sum_{i=1}^4 \e^{ct}  |\delta \Yf_t^i|^2+\E_\tau\bigg[\int_t^T \e^{cr}  |\sigma^\t_r  \delta \Zf_r^i|^2 \d r+\int_t^T\e^{c r-} \d \Tr [\delta \Nf^i]_r  \bigg] \\
  \leq &\  \tilde \eps_3^{-1} \|\delta Y\|^2_{\Sc^{\infty,c}}+ \tilde \eps_4^{-1} \|\delta \Uc\|^2_{\Sc^{\infty,2,c}} + \tilde \eps_{5}^{-1} \|\delta U\|^2_{\Sc^{\infty,2,c}} + \tilde \eps_{6}^{-1} \|\delta \partial U\|^2_{\Sc^{\infty,2,c}} \\
&\; +     6 L_\star^2 R^2( \tilde \eps_1+ \tilde\eps_2+\tilde\eps_{6}) \Big( \| \delta \partial v\|_{\H^{2,2,c}_{{\rm BMO}}}^2 + \| \delta v\|_{\overline \H^{2,2,c}_{{\rm BMO}}}^2 +\| \delta z\|_{\H^{2,c}_{{\rm BMO}}}^2\Big)\\
& \; +4 L_{\star}^2 R^2( \tilde \eps_3+\tilde\eps_4+\tilde\eps_{5}) \  \Big( \| \delta z\|_{\H^{2,c}_{{\rm BMO}}}^2+ \| \delta v \|_{\overline \H^{2,c}_{{\rm BMO}}}^2\Big)
\end{align*}
If we define, for $\tilde \eps_i>10, i\in \{3,4,5,6\}$, $ C_{\tilde \eps}:= {\rm min}\big\{ 1-10/ \tilde \eps_{3},\; 1-10/ \tilde \eps_{4 },\;1-10/ \tilde \eps_{5},\;1-10/ \tilde \eps_{6}\big\}$, we deduce,  
\begin{align}\label{Eq:thm:cont:final}
& \|(\delta Y,\delta Z, \delta N, \delta U, \delta V, \delta M ,\delta \partial U, \delta \partial V, \delta \partial M)\|_{\Hc^c}^2\nonumber \\
\leq&\  10 C_{\tilde \eps}^{-1} L_{\star }^2 R^2 \Big(6 ( \tilde \eps_1+ \tilde\eps_2+\tilde\eps_{6}) \Big( \| \delta \partial v\|_{\H^{2,2,c}_{{\rm BMO}}}^2 + \| \delta v\|_{\overline \H^{2,2,c}_{{\rm BMO}}}^2 +\| \delta z\|_{\H^{2,c}_{{\rm BMO}}}^2\Big)  +  4 ( \tilde \eps_3+\tilde\eps_4+\tilde\eps_{5}) \  \Big( \| \delta z\|_{\H^{2,c}_{{\rm BMO}}}^2+ \| \delta v \|_{\overline \H^{2,c}_{{\rm BMO}}}^2\Big)\nonumber \\
\leq &\ 20 C_{\tilde \eps}^{-1} L_{\star}^2 R^2(3 \tilde \eps_1 + 3 \tilde \eps_2 +2 \tilde \eps_3+2 \tilde \eps_4+2 \tilde \eps_{5}+3 \tilde \eps_{6}) \Big(\| \delta z\|_{\H^{2,c}_{{\rm BMO}}}^2 +    \| \delta v\|_{\overline \H^{2,2,c}_{{\rm BMO}}}^2+  \| \delta \partial v\|_{\H^{2,2,c}_{{\rm BMO}}}^2\Big).
\end{align}

By minimising the previous upper bound for $\tilde \eps_1$ and $\tilde\eps_2$ fixed, see \Cref{Lemma:Contractionbound}, and in light of \eqref{Eq:Rwelldefined} and \eqref{Eq:c:contraction1}, we find that letting
\begin{align*}
R^2 <&\  \frac{1}{2^5\cdot 3\cdot 5^2\cdot 7 \cdot L_{\star}^2},   \text{ and }\\
c\geq&\  \max  \{2L_y+\tilde \eps_1^{-1} 3T L_{\rm u}^2+ \tilde \eps_8+ \tilde \eps_9+ \tilde \eps_{10}+ \tilde \eps_7^{-1}L_u^2+( \tilde \eps_1+ \tilde \eps_2)TL_y^2,\; 2L_u+ \tilde \eps_7 +3T\tilde \eps_2^{-1}+\tilde\eps_8^{-1} L_y^2,\\
&\quad 2 L_u +\eps_9^{-1} L_y^2 \tilde+\tilde\eps_{11}+( \tilde \eps_1+ \tilde \eps_2)TL_u^2,\; 2L_{\rm u} +\tilde  \eps_{11}^{-1} L_u^2+\tilde\eps_{10}^{-1}L_y^2 \}
\end{align*}
we have that
\begin{align*}
& \|(\delta Y,\delta Z, \delta N, \delta U, \delta V, \delta M ,\delta \partial U, \delta \partial V, \delta \partial M)\|_{\Hc^c}^2\\
 <&\  \frac{20}{2^3\cdot  3\cdot 7\cdot 10^2}3(\sqrt{30+(\tilde\eps_1+\tilde\eps_2)}+\sqrt{30})^2   \Big(\| \delta z\|_{\H^{2,c}_{{\rm BMO}}}^2 +    \| \delta v\|_{\overline \H^{2,2,c}_{{\rm BMO}}}^2+  \| \delta v\|_{\H^{2,2,c}_{{\rm BMO}}}^2\Big) \\
 =&\  \frac{ (\sqrt{30+(\tilde\eps_1+\tilde\eps_2)}+\sqrt{30})^2}{2^2\cdot 7\cdot 10}\Big(\| \delta z\|_{\H^{2,c}_{{\rm BMO}}}^2 +    \| \delta v\|_{\overline \H^{2,2,c}_{{\rm BMO}}}^2+  \| \delta v\|_{\H^{2,2,c}_{{\rm BMO}}}^2\Big).
\end{align*}

Thus, letting choosing $(\sqrt{30+(\tilde\eps_1+\tilde\eps_2)}+\sqrt{30})^2 \leq 2^2\cdot 7\cdot 10$, $\Tf$ is contractive, i.e
\begin{align*}
 \|(\delta Y,\delta Z, \delta N, \delta U, \delta V, \delta M ,\delta \partial U, \delta \partial V, \delta \partial M)\|_{\Hc^c}^2  <  \| \delta z\|_{\H^{2,c}_{{\rm BMO}}}^2 +    \| \delta v\|_{\overline \H^{2,2,c}_{{\rm BMO}}}^2+  \| \delta \partial v\|_{\H^{2,2,c}_{{\rm BMO}}}^2.
\end{align*}

{\bf Step 3:} We consolidate our results. To begin with, we collect the constraints of the weight of the norms. In light of
\eqref{Eq:cZwelldefined} and \eqref{Eq:c:contraction1}, $c$ must satisfy 
\begin{align}\label{eq:cfinal}
 c\geq \max &\   \{2 L_y+ \eps_1^{-1} 7T  L_{\rm u}^2+ \eps_1 T  L_y^2 + \eps_2 T  L_y^2+\eps_7^{-1}L_u^2 +\eps_8+\eps_9+\eps_{10}   , \    2  L_u+ \eps_2^{-1} 7T+\eps_7 + \eps_8^{-1}  L_y^2  ,  \nonumber \\
 &\quad   2  L_u +\eps_1TL_u^2 +\eps_2 T L_u^2 + \eps_9^{-1}  L_y^2+\eps_{11} , \   2  L_{\rm u}+ \eps_{11}^{-1} L_u^2 + \eps_{10}^{-1}  L_y^2  ,\nonumber\\
&\quad 2L_y+\tilde \eps_1^{-1} 3T L_{\rm u}^2+ \tilde \eps_8+ \tilde \eps_9+ \tilde \eps_{10}+ \tilde \eps_7^{-1}L_u^2+( \tilde \eps_1+ \tilde \eps_2)TL_y^2,\; 2L_u+ \tilde \eps_7 +3T\tilde \eps_2^{-1}+\tilde\eps_8^{-1} L_y^2,\nonumber \\
&\quad 2 L_u +\eps_9^{-1} L_y^2 \tilde+\tilde\eps_{11}+( \tilde \eps_1+ \tilde \eps_2)TL_u^2  ,\; 2L_{\rm u} +\tilde  \eps_{11}^{-1} L_u^2+\tilde\eps_{10}^{-1}L_y^2 ,\nonumber \\
&\quad 8 L_y+2 T   L_y+ 2 T  L_{\rm u} L_y, \ 4 L_u+2 T^2 L_u +2 T^2 L_{\rm u} L_u  \},\nonumber\\
\begin{split}
= \max &\   \{2 L_y+ \eps_1^{-1} 7T  L_{\rm u}^2+ \eps_1 T  L_y^2 + \eps_2 T  L_y^2 +\eps_7^{-1}L_u^2+\eps_8+\eps_9+\eps_{10}  , \   2  L_u+ \eps_2^{-1} 7T +\eps_7 + \eps_8^{-1}  L_y^2 ,  \\
 &\quad   2  L_u +\eps_1TL_u^2 +\eps_2 T L_u^2 + \eps_9^{-1}  L_y^2+\eps_{11}, \   2  L_{\rm u}+ \eps_{10}^{-1}  L_y^2 + \eps_{11}^{-1} L_u^2  ,\\
&\quad 8 L_y+2 T   L_y+ 2 T  L_{\rm u} L_y, \ 4 L_u+2 T L_u +2 T L_{\rm u} L_u  \}
\end{split}
\end{align}
where the equality follows from the choice $\eps_i=\tilde\eps_i, i\in \{1,2,7,...,11\}$.\medskip

All together we find that given $\gamma\in(0,\infty)$, $\eps_i\in(0,\infty)$, $i\in\{1,...,11\}$, $c\in (0,\infty)$,  such that $ \eps_1+\eps_2 \leq (2\sqrt{70}-\sqrt{30})^2-30$, $\Tf$ is a well--defined contraction in $\Bc_{ R}\subseteq \Hc^c$ for the norm $\| \cdot \|_{\Hc^c}$ provided: $(i)$ $\gamma $, $\eps_i$, $i\in \{1,...,6\}$, and the data of the problem satisfy $I_0^\eps  \leq \gamma R^2/10$; $(ii)$ $c$ satisfies \eqref{eq:cfinal}.
\end{proof}

\begin{proof}[Proof of {\rm \Cref{Thm:solutionBMOnorm}}]
We show for $\|\partial M \|_{{\rm BMO}^{2,2,c}}^2<\infty$, the argument for $(\Nc,M)$ being completely analogous. Without lost of generality we assume $c=0$, see \Cref{remark:defspaces}. In light of \Cref{Assumption:LQgrowth}, we have that $\d r \otimes \d \P\ae$
\begin{align*}
\begin{split}
 | \nabla g_r(s,\partial U_r^s,\partial V_r^s,U_r^s,V_r^s,\Yc_r,\Zc_r) |\leq & L_{\rm u} |U_r^s|  +L_{\rm v} |\sigma^\t_r  V_r^s|^2+L_u |U_r^s|  +L_v |\sigma^\t_r  V_r^s|^2+L_y |\Yc_r|+L_z |\sigma^\t_r \Zc_r|^2+  |\nabla \tilde g_r(s)|.
\end{split}
\end{align*}
Let $\tau\in \Tc_{0,T}$. We now note, recall $\partial V\in  \H^{2,2}_{\rm BMO}$
\[ \E_\tau\bigg[ \bigg(\int_{\tau-}^T {\partial V_r^s}^\t \d X_r\bigg)^2 \bigg]=\E_\tau\bigg[ \bigg(\int_{\tau}^T {\partial V_r^s}^\t \d X_r\bigg)^2 \bigg]=\E_\tau\bigg[\int_{\tau}^T |\sigma_r^\t \partial V_r^s|^2 \d  r \bigg] . \]
All together, it follows from \eqref{Eq:systemBSDEq} and Jensen's inequality that
\begin{align*}
\E_\tau\bigg[\bigg| \int_{\tau-}^T  \d \partial M_r^s\bigg|^2\bigg]  & \leq  10\E_\tau\bigg[ |\partial_s\eta(s)|^2+ \bigg| \int_{\tau-}^T |\nabla \tilde g_r(s)|\d r\bigg|^2 +|\partial U_{\tau-}^s|^2+ T \int_{\tau-}^T L_{\rm u} |\partial  U_r^s|^2 + L_u | U_r^s|^2 +L_y |\Yc_r|^2 \d r \\
&\quad +\bigg| \int_{\tau-}^T L_{\rm v} |\sigma_r^\t  \partial V_r^s|^2\d r\bigg|^2+\bigg| \int_{\tau-}^T L_v |\sigma_r^\t  V_r^s|^2\d r\bigg|^2+ \bigg| \int_{\tau-}^T L_z |\sigma_r^\t  \Zc_r|^2\d r\bigg|^2 +\bigg|\int_{\tau-}^T {V_r^s}^\t \d X_r\bigg|^2 \bigg]\\
& \leq \ 10\Big( \|\partial_s \eta(s)\|^2_{\Lc^{\infty,2}}+ \|\nabla \tilde g\|^2_{\L^{1,\infty,2}}+ (1+L_{\rm u} T^2)\| \partial U\|^2_{\Sc^{\infty,2}} +  L_u T  \|U\|^2_{\Sc^{\infty,2}}+  L_yT  \|\Yc\|^2_{\Sc^{\infty}} \\
&\quad  + \|\partial V\|^2_{ \H_{\rm BMO}^{2,2}}\big( 1+ 2 L_{\rm v}^2 \|\partial V\|^2_{ \H_{\rm BMO}^{2,2}}\big)+ 2 L_v^2 \|V\|^4_{\H^{2,2}_{\rm BMO}}+ 2 L_z^2 \|Z\|^4_{\H^2_{\rm BMO}}\Big).
\end{align*}
\end{proof}

\begin{proof}[Proof of {\rm \Cref{Thm:solBMOnorm}}] Upon close inspection of the proof of \Cref{Thm:wp:smalldata}, we see that the only stages of the argument where both the presence and the norm of $(\Nc,M,\partial M)$ plays a role are in
\eqref{Eq:thm:wd:ineq:final}, \eqref{Eq:thm:wd:rineq} and \eqref{Eq:thm:cont:final}. We address each of them in the following. We consider the case $(i)$. The argument for $(ii)$ follows similarly. \medskip

If we were to require ${\rm BMO}$--norms on $(\Nc,M,\partial M)$ we see that
\begin{enumerate}[label=$(\roman*)$, ref=.$(\roman*)$,wide, labelwidth=!, labelindent=0pt]
\item In \eqref{Eq:thm:wd:ineq:final}, the presence of the $\es$ in the BMO--norm would require us to consider the estimate right before \eqref{Eq:thm:wd:ineq:final} for each of the $9$ processes that define the solution to \eqref{Eq:systemBSDEq}. Thus we obtain a factor $11$ instead of $10$. This yields
\[ C_{\eps}:= \min\big \{  1-11(\eps_{3}^{-1}+\eps_{12}^{-1}+\eps_{13}^{-1}) ,\; 1-11(\eps_{4}^{-1}+\eps_{14}^{-1}+\eps_{15}^{-1}) ,\;1-11(\eps_{5}^{-1}+\eps_{16}^{-1}+\eps_{17}^{-1}) ,\;1-11(\eps_{6}^{-1}+\eps_{18}^{-1}+\eps_{19}^{-1}+\eps_{20}^{-1})\big \}\]
\begin{align*}
&  I_0^\eps  \leq \gamma R^2/11,
\end{align*} 
and
\[
 \|(Y,Z,N,U,V,M,\partial U,\partial V,\partial M)\|^2_{\widehat{\Hc}^{c}}  \leq \ C_{\eps}^{-1} R^2\bigg(  \gamma  +24 L_{z\vee v}^2 R^2\bigg(   \eps_1+\eps_2+\sum_{i=12}^{20} \eps_{i}\bigg)   \bigg) .
\]
\item As a consequence of the previous observation \eqref{Eq:thm:wd:rineq} would be replace by
\[
 R^2 \leq\frac1{24L_\star} \frac{ C_{\eps} -  \gamma }{  \eps_1+\eps_2+ \sum_{i=12}^{20} \eps_{i} } < \frac{ 1}{2^{3}\cdot 3\cdot 7 \cdot 11^2\cdot  L_\star}=\Uc(11).
\]
where the upper bound results from the proper version of the optimisation procedure, i.e. \Cref{Lemma:upperbound}.
\item Likewise, with $C_{\tilde \eps}$ as in the proof, \eqref{Eq:thm:cont:final} is now given by
\begin{align*}
& \|(Y,Z,N,U,V,M,\partial U,\partial V,\partial M)\|^2_{\widehat \Hc^{c}}\\
  \leq &\ 24 C_{\tilde \eps}^{-1}   L_{z\vee v}^2 R^2(3 \tilde \eps_1 + 3 \tilde \eps_2 +2 \tilde \eps_3+2 \tilde \eps_4+2 \tilde \eps_{5}+3 \tilde \eps_{6}) \Big(\| \delta z\|_{\H^{2,c}_{{\rm BMO}}}^2 +    \| \delta v\|_{\overline \H^{2,2,c}_{{\rm BMO}}}^2+  \| \delta \partial v\|_{\H^{2,2,c}_{{\rm BMO}}}^2\Big).\\
<&\ \frac{1}{2^2 \cdot 7\cdot 11}(\sqrt{33} + \sqrt{33 + \tilde \eps_1+ \tilde \eps_2})^2   \Big(\| \delta z\|_{\H^{2,c}_{{\rm BMO}}}^2+ \| \delta v\|_{\H^{2,2,c}_{{\rm BMO}}}^2+  \| \delta \partial v\|_{\H^{2,2,c}_{{\rm BMO}}}^2  \Big)
\end{align*}
where the second inequality follows from the new version of the optimisation procedure, i.e. \Cref{Lemma:Contractionbound}.
\end{enumerate}

All things considered, {\bf Step 3} will lead to require \Cref{Assumption:wp:eta} for $\kappa=11$. By assumption the result follows.
\end{proof}

\section{Proof of the Quadratic case}\label{sec:proofquadratic}

\begin{proof}[Proof of {\rm\Cref{Thm:wpq:smalldata}}]
For $c>0$, let us introduce the mapping
\begin{align*}
\Tf:(\Bc_R, \|\cdot \|_{\Hc^c})&\longrightarrow (\Bc_R, \|\cdot \|_{\Hc^c})\\
(y,z,n,u,v,m,\partial u,\partial v,\partial m)&\longmapsto (Y,Z,N,U,V,M,\partial U,\partial V,\partial M),
\end{align*}
with $\mathfrak{H}=(\Yc,\Zc,\Nc,U,V,M,\partial U, \partial V, \partial M)$ given for any $s\in [0,T]$, $\P\as$ for any $t\in [0,T]$ by
\begin{align*}
\Yc_t&=\xi(T,X_{\cdot\wedge T})+\int_t^T h_r(X,y_r, z_r, u_r^r, v_r^r,\partial U_r^r)\d r-\int_t^T \Zc_r^\t d X_r-\int_t^T \d  \Nc_r,\\
U_t^s&= \eta (s,X_{\cdot\wedge,T})+\int_t^T  g_r(s,X,u_r^s,v_r^s, y_r, z_r) \d r-\int_t^T  {V_r^s}^\t \d X_r-\int_t^T \d  M^s_r,\\
\partial U_t^s&= \partial_s \eta (s,X_{\cdot\wedge,T})+\int_t^T   \nabla g_r(s,X,\partial u_r^s,\partial v_r^s,u_r^s,v_r^s, y_r, z_r) \d r-\int_t^T  \partial {V_r^s}^\t  \d X_r-\int_t^T \d  \partial M^s_r.
\end{align*}

{\bf Step 1:} We first argue that $\Tf$ is well--defined.

\begin{enumerate}[label=$(\roman*)$, ref=.$(\roman*)$,wide, labelwidth=!, labelindent=0pt]
\item Let us first remark that for $u\in \Sc^{\infty,2,c}$
\begin{align}\label{equation:utt}
\E\bigg[\int_0^T |u_t^t|^2\d r\bigg]\leq T\|u\|_{\Sc^{\infty,2,c}}^2.
\end{align} 

In light of \Cref{Assumption:Qgrowth}, there is $c>0$ such that $(\xi,\eta,\partial \eta, \tilde f, \tilde g, \nabla \tilde g) \in \Lc^{\infty, c}\times\big( \Lc^{\infty, 2, c}\big)^2 \times \L^{1,\infty, c}\times \big( \L^{1,\infty, 2,c}\big)^2$, thus, we may use \eqref{Eq:ineqBMO} and \eqref{equation:utt} to obtain
\begin{align*}
& \E \bigg[ |\xi(T)|^2 +\bigg| \int_0^T | h_t(y_t,z_t,u_t^t,v_t^t,0)| \d t\bigg|^2\bigg] + \sup_{s\in [0,T]}\E\bigg[  |\eta (s)|^2 +\bigg| \int_0^T |  g_t(s,u_t^s,v^s_t,y_t,z_t) | \d t\bigg|^2 \bigg]\\
&\; +  \sup_{s\in [0,T]}\E\bigg[  |\partial \eta (s)|^2 +\bigg| \int_0^T |  \nabla g_t(s,\partial u_t^s,\partial v^s_t,u_t^s,v^s_t,y_t^s,z_t) | \d t\bigg|^2 \bigg]\\
 \leq&\ \E \bigg[ |\xi(T)|^2 +5 \bigg|\int_0^T | \tilde h_t | \d t\bigg|^2+17 L_y^2  \bigg|\int_0^T | y_t |^2 \d t\bigg|^2+17 L_z^2 \bigg| \int_0^T | z_t |^2 \d t\bigg|^2+5L_u^2  \bigg|\int_0^T | u_t^t |^2 \d t\bigg|^2+5L_v^2  \bigg|\int_0^T | v_t^t |^2 \d t\bigg|^2\bigg] \\
&\; + \sup_{s\in [0,T]}\E\bigg[  |\eta (s)|^2+5 \bigg|\int_0^T | \tilde g_t(s) | \d t\bigg|^2+12 L_u^2 \bigg|\int_0^T |u_t^s |^2 \d t\bigg|^2+12 L_v^2 \bigg|\int_0^T |v_t^s |^2 \d t\bigg|^2\bigg]\\
 &\; + \sup_{s\in [0,T]}\E\bigg[  |\partial \eta (s)|^2+7 \bigg|\int_0^T | \nabla \tilde g_t(s) | \d t\bigg|^2+7 L_{{\rm u} }^2 \bigg| \int_0^T |\partial u_t^s |^2 \d t\bigg|^2+7 L_{{\rm v} }^2 \bigg| \int_0^T |\partial v_t^s |^2 \d t\bigg|^2\bigg]\\
 \leq&\ \| \xi\|_{\Lc^{\infty,c}}^2+ 5 \| \tilde h\|_{\L^{1,\infty,c}}^2+ \| \eta\|_{\Lc^{\infty,2,c}}^2+ 5  \| \tilde g\|_{\L^{1,\infty,2,c}}^2 + \|\partial  \eta\|_{\Lc^{\infty,2,c}}^2 + 7 \| \nabla \tilde g\|_{\L^{1,\infty,2,c}}^2 \\
&\; +17L_y^2 T^2\|y\|_{\Sc^{\infty,c}}^4 + 34 L_z^2 \|z\|_{\H^{2,c}_{\rm BMO}}^4 +17L_u^2 T^2\|u\|_{\Sc^{\infty,2,c}}^4 + 24 L_v^2 \|v\|_{\overline \H^{2,2,c}_{\rm BMO}}^4 +7L_{\rm u}^2 T^2\|\partial u\|_{\Sc^{\infty,2,c}}^4+ 14 L_{\rm  v}^2 \|\partial v\|_{\H^{2,2,c}_{\rm BMO}}^4<\infty.
\end{align*}
Therefore, by \cite[Theorem 3.2]{hernandez2020unified}, $\Tf$ defines a well--posed system of BSDEs with unique solution in the space $\Hf^2$. We recall the spaces involved in the definition of $\Hf^2$, and their corresponding norms, were introduced in \Cref{Sec:spaces}. \medskip

\item For $\Uc_t:=U_t^t, \; \Vc_t:=V_t^t, \; \Mc_t:=M_t^t-\int_0^t \partial M_r^r \d r,\; t\in [0,T]$, $(\Uc,\Vc,\Mc)\in \S^{2}\times \H^{2}\times \M^{2}$ and satisfy the equation
\[
\Uc_t=\eta(T,X_{\cdot\wedge T})+\int_t^T \big( g_r(r,X,u_r^r,v_r^r,y_r,z_r)-\partial U_r^r\big) \d r -\int_t^T \Vc_r^\t \d X_r-\int_t^T \d \Mc_r, \; t\in [0,T],\P\as
\]

\item $(\Yc,\Uc)\in \Sc^{\infty,c}\times \Sc^{\infty,c}$ and $\|U\|_{\Sc^{\infty,2,c}}+\|\partial U\|_{\Sc^{\infty,2,c}}<\infty $. \medskip

In light of \Cref{Assumption:Qgrowth}, $\d r\otimes\d \P\ae$
\begin{align}\label{Eq:thmq:wd:ineq:lip0}
\begin{split}
| h_r|&\leq L_y |y_r|^2+ L_z |\sigma^\t_r z_r|^2+  L_u  |u_r^r|^2+ L_v |\sigma^\t_r v_r^r|^2+  L_{\rm u}  |\partial U_r^r|+ |\tilde h_r|,  \\
|   g_r|&\leq  L_u  |u_r^r|^2  + L_v |\sigma^\t_r v_r^r|^2+ L_y |y_r|^2+ L_z |\sigma^\t_r z_r|^2+ |\partial U_r^r|+ |\tilde g_r|,  \\
|g_r(s)| &\leq  L_u |u_r^s|^2+ L_v |\sigma^\t_r v_r^s|^2+ L_y |y_r|^2+ L_z| \sigma^\t_r z_r|^2  +    |\tilde g_r(s)|,\\
|\nabla g_r(s)| &\leq  L_{\rm  u} |\partial u_r^s|^2+ L_{\rm v} |\sigma^\t_r \partial v_r^s|^2+L_u |u_r^s|^2+ L_v |\sigma^\t_r v_r^s|^2+ L_y |y_r|^2+ L_z| \sigma^\t_r z_r|^2  +    |\nabla \tilde g_r(s)|.
\end{split}
\end{align}

Again, we apply Meyer--It\^o's formula to $\e^{\frac{c}2  t}\big( |\Yc_t|+|\Uc_t|+|U_t^s|+|\partial U_t^s|\big)$ and take conditional expectations with respect to $\Fc_t$ in \Cref{Eq:thm:wd:ito}. Moreover, in combination with \eqref{Eq:thmq:wd:ineq:lip0} and \Cref{Lemma:estimatescontraction}, we obtain back in \eqref{Eq:thm:wd:ito} that
\begin{align*}
& \e^{\frac{c}2  t}\big( |\Yc_t|+|\Uc_t|+|U_t^s|+|\partial U_t^s|\big)  + \E_t\bigg[ \int_t^T\frac{c}2  \e^{\frac{c}2 r}  |\Yc_r|\d r\bigg] +  \E_t\bigg[ \int_t^T \frac{c}2 \e^{\frac{c}2 r}  |\Uc_r|    \d r\bigg]  \\
&\  +  \sup_{s\in [0,T]} \E_t\bigg[ \int_t^T \frac{c}2 \e^{\frac{c}2 r}  |U_r^s|   \d r \bigg] +  \sup_{s\in [0,T]} \E_t\bigg[ \int_t^T \frac{c}2 \e^{\frac{c}2 r} |\partial U_r^s|     \d r \bigg] \\
\leq &\   \E_t\big[\e^{\frac{c}2 T}\big( |\xi|+|\eta(T)|+|\eta(s)|+|\partial_s \eta(s)|\big)\big] +\E_t\bigg[ \int_t^T  \e^{\frac{c}2 r} \Big( |\tilde h_r| +|\tilde g_r|+   |\tilde g_r(s)|+   |\nabla \tilde g_r(s)| \Big) \d r \bigg]  \\
&\; + (T+ L_{\rm u} T) \Big ( \|\partial_s \eta \|_{\Lc^{\infty,2,c}}   + \|\nabla \tilde g\|_{\L^{1,\infty,2,c}}+ TL_\star\big( \|y\|_{\Sc^{\infty,c}}^2+\|u\|_{\Sc^{\infty,2,c}}^2+\|\partial u\|_{\Sc^{\infty,2,c}}^2\big) \Big)\\
& \;  + (T+ L_{\rm u} T) L_\star\big( \|z\|^2_{\H^{2,c}_{{\rm BMO}}}+ \|v\|^2_{\H^{2,2,c}_{{\rm BMO}}}+\|\partial v\|^2_{\H^{2,2,c}_{{\rm BMO}}}\big) +\E_t\bigg[   \int_t^T \e^{cr}\Big( 4 L_y  |\sigma^\t_r y_r|^2+2 L_u  |\sigma^\t_r u_r^r|^2+2 L_u  |\sigma^\t_r u_r^s|^2\Big) \d r \bigg]\\
& \;  +\E_t\bigg[   \int_t^T \e^{cr}\Big( L_{\rm u}  |\sigma^\t_r \partial u_r^s|^2+4 L_z  |\sigma^\t_r z_r|^2 +  2 L_v  |\sigma^\t_r v_r^r|^2+2  L_v  |\sigma^\t_r v_r^s|^2+  L_{\rm v}|\sigma^\t_r \partial v_r^s|^2\Big) \d r \bigg].
\end{align*}
where we recall the notation $L_\star=\max\{L_y,L_u,L_{\rm u},L_z,L_v,L_{\rm v}\}$. Thus, for any $c>0$ we obtain 
\begin{align*}
& \mathrm{max}\big\{ \e^{\frac{c}2 t} |\Yc_t| ,\, \e^{\frac{c}2 t} |\Uc_t| ,\,\e^{\frac{c}2 t} |U_t^s|,\,\e^{\frac{c}2 t} |\partial U_t^s| \big\} \\
  \leq &\ \|\xi\|_{\Lc^{\infty,c}} +  \|\tilde h\|_{\L^{1,\infty,2,c}}+ 2\big(  \|\eta \|_{\Lc^{\infty,2,c}} + \|\tilde g\|_{\L^{1,\infty,2,c}} \big)    +(1+T+TL_{{\rm u}} )\big( \|\partial_s \eta \|_{\Lc^{\infty,2,c}}+ \|\nabla \tilde g\|_{\L^{1,\infty,2,c}} \big)   \\
 & \;   + (4+T+L_{\rm u} T)  L_{\star} T   \Big( \|y\|^2_{\Sc^{2\infty,c}}+\|u\|^2_{\Sc^{\infty,2,c}}+ \|\partial u\|^2_{\Sc^{\infty,2,c}}\Big)  + (4+T+L_{\rm u} T)  L_{\star}   \Big( \|z\|^2_{\H^{2,c}_{{\rm BMO}}}+\|v\|^2_{\overline \H^{2,2,c}_{{\rm BMO}}}+ \|\partial v\|^2_{\H^{2,2,c}_{{\rm BMO}}}\Big),
\end{align*}

\item We show $(\Zc,\Vc,\Nc,\Mc)\in \big(\H^{2,c}_{{\rm BMO}}\big)^2\times \big(\M^{2,c}\big)^2$ and $\| V\|_{ \H^{2,2,c}_{{\rm BMO}}}^2+\|M\|_{{\M}^{2,2,c}}^2+\| \partial V\|_{ \H^{2,2,c}_{{\rm BMO}}}^2+\|\partial M\|_{{\M}^{2,2,c}}^2<\infty $. \medskip

From $(iii)$, \Cref{Assumption:LQgrowth}\ref{Assumption:LQgrowth:2} and \ref{Assumption:LQgrowth}\ref{Assumption:LQgrowth:3}, together with Young's inequality, yield that, for any $\eps_i>0$, $i\in\{1,2\}$, and defining $  C_{\eps_{1}}:=\eps_1^{-1} 7T  L_{\rm u}^2$, and $  C_{\eps_{2}}:= \eps_2^{-1} 7T$, we have
\begin{align*}
2 \Yc_r \cdot h_r-c |\Yc_r|^2 & \leq 2 \|\Yc\|_{\Sc^{\infty,c}}  \big( L_y  |y_r|^2+L_z  |\sigma^\t_r z_r|^2+L_u  |u_r^r|^2+L_v  |\sigma^\t_r v_r^r|^2+ |\tilde h_r|\big)\\
&\quad + \eps_1(7T)^{-1}  |\partial U_r^r|^2+  (  \widetilde C_{\eps_{1}}-c) |\Yc_r|^2 \\[0,5em]
2 \Uc_r \cdot  g_r-c |\Uc_r|^2 & \leq  2 \|\Uc\|_{\Sc^{\infty,c}}  \big( L_y  | y_r|^2+L_z  |\sigma^\t_r z_r|^2+ | u_r^r|^2+L_v  |\sigma^\t_r v_r^r|^2+ |\tilde g_r|\big)\\
&\quad + \eps_2(7T)^{-1}  |\partial U_r^r|^2 + (  \widetilde C_{\eps_{2}}-c) |\Uc_r|^2 , \\[0,5em]
2 U_r^s \cdot  g_r(s)-c |U_r^s|^2 & \leq   2 \|U\|_{\Sc^{\infty,2,c}}  \big( L_u  |  u_r^s|^2  + L_{ v} |\sigma^\t_r  v_r^s|^2+L_y  |  y_r|^2+L_z  |\sigma^\t_r z_r|^2+ |\tilde g_r(s)| \big)-c |U_r^s|^2 , \\[0,5em]
 2 \partial U_r^s \cdot \nabla g_r(s)-c |\partial U_r^s|^2&  \leq  2\|\partial U\|_{\Sc^{\infty,c,2}}  \big(L_{\rm u}  |  \partial u_r^s|^2  + L_{ v} |\sigma^\t_r \partial v_r^s|^2+ L_u  |  u_r^s|^2  + L_{\rm v} |\sigma^\t_r  v_r^s|^2+ |\nabla \tilde g_r(s)|  \big)\\
&\quad  +2\|\partial U\|_{\Sc^{\infty,c,2}}  \big( L_y  | y_r|^2+L_z  |\sigma^\t_r z_r|^2\big) -c  |\partial U_r^s|^2
\end{align*}

These inequalities in combination with the analogous version of \Cref{Lemma:estimatescontraction} (which holds for $c>2L_{\rm u}$), Young's inequality,  and It\^o's formula, as in \eqref{eq:eq2}, show that for any $\eps_i>0$, $i\in\{3,...,24\}$
\begin{align*}
&\sum_{i=1}^4 \e^{ct}  |\Yf_t^i|^2+\E_t\bigg[  \int_t^T \e^{cr}  |\sigma^\t_r  \Zf_r^i|^2 \d r+\int_t^T\e^{c r-} \d \Tr [\Nf^i]_r\bigg]+ \E_t\bigg[ \int_t^T \e^{cr}\big(  |\Yc_r|^2 (c-C_{\eps_{1}})+|\Uc_r|^2 (c- C_{\eps_{2}} )\big)\d r\bigg] \\
&\; +\sup_{s\in [0,T]}  \E_t\bigg[ \int_t^T c \e^{cr}|U_r^s|^2 \d r\bigg] +\sup_{s\in [0,T]}  \E_t\bigg[ \int_t^T c \e^{cr} |\partial U_r^s|^2  \d r\bigg]\\
= & \  + \E_t\Big[  \e^{cT}\big( |\xi|^2+ |\eta(T)|^2+|\eta(s)|^2+|\partial_s \eta(s)|^2\big)\Big] + ( \eps_1+\eps_2)  \Big( \|\partial_s \eta \|_{\Lc^{\infty,2,c}}^2  +\|\nabla \tilde g\|^2_{\L^{1,\infty,2,c}} \Big) \\
&\; +( \eps_1+\eps_2) \Big( L_\star T^2 \|y\|_{\Sc^{\infty,c}}^4+L_\star T^2 \|u\|_{\Sc^{\infty,c}}^4+L_\star T^2 \|\partial u\|_{\Sc^{\infty,c}}^4+ 2 L_\star^2  \Big( \|\partial v\|^4_{\H^{2,2,c}_{{\rm BMO}}} + \|v\|^4_{\H^{2,2,c}_{{\rm BMO}}}+ \|z\|^4_{\H^{2,c}_{{\rm BMO}}}\Big) \Big)\\
&\; +\big(\eps_{3}^{-1}+ \eps_{7}^{-1}+\eps_{8}^{-1}+\eps_{9}^{-1} +\eps_{10}^{-1}\big) \|\Yc\|^2_{\Sc^{\infty,c}}  +\big(\eps_{4}^{-1}+\eps_{11}^{-1}+ \eps_{12}^{-1}+ \eps_{13}^{-1}+\eps_{14}^{-1}\big) \|\Uc\|^2_{\Sc^{\infty,c}} \\
&\; +\big(\eps_{5}^{-1}+ \eps_{15}^{-1}+\eps_{16}^{-1}+\eps_{17}^{-1} +\eps_{18}^{-1}\big)  \|U\|^2_{\Sc^{\infty,c,2}} +\big(\eps_{6}^{-1}+\eps_{19}^{-1}+\eps_{20}^{-1} +\eps_{21}^{-1}+\eps_{22}^{-1}+\eps_{23}^{-1}+\eps_{24}^{-1}\big)  \|\partial U\|^2_{\Sc^{\infty,c,2}} \\
&\; +\E_t\bigg[\eps_3  \bigg|\int_t^T  \e^{cr}  |\tilde h_r|\d r\bigg|^2 + \eps_{4} \bigg|\int_t^T  \e^{cr} |\tilde g_r|\d r\bigg|^2 \bigg]   + \eps_{5} \bigg|\int_t^T \e^{cr} |\tilde g_r(s)|\d r\bigg|^2+ \eps_{6} \bigg|\int_t^T \e^{cr} |\nabla \tilde g_r|\d r\bigg|^2 \bigg]\\
&\; +( \eps_{7} +\eps_{11}+\eps_{15}+\eps_{19} )L_\star T^2\|y\|^4_{\Sc^{\infty,c}}
+( \eps_{9} +\eps_{13} +\eps_{17}+\eps_{21})L_\star T^2 \|u\|^4_{\Sc^{\infty,2,c}}+ \eps_{23}L_\star T^2 \|\partial u\|^4_{\Sc^{\infty,2,c}}\\
&\; +( \eps_{8}+\eps_{12}+\eps_{16}+\eps_{20} ) L_z^2\E_t\bigg[  \bigg|\int_t^T \e^{cr} |\sigma^\t_r z_r|^2 \d r\bigg|^2\bigg]+( \eps_{10}+\eps_{14} ) L_v^2\E_t\bigg[  \bigg|\int_t^T \e^{cr} |\sigma^\t_r v_r^r|^2 \d r\bigg|^2\bigg]\\
&\;  + (\eps_{18}+\eps_{22}) L_v^2\E_t\bigg[  \bigg|\int_t^T \e^{cr} |\sigma^\t_r v_r^s|^2 \d r\bigg|^2\bigg] +\eps_{24} L_{\rm v}^2\E_t\bigg[  \bigg|\int_t^T \e^{cr} |\sigma^\t_r \partial v_r^s|^2 \d r\bigg|^2\bigg]
\end{align*}

We now let $\tau \in \Tc_{0,T}$. In light of \eqref{Eq:cYwelldefined}, for
\begin{align}\label{Eq:cZwelldefinedq}
\begin{split}
 c\geq \max &\   \{  \eps_1^{-1} 7T  L_{\rm u}^2,  \eps_2^{-1} 7T , 2L_{\rm u}\} ,
\end{split}
\end{align}
\Cref{Eq:ineqBMO} yields
\begin{align*}
&\sum_{i=1}^4 \e^{ct}  |\Yf_t^i|^2+\E_t\bigg[  \int_t^T \e^{cr}  |\sigma^\t_r  \Zf_r^i|^2 \d r+\int_t^T\e^{c r-} \d \Tr [\Nf^i]_r\bigg]\\
 =&\ \|\xi\|_{\Lc^{\infty,c}}^2+2 \|\eta \|_{\Lc^{\infty,2,c}}^2 \! + (1+ \eps_1+\eps_2)\|\partial_s \eta \|_{\Lc^{\infty,2,c}}^2\! +  \eps_3  \|  \tilde h\|^2_{\L^{1,\infty,c}} \!+ ( \eps_{4} +\eps_{5}) \|  \tilde g\|^2_{\L^{1,\infty,2,c}} \! + ( \eps_1+\eps_2+ \eps_{6}) \| \nabla  \tilde g\|^2_{\L^{1,\infty,2,c}}   \\
&\; + L_\star^2  T^2( \eps_1+\eps_2+\eps_{7} +\eps_{11}+\eps_{15}+\eps_{19}  )  \|y\|^4_{\Sc^{\infty,c}} + L_\star^2  T^2( \eps_1+\eps_2+ \eps_{9}+\eps_{13}+\eps_{17}+\eps_{21} )  \|u\|^4_{\Sc^{\infty,2,c}}\\
&\; +L_\star^2  T^2(\eps_1+\eps_2+\eps_{23})\|\partial u\|_{\Sc^{\infty,2,c}}^4 +2 L_\star^2  ( \eps_1+\eps_2+ \eps_{8}+\eps_{12}+\eps_{16}+\eps_{20} )  \|z\|^4_{\H^{2,c}_{{\rm BMO}}}  \\
&\; +2 L_\star^2  ( \eps_1+\eps_2+ \eps_{10}+\eps_{14}+\eps_{18}+\eps_{22} ) \|v\|^4_{\overline \H^{2,2,c}_{{\rm BMO}}} + 2 L_\star^2  ( \eps_1+\eps_2+\eps_{24}) \|\partial v\|^4_{\H^{2,2,c}_{{\rm BMO}}}\\
&\; +\big(\eps_{3}^{-1}+ \eps_{7}^{-1}+\eps_{8}^{-1}+\eps_{9}^{-1} +\eps_{10}^{-1}\big) \|\Yc\|^2_{\Sc^{\infty,c}}  +\big(\eps_{4}^{-1}+\eps_{11}^{-1}+ \eps_{12}^{-1}+\eps_{13}^{-1} + \eps_{14}^{-1}\big) \|\Uc\|^2_{\Sc^{\infty,c}} \\
&\; +\big(\eps_{5}^{-1}+ \eps_{15}^{-1}+\eps_{16}^{-1}+\eps_{17}^{-1}+\eps_{18}^{-1} \big)  \|U\|^2_{\Sc^{\infty,c,2}} +\big(\eps_{6}^{-1}+\eps_{19}^{-1}+\eps_{20}^{-1} +\eps_{21}^{-1}+\eps_{22}^{-1}+\eps_{23}^{-1}+\eps_{24}^{-1}\big)  \|\partial U\|^2_{\Sc^{\infty,c,2}} 
\end{align*}
which in turn leads to
\begin{align}\label{Eq:thm:wdq:ineq:final}
\begin{split}
&\frac{1}{10}\Big(\|\Yc\|^2_{\Sc^{\infty,c}} +\|\Uc\|^2_{\Sc^{\infty,c}} +\|U\|^2_{\Sc^{\infty,2,c}}+\|\partial U\|^2_{\Sc^{\infty,2,c}} +    \|\Zc\|_{\H^{2,c}_{ {\rm BMO}}}^2  \\
&\quad + \|V\|_{\overline \H^{2,2,c}_{{\rm BMO}}}^2+ \|\partial V\|_{\H^{2,2,c}_{{\rm BMO}}}^2 + \|\Nc\|_{{\M}^{2,c}}^2+  \|M\|_{{\M}^{2,2,c}}^2+  \|\partial M\|_{{\M}^{2,2,c}}^2
\Big)   \\
 \leq&\ \|\xi\|_{\Lc^{\infty,c}}^2+2 \|\eta \|_{\Lc^{\infty,2,c}}^2 \! + (1+ \eps_1+\eps_2)\|\partial_s \eta \|_{\Lc^{\infty,2,c}}^2\! +  \eps_3  \|  \tilde h\|^2_{\L^{1,\infty,c}} \\
 &\; + ( \eps_{4} +\eps_{5}) \|  \tilde g\|^2_{\L^{1,\infty,2,c}} \! + ( \eps_1+\eps_2+ \eps_{6}) \| \nabla  \tilde g\|^2_{\L^{1,\infty,2,c}}   \\
&\; + L_\star^2  T^2( \eps_1+\eps_2+\eps_{7} +\eps_{11}+\eps_{15}+\eps_{19}  )  \|y\|^4_{\Sc^{\infty,c}} + L_\star^2  T^2( \eps_1+\eps_2+ \eps_{9}+\eps_{13}+\eps_{17}+\eps_{21} )  \|u\|^4_{\Sc^{\infty,2,c}}\\
&\; +L_\star^2  T^2(\eps_1+\eps_2+\eps_{23})\|\partial u\|_{\Sc^{\infty,2,c}}^4 +2 L_\star^2  ( \eps_1+\eps_2+ \eps_{8}+\eps_{12}+\eps_{16}+\eps_{20} )  \|z\|^4_{\H^{2,c}_{{\rm BMO}}}  \\
&\; +2 L_\star^2  ( \eps_1+\eps_2+ \eps_{10}+\eps_{14}+\eps_{18}+\eps_{22} ) \|v\|^4_{\overline \H^{2,2,c}_{{\rm BMO}}} + 2 L_\star^2  ( \eps_1+\eps_2+\eps_{24}) \|\partial v\|^4_{\H^{2,2,c}_{{\rm BMO}}}\\
&\; +\big(\eps_{3}^{-1}+ \eps_{7}^{-1}+\eps_{8}^{-1}+\eps_{9}^{-1} +\eps_{10}^{-1}\big) \|\Yc\|^2_{\Sc^{\infty,c}}  +\big(\eps_{4}^{-1}+\eps_{11}^{-1}+ \eps_{12}^{-1}+\eps_{13}^{-1} + \eps_{14}^{-1}\big) \|\Uc\|^2_{\Sc^{\infty,c}} \\
&\; +\big(\eps_{5}^{-1}+ \eps_{15}^{-1}+\eps_{16}^{-1}+\eps_{17}^{-1}+\eps_{18}^{-1} \big)  \|U\|^2_{\Sc^{\infty,c,2}} +\big(\eps_{6}^{-1}+\eps_{19}^{-1}+\eps_{20}^{-1} +\eps_{21}^{-1}+\eps_{22}^{-1}+\eps_{23}^{-1}+\eps_{24}^{-1}\big)  \|\partial U\|^2_{\Sc^{\infty,c,2}} 
 \end{split}
\end{align}

From \eqref{Eq:thm:wdq:ineq:final} we conclude $(Z,N)\in \H^{2,c}_{{\rm BMO}}\times  {\M}^{2,c}$, $\| V\|_{\overline \H^{2,2,c}_{{\rm BMO}}}^2+\| \partial V\|_{\H^{2,2,c}_{{\rm BMO}}}^2+ \|M\|_{{\M}^{2,2,c}}^2+ \|\partial M\|_{{\M}^{2,2,c}}^2<\infty$.\medskip

Defining $C_{\eps}$ analogously and if for some $\gamma\in(0,\infty)$
\begin{align}\label{Eq:thm:wpq:smalldatacond}
\begin{split}
 I_0^\eps\leq \gamma R^2/10,
\end{split}
 \end{align} 
we obtain back in \eqref{Eq:thm:wdq:ineq:final}
\begin{align*}
& \|(Y,Z,N,U,V,M,\partial U,\partial V,\partial M)\|^2_{\Hc^{c}}\\
  \leq &\  C_{\eps}^{-1} \Big( 10 I_0^\eps  +10L_\star^2 \max\{2,T^2\}  \big( 
 ( \eps_1+\eps_2+\eps_{7} +\eps_{11}+\eps_{15}+\eps_{19}  )  \|y\|^4_{\Sc^{\infty,c}} +( \eps_1+\eps_2+ \eps_{9}+\eps_{13}+\eps_{17}+\eps_{21} )  \|u\|^4_{\Sc^{\infty,2,c}}\\
&\qquad +(\eps_1+\eps_2+\eps_{23})\|\partial u\|_{\Sc^{\infty,2,c}}^4 +  ( \eps_1+\eps_2+ \eps_{8}+\eps_{12}+\eps_{16}+\eps_{20} )  \|z\|^4_{\H^{2,c}_{{\rm BMO}}}  \\
&\qquad + ( \eps_1+\eps_2+ \eps_{10}+\eps_{14}+\eps_{18}+\eps_{22} ) \|v\|^4_{\overline \H^{2,2,c}_{{\rm BMO}}} +   ( \eps_1+\eps_2+\eps_{24}) \|\partial v\|^4_{\H^{2,2,c}_{{\rm BMO}}}\Big)  \\
  \leq  &\   C_{\eps}^{-1} R^2  \bigg(\gamma  +10L_\star^2 \max\{2,T^2\} R^2   \bigg( \eps_1+\eps_2+\sum_{i=7}^{24} \eps_{i} \bigg)   \bigg) 
\end{align*}

Therefore, to obtain $\Tf(\Bc_R)\subseteq \Bc_R$, that is to say that the image under $\Tf$ of the ball of radius $R$ is contained in the ball of radius $R$, it is necessary to find $R^2$ such that the term in parentheses above is less or equal than $C_{\eps}$, i.e.
\[
 R^2 \leq \frac1{1 0 L_\star^2\max\{ 2, T^2\} }\frac{ C_{\eps} -  \gamma }{  \eps_1+\eps_2+ \sum_{i=7}^{24} \eps_{i} }
\]
which after optimising the choice of $\eps$'s renders
\begin{align}\label{Eq:Rwelldefinedq}
 R^2 < \frac{1}{2^6\cdot 3\cdot 5^2\cdot 7\cdot L^2_\star\cdot \max\{ 2, T^2\} }
\end{align}
\item The continuity of the applications $([0,T],\Bc([0,T])) \longrightarrow (\Sc^{\infty,c},\|\cdot \|_{ \Sc^{\infty,c}})\big($resp. $(\H_{{\rm BMO}}^{2,c},\|\cdot \|_{\H_{{\rm BMO}}^{2,c}}),\, ({\M}^{2,c},\|\cdot \|_{{\M}^{2,c}} )\big)  : s \longmapsto \varphi^s $ for $\varphi=U^s,\partial U^s\, ($resp. $V^s,\partial V^s, M^s,\partial M^s).$ follows analogously as in the proof \Cref{Thm:wp:smalldata}.
\end{enumerate}

We conclude, $\Tf(\Bc_R)\subseteq \Bc_R$ for all $R$ satisfying \eqref{Eq:Rwelldefinedq}.\medskip

{\bf Step 2:} We now argue that $\Tf$ is a contraction in $\Bc_R\subseteq \Hc$ for the norm $\| \cdot \|_{\Hc^c}$. Let 
\begin{align*}
  \delta h_t&:=h_t(y^1_t,z^1_t,u_t^{1,t}, v_t^{1,t},\partial U_t^{1,t})-h_t(y^2_t,z^2_t,u_t^{2,t}, v_t^{2,t},\partial U_t^{2,t}),\\
   \delta   g_t&:=g_t(t,u_t^{1,t},v^{1,t}_t,y_t^1,z_t^1)-\partial U_t^{1,t} -g_t(t,u_t^{2,t},v^{2,t}_t,y_t^2,z_t^2)+\partial U_t^{2,t},\\ 
   \delta \tilde h_t&:=h_t(y^1_t,z^1_t,u_t^{1,t}, v_t^{1,t},\partial U_t^{2,t})-h_t(y^2_t,z^2_t,u_t^{2,t}, v_t^{2,t},\partial U_t^{2,t}),\\
 \delta \tilde g_t&:=g_t(t,u^{1,t}_t,v^{1,t}_t,y_t^1,z_t^1) -g_t(t,u^{2,t}_t,v^{2,t}_t,y_t^2,z_t^2),\\
 \delta  \tilde g_t(s)&:=g_t(s,u^{1,s}_t,v^{1,s}_t,y_t^1,z_t^1)-g_t(s,u^{2,s}_t,v^{2,s}_t,y_t^2,z_t^2),\\
      \delta \nabla  \tilde g_t(s)&:=\nabla g_t(s,\partial u^{1,s}_t,\partial v^{1,s}_t,u^{1,s}_t,v^{1,s}_t,y_t^1,z_t^1)-g_t(s,\partial u^{2,s}_t,\partial v^{2,s}_t,u^{2,s}_t,v^{2,s}_t,y_t^2,z_t^2).
 \end{align*}

Applying It\^o's formula we obtain that for any $t\in[0,T]$
\begin{align*}
&\sum_{i=1}^4 \e^{ct}  |\delta \Yf_t^i|^2+\int_t^T \e^{cr}  |\sigma^\t_r  \delta \Zf_r^i|^2 \d r+\int_t^T\e^{c r-} \d \Tr [\delta \Nf^i]_r  +\delta \widetilde  \Mf_t -\delta \widetilde  \Mf_T\\
 =&\  \int_t^T  \e^{c  r} \bigg( 2  \delta \Yc_r \cdot \delta h_r +  2 \delta \Uc_r  \cdot \delta  g_r +  2 \delta U_r^s  \cdot \delta \tilde g_r(s)+ 2 \delta \partial U_r^s \cdot \delta \nabla \tilde g_r(s)\bigg) \d r \\
 \leq &\ \int_t^T  \e^{c  r} \bigg(2 | \delta \Yc_r| \big( L_{\rm u} |\delta \partial U_r^r|+|\delta  \tilde h_r|\big)  +2 | \delta \Uc_r| \big(  |\delta \partial U_r^r|+|\delta  \tilde g_r| \big) +2 | \delta U_r^s|  |\delta  \tilde g_r(s)|  +2 | \delta \partial U_r^s|  |\delta  \nabla \tilde g_r(s)| -c \sum_{i=1}^4    |\delta \Yf_r^i|^2 \bigg) \d r
\end{align*}

where $\delta \widetilde \Mf$ denotes the corresponding martingale term. Let $\tau \in \Tc_{0,T}$, as in \Cref{Lemma:estimatescontraction} we obtain for $c>2L_{\rm u}$
\begin{align}\label{Eq:ineqdeltaUtt:quadratic}
\begin{split}
\E_\tau\bigg[ \int_\tau^T \frac{ \e^{cr}}{3   T}  |\delta \partial U_r^r|^2\d r\bigg] &\leq    \sup_{s\in [0,T]}  \es_{\tau \in \Tc_{0,T}}  \bigg|  \E_\tau \bigg[  \int_\tau^T \e^{c r} |\delta \nabla \tilde g_r(s)|\d r \bigg] \bigg|^2 
\end{split}
\end{align}

We now take conditional expectation with respect to $\Fc_\tau$ in the expression above and use \Cref{Assumption:LQgrowth} in combination with \eqref{Eq:ineqdeltaUtt:quadratic}. We then obtain from Young's inequality that for any $\tilde \eps_i\in (0,\infty)$, $i\in \{1,2\}$,
and 
\begin{align}\label{Eq:c:contraction1q}
\begin{split}
c\geq  \max&\  \{ \tilde \eps_1^{-1} 3T L_{\rm u}^2,\;  3T\tilde \eps_2^{-1},\; 2 L_{\rm u}  \},
\end{split}
\end{align}
it follows that
\begin{align}\label{Eq:contractionItoq}\begin{split}
& \sum_{i=1}^4 \e^{ct}  |\delta \Yf_t^i|^2+\E_\tau\bigg[\int_t^T \e^{cr}  |\sigma^\t_r  \delta \Zf_r^i|^2 \d r+\int_t^T\e^{c r-} \d \Tr [\delta \Nf^i]_r  \bigg] \\
  \leq &\  \tilde \eps_3^{-1} \|\delta Y\|^2_{\Sc^{\infty,c}}+ \tilde \eps_4^{-1} \|\delta \Uc\|^2_{\Sc^{\infty,2,c}} + \tilde \eps_{5}^{-1} \|\delta U\|^2_{\Sc^{\infty,2,c}} + \tilde \eps_{6}^{-1} \|\delta \partial U\|^2_{\Sc^{\infty,2,c}} \\
&\; + ( \tilde \eps_1+\tilde \eps_2+\tilde \eps_6)  \sup_{s\in [0,T]}  \es_{\tau \in \Tc_{0,T}}  \bigg|  \E_\tau \bigg[  \int_\tau^T \e^{c r} |\delta \nabla \tilde g_r(s)|\d r \bigg] \bigg|^2 +  \tilde \eps_3 \es_{\tau \in \Tc_{0,T}}  \bigg|  \E_\tau \bigg[  \int_\tau^T \e^{c r} |\delta \tilde h_t |\d r \bigg] \bigg|^2 \\
&\; + \tilde \eps_4     \es_{\tau \in \Tc_{0,T}}  \bigg|  \E_\tau \bigg[  \int_\tau^T \e^{c r} |\delta \tilde g_t|\d r \bigg] \bigg|^2 + \tilde \eps_{5}  \sup_{s\in [0,T]}  \es_{\tau \in \Tc_{0,T}}  \bigg|  \E_\tau \bigg[  \int_\tau^T \e^{c r} |\delta \tilde g_t(s)|\d r \bigg] \bigg|^2 
\end{split}
\end{align}

We now estimate the terms on the right side of \eqref{Eq:contractionItoq}. Note that in light of \Cref{Assumption:LQgrowth}\ref{Assumption:LQgrowth:3} we have

\begin{align*}
 &\, \bigg|  \E_\tau \bigg[  \int_\tau^T \e^{c r} |\delta \nabla \tilde g_t(s)|\d r \bigg] \bigg|^2  \\
  \leq &\,  \bigg|  \E_\tau \bigg[  \int_\tau^T \e^{c r} \Big( L_{\rm u} | \delta \partial u^s_r|\big( | \partial u_r^{1,s}|+| \partial u_r^{2,s}|\big)+ L_{\rm v} |\sigma^\t_r \delta \partial v^s_r|\big( |\sigma^\t_r \partial v_r^{1,s}|+|\sigma^\t_r \partial v_r^{2,s}|\big)\\
 &\hspace{6em}  + L_u|  \delta u^s_r|\big( | u_r^{1,s}|+| u_r^{2,s}|\big) + L_v|\sigma^\t_r \delta v^s_r|\big( |\sigma^\t_r v_r^{1,s}|+|\sigma^\t_r v_r^{2,s}|\big) \\
&\hspace{6em}+ L_y|  \delta y_r|\big( | y_r^1|+| y_r^2|\big) + L_z|\sigma^\t_r \delta z_r|\big( |\sigma^\t_r  z_r^1|+|\sigma^\t_r z_r^2|\big)\Big) 	\d r  \bigg]  \bigg|^2 \\
\leq &\ 6 L_\star^2   \E_\tau \bigg[  \int_\tau^T \e^{c r}   | \delta \partial u^s_r|^2\d r \bigg] \E_\tau \bigg[  \int_\tau^T \e^{c r} \big( |  \partial u_r^{1,s}|+|  \partial u_r^{2,s}|\big)^2\d r\bigg]\\
&+6 L_\star^2   \E_\tau \bigg[  \int_\tau^T \e^{c r}   |\sigma^\t_r \delta \partial v^s_r|^2\d r \bigg] \E_\tau \bigg[  \int_\tau^T \e^{c r} \big( |\sigma^\t_r \partial v_r^{1,s}|+|\sigma^\t_r \partial v_r^{2,s}|\big)^2\d r\bigg]\\
&+6 L_\star^2   \E_\tau \bigg[  \int_\tau^T \e^{c r}   | \delta  u^s_r|^2\d r \bigg] \E_\tau \bigg[  \int_\tau^T \e^{c r} \big( |    u_r^{1,s}|+|   u_r^{2,s}|\big)^2\d r\bigg]\\
& + 6 L_\star^2 \E_\tau \bigg[  \int_\tau^T \e^{c r}   |\sigma^\t_r \delta  v^s_r|^2\d r \bigg] \E_\tau \bigg[  \int_\tau^T \e^{c r} \big( |\sigma^\t_r   v_r^{1,s}|+|\sigma^\t_r   v_r^{2,s}|\big)^2\d r\bigg]\\
&+6 L_\star^2   \E_\tau \bigg[  \int_\tau^T \e^{c r}   | \delta  y_r|^2\d r \bigg] \E_\tau \bigg[  \int_\tau^T \e^{c r} \big( |    y_r^{1}|+|   y_r^{2}|\big)^2\d r\bigg]\\
&+  6 L_\star^2  \E_\tau \bigg[  \int_\tau^T \e^{c r}   |\sigma^\t_r \delta z_r|^2\d r \bigg] \E_\tau \bigg[  \int_\tau^T \e^{c r} \big( |\sigma^\t_r z_r^1|+|\sigma^\t_r z_r^2|\big)^2\d r\bigg]\\
\leq&\  6 L_\star^2 R^2\max\{ 2, T\}  \bigg( \E_\tau \bigg[  \int_\tau^T \e^{c r}   | \delta \partial u^s_r|^2\d r \bigg]+  \E_\tau \bigg[  \int_\tau^T \e^{c r}   |\sigma^\t_r \delta \partial v^s_r|^2\d r \bigg] + \E_\tau \bigg[  \int_\tau^T \e^{c r}   | \delta  u^s_r|^2\d r \bigg]\\
&\hspace{8em} +\E_\tau \bigg[  \int_\tau^T \e^{c r}   |\sigma^\t_r \delta  v^s_r|^2\d r \bigg]+ \E_\tau \bigg[  \int_\tau^T \e^{c r}   | \delta  y_r|^2\d r \bigg]+\E_\tau \bigg[  \int_\tau^T \e^{c r}   |\sigma^\t_r \delta z_r|^2\d r \bigg]\bigg)\\
\leq&\  6 L_\star^2 R^2\max\{ 2, T^2\}  \Big( \|\delta \partial u\|_{\Sc^{\infty,2,c}}^2+ \| \delta \partial v\|_{\H^{2,2,c}_{{\rm BMO}}}^2 +\|\delta  u\|_{\Sc^{\infty,2,c}}^2+ \| \delta v\|_{\overline \H^{2,2,c}_{{\rm BMO}}}^2 +\|\delta y\|_{\Sc^{\infty,c}}^2+\| \delta z\|_{\H^{2,c}_{{\rm BMO}}}^2\Big)
\end{align*}

where in the second inequality we used \eqref{Eq:ineqsquare} and Cauchy--Schwartz's inequality. Similarly
\begin{align*}
&\max\bigg\{ \bigg|\E_\tau \bigg[  \int_\tau^T \e^{c r} |\delta\tilde  h_r|\d r \bigg]  \bigg|^2 , \bigg|\E_\tau \bigg[  \int_\tau^T \e^{c r} |\delta \tilde g_r(s)|\d r \bigg]  \bigg|^2, \bigg|\E_\tau \bigg[  \int_\tau^T \e^{c r} |\delta g_r|\d r \bigg]  \bigg|^2\bigg\}\\
&\ \leq 4 L_{\star}^2 R^2\max\{ 2, T^2\}  \Big(\|\delta y\|_{\Sc^{\infty,c}}^2+ \| \delta z\|_{\H^{2,c}_{{\rm BMO}}}^2+ \|\delta  u\|_{\Sc^{\infty,2,c}}^2+ \| \delta v \|_{\overline \H^{2,c}_{{\rm BMO}}}^2\Big)
\end{align*}
Overall, we obtain back in \eqref{Eq:contractionItoq} that 
\begin{align*}
& \sum_{i=1}^4 \e^{ct}  |\delta \Yf_t^i|^2+\E_\tau\bigg[\int_t^T \e^{cr}  |\sigma^\t_r  \delta \Zf_r^i|^2 \d r+\int_t^T\e^{c r-} \d \Tr [\delta \Nf^i]_r  \bigg] \\
  \leq &\   \tilde \eps_3^{-1} \|\delta Y\|^2_{\Sc^{\infty,c}}+ \tilde \eps_4^{-1} \|\delta \Uc\|^2_{\Sc^{\infty,2,c}} + \tilde \eps_{5}^{-1} \|\delta U\|^2_{\Sc^{\infty,2,c}} + \tilde \eps_{6}^{-1} \|\delta \partial U\|^2_{\Sc^{\infty,2,c}} \\
&\; +     6 ( \tilde \eps_1+\tilde\eps_2+\tilde\eps_{6}) L_\star^2 R^2\max\{ 2, T^2\}  \Big( \|\delta \partial u\|_{\Sc^{\infty,2,c}}^2+ \| \delta \partial v\|_{\H^{2,2,c}_{{\rm BMO}}}^2 +\|\delta  u\|_{\Sc^{\infty,2,c}}^2+ \| \delta v\|_{\overline \H^{2,2,c}_{{\rm BMO}}}^2 +\|\delta y\|_{\Sc^{\infty,c}}^2+\| \delta z\|_{\H^{2,c}_{{\rm BMO}}}^2\Big) \\
& \; +4( \tilde \eps_3+\tilde\eps_4+\tilde\eps_{5})  L_{\star}^2 R^2\max\{ 2, T^2\}   \   \Big(\|\delta y\|_{\Sc^{\infty,c}}^2+ \| \delta z\|_{\H^{2,c}_{{\rm BMO}}}^2+ \|\delta  u\|_{\Sc^{\infty,2,c}}^2+ \| \delta v \|_{\overline \H^{2,c}_{{\rm BMO}}}^2\Big)
\end{align*}
If we define, for $\tilde \eps_i>10, i\in \{3,4,5,6\}$, $ C_{\tilde \eps}:= {\rm min}\big\{ 1-10/ \tilde \eps_{3 },\; 1-10/ \tilde \eps_{4 },\;1-10/ \tilde \eps_{5},\;1-10/ \tilde \eps_{6}\big\}$, we deduce,  
\begin{align}\label{Eq:thm:contq:final}\begin{split}
 \| \delta \mathfrak{H}\|_{\Hc^c}^2 \leq  20 C_{\tilde \eps}^{-1} L_{\star}^2 R^2\max\{ 2, T^2\}  (3 \tilde \eps_1 + 3 \tilde \eps_2 +2 \tilde \eps_3+2 \tilde \eps_4+2 \tilde \eps_{5}+3 \tilde \eps_{6 })\|\delta \mathfrak{h}\|_{\Hc^c}.
\end{split}
\end{align}

Minimising for $\tilde \eps_1$ and $\tilde\eps_2$ fixed, we find that letting
\[
 R^2 <  \frac{1}{2^6\cdot 3\cdot 5^2\cdot 7\cdot L^2_\star\cdot \max\{ 2, T^2\} } ,\; c\geq   \max  \{  \eps_1^{-1} 7T  L_{\rm u}^2,  \eps_2^{-1} 7T , \tilde \eps_1^{-1} 3T L_{\rm u}^2,\;  3T\tilde \eps_2^{-1},  2 L_{\rm u}  \}
\]
we have that
\[
 \|\delta \mathfrak{H}\|_{\Hc^c}^2  < \  \frac{20}{2^4\cdot  3\cdot 7\cdot 10^2}3(\sqrt{30+(\tilde\eps_1+\tilde\eps_2)}+\sqrt{30})^2  \|\delta \mathfrak{h}\|_{\Hc^c}   = \frac{ (\sqrt{30+(\tilde\eps_1+\tilde\eps_2)}+\sqrt{30})^2}{2^3\cdot 7\cdot 10} \|\delta \mathfrak{h}\|_{\Hc^c}.
\]

Thus, letting choosing $(\sqrt{30+(\tilde\eps_1+\tilde\eps_2)}+\sqrt{30})^2 \leq 2^3\cdot 7\cdot 10$, $\Tf$ is contractive.\medskip

{\bf Step 3:} We consolidate our results.. In light of
\eqref{Eq:cZwelldefinedq} and \eqref{Eq:c:contraction1q}, taking $\eps_i=\tilde\eps_i, i\in \{1,2\}$, $c$ must satisfy 
\begin{align}\label{eq:cfinalq}
 c\geq \max   \{  \eps_1^{-1} 7T  L_{\rm u}^2,  \eps_2^{-1} 7T , \tilde \eps_1^{-1} 3T L_{\rm u}^2,\;  3T\tilde \eps_2^{-1},\; 2 L_{\rm u}  \}= \max   \{  \eps_1^{-1} 7T  L_{\rm u}^2,  \eps_2^{-1} 7T ,\; 2 L_{\rm u}  \}
\end{align}

All together we find that given $\gamma\in(0,\infty)$, $\eps_i\in(0,\infty)$, $i\in\{1,2\}$, $c\in (0,\infty)$,  such that $ \eps_1+\eps_2 \leq (4\sqrt{35}-\sqrt{30})^2-30$, $\Tf$ is a well--defined contraction in $\Bc_{ R}\subseteq \Hc^c$ for the norm $\| \cdot \|_{\Hc^c}$ provided: $(i)$ $\gamma $, $\eps_i$, $i\in \{1,2\}$, and the data of the problem satisfy \eqref{Eq:thm:wpq:smalldatacond}; $(ii)$ $c$ satisfies \eqref{eq:cfinalq}.
\end{proof}

\bibliography{Bibliography}

\begin{appendix}

\section{Proofs of Section \ref{sec:prelim}}\label{sec:proofsprelim}

\begin{proof}[Proof of {\rm \Cref{lemma:energy}}]
First note that for $Z\in \H^2_{{\rm BMO}}(\R^{n\times \tilde d})$, $Z\bullet X$ is a continuous local martingale, thus we have that 
\[ \| Z\bullet  X\|_{{\rm BMO}^{2,c}}=\sup_{\tau\in\Tc_{0,T}} \Big \| \E\big[ \big \langle\e^{\frac{c}2} Z\bullet X\big\rangle_T- \big \langle \e^{\frac{c}2} Z\bullet X\big \rangle_\tau \big|\Fc_\tau\big]\Big\|_\infty<\infty.\] 
Therefore, letting $X_t:=\E[ \langle \e^{\frac{c}2 }Z\bullet X\rangle_T- \langle \e^{\frac{c}2 } Z\bullet X\rangle_t |\Fc_t]$, we have: $(i)$ $|X_t|\leq  \| Z\bullet X\|_{{\rm BMO}^{2,c}}=\| Z\|_{\H^{2,c}_{{\rm BMO}}}^2$;
$(ii)$ $A=\langle \e^{\frac{c}2 } Z\bullet X\rangle$. Indeed, note  $X_t=\E\big[  \big\langle \e^{\frac{c}2 }Z\bullet X\big\rangle_T  \big|\Fc_t\big]- \big\langle \e^{\frac{c}2 } Z\bullet X\big\rangle_t $. The result then follows immediately from the energy inequality, i.e.
\[ \E\bigg[ \bigg(\int_0^T\e^{cr }|\sigma^\t_r Z_r|^2 \d r\bigg)^p\bigg]=\E[ (A)_\infty^p] \leq  p !\| Z\|_{\H^{2,c}_{{\rm BMO}}}^{2p}.\]

To obtain the second part of the statement, recall that by definition of $\Ho(\R^{n\times \tilde d})$, $s\longmapsto \partial Z^s$ is the density of $s\longmapsto Z^s$ with respect to the Lebesgue measure and $\Zc$ is given as in \Cref{remark:defspaces}. By definition of $\Zc$, Fubini's theorem and Young's inequality we have that for $\eps>0$ 
\begin{align*}
\int_t^T \e^{cu} | \sigma^\t_uZ_u^u|^2-\e^{cu}  |\sigma^\t_u Z_u^t|^2 \d u& = \int_t^T \int_r^T 2 \e^{cu}  \Tr\Big [ {Z_u^r}^\t{\sigma_u}   {\sigma^\t_u} \partial Z_u^r   \Big ] \d u \d r\\
& \leq \int_t^T\int_r^T \eps \e^{cu}  |\sigma_u^\t  Z^r_u|^2+ \eps^{-1}\e^{cu}   |\sigma_u^\t \partial Z^r_u |^2 \d u \d r .
\end{align*} 
This proves the first first statement. For the second claim, we may use \eqref{Eq:ineqsquare} and \eqref{Eq:ineqBMO} to obtain
\begin{align*}
\E_t\bigg[ \bigg( \int_t^T \e^{cu} |\sigma^\t_u \Zc_u|^2 \d u\bigg)^2\bigg] & \leq 3\bigg(\E_t\bigg[ \bigg( \int_t^T\e^{cu}  | \sigma^\t_uZ_u^t|^2 \d u\bigg)^2\bigg]\\
&\quad + T \int_t^T\E_t\bigg[ \bigg(\int_t^T \e^{cu}   |\sigma_u^\t  Z^r_u|^2\d u \bigg)^2\bigg] \d r  + T \int_t^T \E_t\bigg[ \bigg( \int_t^T   \e^{cu}  |\sigma_u^\t \partial Z^r_u |^2 \d u \bigg)^2\bigg]\d r\bigg)\\
& \leq 6 ( (1+T^2)\|Z\|_{\H^{2,2,c}_{\rm BMO}}^4+T^2 \|\partial Z\|_{\H^{2,2,c}_{\rm BMO}}^4)
\end{align*}
The inequality for the $\H^2$ norm is argued similarly taking expectations.
\end{proof}

\section{Proofs of Section \ref{sec:prooflinearquadratic}}

We next lemma helps derive appropriate auxiliary estimates of the terms $U_t^t$ and $\partial U_t^t$ as in \Cref{sec:prooflinearquadratic}.

\begin{lemma}\label{Lemma:estimatescontraction}
Let $\partial U$ satisfy the equation
\[
\partial U_t^s= \partial_s \eta (s,X_{\cdot\wedge,T})+\int_t^T   \nabla g_r(s,X,\partial U_r^s,\partial v_r^s,U_r^s,v_r^s, \Yc_r, z_r) \d r-\int_t^T  \partial {V_r^s}^\t  \d X_r-\int_t^T \d  \partial M^s_r,
\]
 and $c\geq \max\{ 2L_u, 2L_{\rm u}\}$, the following estimates hold for $t\in [0,T]$
\begin{align*}
\E_t\bigg[ \int_t^T \frac{ \e^{cr}}{7   T}  |\partial U_r^r|^2\d r\bigg]&  \leq    \|\partial_s \eta \|_{\Lc^{\infty,2,c}}^2+ \|\nabla \tilde g\|^2_{\L^{1,\infty,2,c}} +  T L_y^2  \E_t\bigg[   \int_t^T \e^{cr}|Y_r|^2\d r\bigg]+  T L_u^2 \sup_{s\in [0,T]} \E_t\bigg[   \int_t^T \e^{cr}|U_r^s|^2\d r\bigg]\\
& \quad +   2 L_{\star}^2    \Big( \|\partial v\|^4_{\H^{2,2,c}_{{\rm BMO}}} + \|z\|^4_{\H^{2,c}_{{\rm BMO}}}+\|v\|^4_{\H^{2,2,c}_{{\rm BMO}}}\Big)\\
\E_t\bigg[ \int_t^T \frac{ \e^{\frac{c}2 r}}{ T}  |\partial U_r^r|\d r\bigg] & \leq    \|\partial_s \eta \|_{\Lc^{\infty,2,c}}+ \|\nabla \tilde g\|_{\L^{1,\infty,2,c}} +   L_y  \E_t\bigg[   \int_t^T \e^{\frac{c}2 r}|Y_r|\d r\bigg]+   L_u \sup_{s\in [0,T]} \E_t\bigg[   \int_t^T \e^{cr}|U_r^s|\d r\bigg]\\
&\quad  +    L_{\star}    \Big(\|\partial v\|^2_{\H^{2,2,c}_{{\rm BMO}}}+  \|v\|^2_{\H^{2,2,c}_{{\rm BMO}}}+ \|z\|^2_{\H^{2,c}_{{\rm BMO}}}\Big)
\end{align*}
\end{lemma}
\begin{proof}
By Meyer--It\^o's formula for $\e^{\frac{c}2  t} |\partial U_t^s|$, see \citet*[Theorem 70]{protter2005stochastic}
\begin{align}\label{eq:eq1}
\begin{split}
&\e^{\frac{c}2  t}|\partial U_t^s|+ L_T^0 -\int_t^T \e^{\frac{c}2  r} \sgn( \partial U_r^s)\cdot  \partial {V_r^s}^\t  \d X_r -\int_t^T \e^{\frac{c}2  r-} \sgn( \partial U_{r-}^s)\cdot  \d \partial M_r^s \\
&\; =\e^{\frac{c}2  T}  |\partial_s \eta (s)|  + \int_t^T  \e^{\frac{c}2  r} \bigg(  \sgn(  \partial U_r^s)  \cdot \nabla g_r(s,\partial U_r^s,\partial v_r^s,U_r^s,v_r^s,\Yc_r,z_r)-\frac{c}2  |\partial U_r^s| \bigg) \d r ,\; t\in[0,T],
\end{split}
\end{align} 
where $L^0:=L^0(\partial U^s)$ denotes the non-decreasing and pathwise-continuous local time of the semi-martingale $\partial U^s$ at $0$, see \cite[Chapter IV, pp. 216]{protter2005stochastic}. We also notice that for any $s\in [0,T]$ the last two terms on the left-hand side are martingales, recall that $\partial V^s\in \H^2$ by \cite[Theorem 3.5]{hernandez2020unified}.\medskip

In light of \Cref{Assumption:LQgrowth}, letting $ \nabla g_r(s):=\nabla g_r(s,\partial U_r^s,\partial v_r^s,U_r^s,v_r^s,Y_r, z_r)$, we have that $\d t\otimes \d \P\ae$
\begin{align}\label{Eq:ineqLipUts0} 
\begin{split}
 |\nabla g_r(s) |\leq & L_{\rm u} |\partial U_r^s|   +L_{\rm v} |\sigma^\t_r  \partial v_r^s|^2+L_u |U_r^s| +L_v |\sigma^\t_r  v_r^s|^2+L_y |Y_r|+L_z |\sigma^\t_r  z_r|^2+  |\nabla \tilde g_r(s)|,
\end{split}
\end{align}

We now take conditional expectation with respect to $\Fc_t$ in \Cref{eq:eq1}. We may use \eqref{Eq:ineqLipUts0} and the fact $\tilde L^0$ is non--decreasing to derive that for $c>2 L_{\rm u}$ and $ t\in[0,T]$
\begin{align}\label{Eq:ineqUst}
\e^{\frac{c}2 t}| \partial U_t^s| & \leq  \E_t \bigg[ \e^{\frac{c}2 T} |\partial \eta(s)|+\int_t^T \e^{\frac{c}2 r} \big(  |\nabla \tilde g_r(s)| +L_{\rm v} |\sigma^\t_r  \partial v_r^s|^2+L_u |U_r^s|  +L_v |\sigma^\t_r   v_r^s|^2+L_y |Y_r|+L_z |\sigma^\t_r   z_r|^2\big)  \d r \bigg].
\end{align}

Squaring in \eqref{Eq:ineqUst}, we may use \eqref{Eq:ineqsquare} and Jensen's inequality to derive that for $t\in [0,T]$
\begin{align*}
\frac{\e^{ct}}{7} |\partial U_t^t|^2  \leq    &\  \E_t\bigg[ \e^{cT} |\partial_s \eta(t)|^2+ \bigg(\int_t^T \e^{\frac{c}2 r} |\nabla \tilde g_r(t)|\d r\bigg)^2+ T L_{u}^2 \int_t^T \e^{c r}    |U_r^t|^2 \d r + T L_y^2 \int_t^T \e^{c r}    |Y_r|^2 \d r\\
& + L_{\rm v}^2 \bigg(\int_t^T \e^{\frac{c}2 r} |\sigma^\t_r \partial v_r^t|^2\d r \bigg)^2+ L_v^2 \bigg(\int_t^T \e^{\frac{c}2 r} |\sigma^\t_r v_r^t|^2\d r\bigg)^2+ L_z^2 \bigg(\int_t^T \e^{\frac{c}2 r} |\sigma^\t_rz_r|^2\d r\bigg)^2 \bigg].
\end{align*}

By integrating the previous expression and taking conditional expectation with respect to $\Fc_t$, it follows from the tower property that for any $t\in[0,T]$
\begin{align*}
\frac{  1}7\E_t\bigg[\int_t^T \e^{cr}|\partial U_r^r|^2\d r\bigg]\leq &\ \E_t\bigg[ \int_t^T   \e^{cT} |\partial_s \eta(r)|^2\d r\bigg]+\E_t\bigg[ \int_t^T\bigg[ \bigg(  \int_r^T \e^{\frac{c}2 u}|\nabla \tilde g_u(r)|\d u\bigg)^2\d r\bigg] \\
&   + T L_{u}^2 \E_t\bigg[  \int_t^T   \int_r^T \e^{c u}    |U_u^r|^2 \d u\bigg] \d r +  T  L_y^2  \E_t\bigg[  \int_t^T   \int_r^T \e^{cu} |Y_u|^2\d u \bigg]\d r\bigg]  \\
& +  L_{\rm v}^2     \int_t^T\E_t\bigg[    \bigg(\int_r^T \e^{\frac{c}2 u} |\sigma^\t_u \partial v_u^r|^2\d u\bigg)^2\bigg] \d r +  L_v^2     \int_t^T\E_t\bigg[    \bigg(\int_r^T \e^{\frac{c}2 u} |\sigma^\t_u v_u^r|^2\d u\bigg)^2\bigg] \d r\\
& + L_z^2    \int_t^T\E_t\bigg[  \bigg(\int_r^T \e^{\frac{c}2 u} | \sigma^\t_u z_u|^2\d u\bigg)^2\bigg] \d r\\
\leq &\ T \sup_{r\in [0,T]} \bigg\{ \|  \e^{cT} |\eta(r)|^2]\|_\infty+  \bigg \|  \int_r^T \e^{\frac{c}2 u} |\nabla \tilde g_u(r)|\d u \bigg \|_\infty^2 \bigg\} +T^2L_y^2 \E_t\bigg[   \int_t^T \e^{cu}|Y_u|^2\d u\bigg]\\
&  + T^2 L_{u}^2 \sup_{r\in [0,T]}  \bigg\{ \E_t\bigg[  \int_t^T  \e^{c u}    |U_u^r|^2 \d u\bigg]  \bigg\} +    T L_{\rm v}^2    \sup_{r\in [0,T]} \bigg\{ \E_t\bigg[    \bigg(\int_t^T \e^{\frac{c}2 u} |\sigma^\t_u \partial v_u^r|^2\d u\bigg)^2\bigg] \d r\bigg\}    \\
&+    T L_v^2    \sup_{r\in [0,T]} \bigg\{ \E_t\bigg[    \bigg(\int_t^T \e^{\frac{c}2 u} |\sigma^\t_u v_u^r|^2\d u\bigg)^2\bigg] \d r\bigg\}  + TL_z^2    \E_t\bigg[  \bigg(\int_t^T \e^{\frac{c}2 u} | \sigma^\t_u z_u|^2\d u\bigg)^2\bigg] \d r,
\end{align*}

and by \eqref{Eq:ineqBMO} we obtain for $c>2L_u$, and any $t\in[0,T]$
\begin{align*}
\E_t\bigg[ \int_t^T \frac{ \e^{cr}}{7   T}  |\partial U_r^r|^2\d r\bigg]& \leq    \|\partial_s \eta \|_{\Lc^{\infty,2,c}}^2+ \|\nabla \tilde g\|^2_{\L^{1,\infty,2,c}} +  T L_y^2  \E_t\bigg[   \int_t^T \e^{cr}|Y_r|^2\d r\bigg]+ + T L_{u}^2 \sup_{r\in [0,T]}   \E_t\bigg[  \int_t^T  \e^{c u}    |U_u^r|^2 \d u\bigg]    \\
&\; +   2 L_\star^2    \Big( \|\partial v\|^4_{\H^{2,2,c}_{{\rm BMO}}} + \|z\|^4_{\H^{2,c}_{{\rm BMO}}}+\|v\|^4_{\H^{2,2,c}_{{\rm BMO}}}\Big) . 
\end{align*}

Evaluating at $s=t$ in \eqref{Eq:ineqUst} and integrating with respect to $t$ we derive the second estimate.
\end{proof}

\begin{lemma}[Optimal upper bound for $R$]\label{Lemma:upperbound}
$({\rm OPT1})=1/(2^{4}5 )$, where 
\begin{align*}
 \sup\; &\frac{ \min\big \{  \alpha(\eps_3,\eps_{12},\eps_{13}),\;  \alpha(\eps_4,\eps_{14},\eps_{15}) ,\; \alpha(\eps_{5},\eps_{16},\eps_{17}) ,\; \alpha(\eps_{6},\eps_{18},\eps_{19},\eps_{20}) \big \} -\gamma}{ \eps_1+\eps_2+\sum_{i=12}^{20} \eps_i  }\\
 {\rm s.t.}\;  &  \alpha(\eps_8,\eps_{12},\eps_{13})=1-10(\eps_{8}^{-1}+\eps_{12}^{-1}+\eps_{13}^{-1})   \in (0,1], \;
   \alpha(\eps_9,\eps_{14},\eps_{15})=1-10(\eps_{9}^{-1}+\eps_{14}^{-1}+\eps_{15}^{-1})   \in (0,1], \\
& \alpha(\eps_{10},\eps_{16},\eps_{17}) =1-10(\eps_{10}^{-1}+\eps_{16}^{-1}+\eps_{17}^{-1})   \in (0,1], \;  \alpha(\eps_{11},\eps_{18},\eps_{19},\eps_{20})=1-10(\eps_{11}^{-1}+\eps_{18}^{-1}+\eps_{19}^{-1}+\eps_{20}^{-1})  \in (0,1], \\
& \gamma\in (0,\infty);\;  \eps_i\in (0,\infty), \forall i .
\end{align*}
\end{lemma}

\begin{proof} 
We begin by noticing that as a function of $(\gamma,\eps_1,\eps_2,\eps_3,\eps_4,\eps_{5},\eps_{6})$ the objective is bounded by the value when $(\gamma,\eps_1,\eps_2,\eps_3,\eps_4,\eps_{5},\eps_{6})\longrightarrow (0,0,0,\infty,\infty,\infty,\infty)$. Thus, we will maximise

\[\frac{\min\{ 1-10( \eps_{12}^{-1}+\eps_{13}^{-1}) ,\;1-10( \eps_{14}^{-1}+\eps_{15}^{-1})  ,\;1-10( \eps_{16}^{-1}+\eps_{17}^{-1})   ,\;1-10(\eps_{18}^{-1}+\eps_{19}^{-1}+\eps_{20}^{-1})  \} }{\sum_{i=12}^{20} \eps_i  }   . \]

From this we observe that the optimal value is positive. Indeed, there is a feasible solution with positive value, and the $\min$ in the objective function does not involve common $\eps_i$ terms, so the minima is attained at one of the terms. Since the value function is symmetric in each of the variables inside each term of the mean we can assume with out lost of generality
\[
\eps_{12}=\eps_{13}=2\alpha_1,\;\eps_{14}=\eps_{15}=2\alpha_2,\; \eps_{16}=\eps_{17}=2\alpha_3, \; \eps_{18}=\eps_{19}=\eps_{20}=3\alpha_4,\; \{\alpha_1,\alpha_2,\alpha_3,\alpha_4\} \in (0,\infty)^4
\]
So we can write the objective function as $\min\{ 1-10\alpha_1^{-1} ,\;1-10\alpha_2^{-1}  ,\;1-10\alpha_3^{-1}   ,\;1-10\alpha_4^{-1}  \} /( 4\alpha_1+4\alpha_2+4\alpha_3+9\alpha_4)$. Now, without lost of generality the $\min$ is attained by the first quantity. This is, 
the optimisation problem becomes
\begin{align*}
\sup\;\frac{ 1-10\alpha_1^{-1}  }{ 4\alpha_1+4\alpha_2+4\alpha_3+9\alpha_4   }\;  {\rm s.t.}\;   \alpha_1\leq \min\{\alpha_2,\alpha_3,\alpha_4\}, 1-10\alpha_i^{-1}\in (0,1], \alpha_i\in (0,\infty), i\in \{1,2,3,4\}.
\end{align*} 
Now, as the objective function is decreasing in $\alpha_2,\alpha_3,\alpha_4$, and $\alpha_1\leq \min\{\alpha_2,\alpha_3,\alpha_4\}$, we see $\alpha_1=\alpha_2=\alpha_3=\alpha_4$. Thus
\begin{align*}
\sup\;\frac{ 1-10\alpha_1^{-1}  }{ 21 \alpha_1   }\;  {\rm s.t.}\;   1-10\alpha_1^{-1}\in (0,1], \alpha_1\in (0,\infty).
\end{align*} 
Let $f(\alpha_1):= \frac{ \alpha_1 -10}{21 \alpha_1^2}$. By first order analysis
\[ \partial_{\alpha_1}f(\alpha_1) = \frac{-\alpha_1(\alpha_1-20)}{21 \alpha_1^4 }=0, \text{ yields, }\alpha_1\in \{0,20 \} 
\]
By inspecting the sign of the derivative, one sees that $\alpha_1=0$ corresponds to a minima and $\alpha_1=20$ is the maximum and it is feasible. Thus we obtain that
\begin{align*}
f\big(\alpha_1^\star\big)=\frac{1}{2^3\cdot 3\cdot 5\cdot 7},
\end{align*}

We conclude the maxima when $(\eps_{12},\eps_{13},\eps_{14},\eps_{15},\eps_{16},\eps_{17},\eps_{18},\eps_{19},\eps_{20})=(40,40,40,40,40,40,60,60,60)$. Evaluating the value function in these values and letting $(\gamma,\eps_1,\eps_3,\eps_8,\eps_9,\eps_{10},\eps_{11})\longrightarrow (0,0,0,\infty,\infty,\infty,\infty)$, we obtain this bound. This is, $f$ does not attain its maximum value, but in the feasible region it can get as close as possible.
\end{proof}

\begin{lemma}[Minimal bound for Contraction]\label{Lemma:Contractionbound}
$({\rm OPT2})=3(\sqrt{30+(\tilde\eps_1+\tilde\eps_2)}+\sqrt{30})^2$, where 
\begin{align*} 
({\rm OPT2}):=&\ \inf \bigg\{ \big(3 \tilde \eps_1 + 3 \tilde \eps_2 +2\tilde \eps_3+2\tilde \eps_4 +2\tilde \eps_{5}+3\tilde \eps_{6}\big)  \min \bigg\{   \frac{\tilde \eps_3}{ \tilde \eps_3 -10} ,\;  \frac{\tilde \eps_4}{ \tilde \eps_4 -10},\;  \frac{\tilde \eps_{5}}{ \tilde \eps_{5} -10},\;  \frac{\tilde \eps_{6}}{ \tilde \eps_{6} -10} \bigg\}\bigg\} \\
 {\rm s.t.}\;  &  1-10\tilde\eps_i^{-1} \in (0,1] ,\;  \tilde \eps_i \in (0,\infty),\; i\in\{3,4,5,6\}.
\end{align*}
\end{lemma}
\begin{proof}

Without lost of generality let us assume the min is attained by the first quantity, i.e.  the optimisation problem becomes
\[ \inf   \big(3(\tilde\eps_1+\tilde\eps_2)+ 2\tilde \eps_3+2\tilde \eps_4 +2\tilde \eps_{5}+3\tilde \eps_{6}\big)    \frac{\tilde \eps_3}{ \tilde \eps_3-10} , \; \text{ s.t. } \tilde \eps_8\leq \min\{\tilde\eps_4,\tilde\eps_{5},\tilde\eps_{6}\},\;  1-10\tilde\eps_3^{-1} \in (0,1],\;  \tilde \eps_i \in (0,\infty)\ \forall i.\]

As the value function is increasing in $(\tilde\eps_4,\tilde\eps_{5},\tilde\eps_{6})$, $\tilde \eps_3\leq \min\{\tilde\eps_4,\tilde\eps_{5},\tilde\eps_{6}\}$ implies we must have $\tilde \eps_3=\tilde\eps_4=\tilde\eps_{5}=\tilde\eps_{6}$ a thus we minimise
\[ f(\tilde \eps):=3\frac{3 \tilde \eps^2 +\tilde\eps(\tilde \eps_1+\tilde \eps_2)}{\tilde \eps-10}.\]

First order analysis renders
\[ \partial_{\tilde\eps}f(\tilde\eps) = \frac{9\tilde\eps-180\tilde\eps-30(\tilde\eps_1+\tilde\eps_2)}{\tilde\eps-10}=0, \text{ yields, }\tilde\eps^\pm= 10\pm \frac{1}6\sqrt{60^2+120(\tilde\eps_1+\tilde\eps_2)}. 
\]

The minimum occurs at $\tilde\eps^\star=10+ \frac{1}6\sqrt{60^2+120(\tilde\eps_1+\tilde\eps_2)}$, and $f(\tilde\eps^\star)=3(\sqrt{30+(\tilde\eps_1+\tilde\eps_2)}+\sqrt{30})^2$. We conclude the minima occurs when $(\tilde\eps_3,\tilde\eps_4,\tilde\eps_{5},\tilde\eps_{6})=(20,20,20,20)$. Evaluating the value function in these values and letting $(\tilde \eps_1,\tilde \eps_2)\longrightarrow (0,0)$, we obtain this bound. This is, $f$ does not attain its minimum value, but in the feasible region it can get as close as possible.
\end{proof}

\end{appendix}

\end{document}